\documentclass[12pt]{amsart}

\usepackage{fullpage}
\usepackage{amssymb,amsfonts}
\usepackage[all,arc]{xy}
\usepackage{enumerate}
\usepackage{mathrsfs}
\usepackage{mathtools}
\usepackage{tikz}
\usepackage{tikz-cd}
\usepackage{graphicx}
\usepackage{hyperref}
\usepackage{rotating}

\DeclareMathOperator{\Hom}{Hom}

\DeclareMathOperator{\F}{\mathbb{F}}
\DeclareMathOperator{\Z}{\mathbb{Z}}
\DeclareMathOperator{\Q}{\mathbb{Q}}
\DeclareMathOperator{\R}{\mathbb{R}}

\DeclareMathOperator{\Hbb}{\mathbb{H}}
\DeclareMathOperator{\Sym}{Sym}

\DeclareMathOperator{\Image}{Im}
\DeclareMathOperator{\Spinc}{Spin^c}
\DeclareMathOperator{\spinc}{spin^c}
\DeclareMathOperator{\Char}{Char}
\DeclareMathOperator{\Hplus}{\mathbb{H}^+}
\DeclareMathOperator{\Kplus}{\mathrm{K}^+}

\DeclareMathOperator{\rank}{rank}
\DeclareMathOperator{\gr}{gr}
\DeclareMathOperator{\Map}{Map}
\DeclareMathOperator{\coker}{coker}

\newcommand*{\isoarrow}[2]{\arrow[#1,"\rotatebox{90}{\(\sim\)}", "#2"']}
\newcommand*{\doubleisoarrow}[1]{\arrow[#1,leftrightarrow, swap,"\rotatebox{90}{\(\sim\)}"]}

\newtheorem{thm}{Theorem}[section]
\newtheorem*{thm*}{Theorem}
\newtheorem{cor}[thm]{Corollary}
\newtheorem*{cor*}{Corollary}
\newtheorem{prop}[thm]{Proposition}
\newtheorem{lem}[thm]{Lemma}

\theoremstyle{definition}
\newtheorem{defn}[thm]{Definition}
\newtheorem{defns}[thm]{Definitions}
\newtheorem{con}[thm]{Construction}
\newtheorem{exmp}[thm]{Example}
\newtheorem{exmps}[thm]{Examples}
\newtheorem{notn}[thm]{Notation}

\theoremstyle{remark}
\newtheorem{rem}[thm]{Remark}
\newtheorem{rems}[thm]{Remarks}

\makeatletter
\let\c@equation\c@thm
\makeatother
\numberwithin{equation}{section}

\title{On the involutive Heegaard Floer homology of negative semi-definite plumbed 3-manifolds with $b_{1}=1$}

\author[P. Johnson]{Peter K. Johnson}
\address{Department of Mathematics, University of Virginia, Charlottesville, VA 22904-4137}
\email{\href{mailto:pkj4vj@virginia.edu}{pkj4vj@virginia.edu}}

\begin{document}

\begin{abstract}
In \cite{MR1957829}, Ozsváth and Szabó use Heegaard Floer homology to define numerical invariants $d_{1/2}$ and $d_{-1/2}$ for 3-manifolds $Y$ with $H_{1}(Y;\Z)\cong \Z$. We define involutive Heegaard Floer theoretic versions of these invariants analogous to the involutive $d$ invariants $\bar{d}$ and $\underline{d}$ defined for rational homology spheres by Hendricks and Manolescu in \cite{MR3649355} . We prove their invariance under spin integer homology cobordism and use them to establish spin filling constraints and $0$-surgery obstructions analogous to results by Ozsváth and Szabó for their Heegaard Floer counterparts $d_{1/2}$ and $d_{-1/2}$. We then apply calculation techniques of Dai and Manolescu developed in \cite{MR4021102} and Rustamov in \cite{Rustamov} to compute the involutive Heegaard Floer homology of some negative semi-definite plumbed 3-manifolds with $b_{1} =1$. By combining these calculations with the $0$-surgery obstructions, we are able to produce an infinite family of small Seifert fibered spaces with weight 1 fundamental group and first homology $\Z$ which cannot be obtained by $0$-surgery on a knot in $S^3$, extending a result of Hedden, Kim, Mark, and Park in \cite{MR4029676}. 
\end{abstract}

\maketitle

\section{Introduction}

Involutive Heegaard Floer homology is an extension of Heegaard Floer homology due to Hendricks and Manolescu (see \cite{MR3649355}). It is constructed by considering the mapping cone of a naturally arising involution on the Heegaard Floer chain complex associated to a given Heegaard diagram. For certain 3-manifolds, involutive Heegaard Floer homology contains more information than Heegaard Floer homology. In particular, it has had success illuminating the structure of the integer homology cobordism group. 

Over the past several years there has been significant progress in understanding how to calculate involutive Heegaard Floer homology. Some of the methods developed include the large surgery formula of Hendricks and Manolescu in \cite{MR3649355}, the results on almost rational negative definite plumbings by Dai and Manolescu in \cite{MR4021102}, the connected sum formula of Hendricks, Manolescu, and Zemke in \cite{MR3782421}, and most recently the involutive surgery exact triangle established by Hendricks, Hom, Stoffregen, and Zemke (see \cite{Surgery_triangle}).

To date, much of the focus of these calculation techniques and applications has been on rational homology 3-spheres. The goals of this paper are to (1) establish topological applications of involutive Heegaard Floer homology for 3-manifolds with $b_{1}=1$, and (2) to find an efficient way to compute the involutive Heegaard Floer homology of a certain class of such manifolds.  

For rational homology spheres, important topological information is encoded by the involutive $d$ invariants $\bar{d}$ and  $\underline{d}$ defined by Hendricks and Manolescu in \cite{MR3649355}. These are numerical invariants extracted from the plus (or equivalently minus) version of involutive Heegaard Floer homology with respect to a self-conjugate $\spinc$ structure. 

\newpage In this paper, we define analogous involutive $d$ invariants $\bar{d}_{-1/2}$, $\bar{d}_{1/2}$, $\underline{d}_{-1/2}$, and $\underline{d}_{1/2}$ for 3-manifolds $Y$ with $H_{1}(Y;\Z)\cong \Z$. These invariants are generalizations of the invariants $d_{-1/2}$ and $d_{1/2}$ defined by Ozsváth and Szabó in \cite{MR1957829} and also encode important topological information. In particular, they are spin integer homology cobordism invariants. Moreover, in section \ref{section:involutive}, we prove the following theorems which generalize \cite[Theorem 9.11]{MR1957829} and \cite[Proposition 4.11]{MR1957829} of Ozsváth and Szabó.

\begin{thm*}[A]\label{thm: A}
Suppose $X$ is a smooth oriented negative semi-definite spin 4-manifold with boundary a 3-manifold $Y$ with $H_{1}(Y;\Z)\cong \Z$.   
\begin{enumerate}
    \item\label{hyp: 1} If the restriction $H^{1}(X;\Z)\to H^{1}(Y;\Z)$ is trivial, then 
    \begin{align*}
        b_{2}(X)-3\leq 4\underline{d}_{-1/2}(Y)
    \end{align*}
    
    \item If the restriction $H^{1}(X;\Z)\to H^{1}(Y;\Z)$ is non-trivial, then
    \begin{align*}
        b_{2}(X)+2\leq 4\underline{d}_{1/2}(Y)
    \end{align*}
\end{enumerate}
\end{thm*}

\begin{rem}
Hypothesis (\ref{hyp: 1}) implies $b_{2}(X)\geq 1$.
\end{rem}

\begin{thm*}[B]\label{thm: B}
Let $M$ be an oriented integer homology 3-sphere and let $Y$ and $M'$ be the 3-manifolds obtained via $0$ and $+1$ surgery respectively on a knot $K$ in $M$. Then, 
\begin{enumerate}
\item 
$\underline{d}(M)-\frac{1}{2}\leq \underline{d}_{-1/2}(Y)\hspace{2em}\text{and}\hspace{2em}\bar{d}(M)-\frac{1}{2}\leq \bar{d}_{-1/2}(Y)$

\item 
$\underline{d}_{1/2}(Y)-\frac{1}{2}\leq \underline{d}(M')\hspace{2.35em}\text{and}\hspace{2em}\bar{d}_{1/2}(Y)-\frac{1}{2}\leq \bar{d}(M')$
\end{enumerate}
\end{thm*}
\noindent As a consequence of these theorems, we obtain the following two corollaries:
\begin{cor*}[C]\label{cor: C}
Suppose $K$ is a knot in $S^3$ and $Y$ is the result of $0$-surgery on $K$. Then, 
\begin{enumerate}
\item $-\frac{1}{2}\leq \underline{d}_{-1/2}(Y)$
\item $\bar{d}_{1/2}(Y)\leq \frac{1}{2}$
\end{enumerate}
\end{cor*}

\begin{cor*}[D]\label{cor: D}
Suppose $Y$ is a closed oriented 3-manifold with $H_{1}(Y;\Z)\cong \Z$. If 
\begin{align*}
    \underline{d}_{-1/2}(Y)< -1/2\hspace{2em}\text{and}\hspace{2em}\underline{d}_{1/2}(Y)<1/2
\end{align*}
then $Y$ is not the boundary of any negative semi-definite spin manifold. 
\end{cor*}

To put the above results to use, we need a practical way to calculate $\bar{d}_{\pm 1/2}$ and $\underline{d}_{\pm 1/2}$. The approach we take to achieve this is to adapt existing methods for computing $\bar{d}$ and $\underline{d}$ for rational homology spheres to the setting of 3-manifolds with $b_{1} =1$. In \cite{MR4021102}, Dai and Manolescu provide a combinatorial method to compute the involutive Heegaard Floer homology of a certain class of negative definite plumbed 3-manifolds called almost rational (or AR) plumbed manifolds. In particular, their methods provide a way to compute the invariants $\bar{d}$ and $\underline{d}$ for rational homology spheres which admit such a plumbing. 

Their approach utilizes the framework of lattice cohomology and graded roots introduced by Némethi in \cite{MR2426357} and \cite{MR2140997}. Lattice cohomology itself builds upon earlier work by Ozsváth and Szabó in which they show how to combinatorially compute the Heegaard Floer homology of a subclass of almost rational plumbings, namely negative definite plumbings with at most one \hyperlink{bad vertex}{\textit{bad}} vertex (see \cite{MR1988284}). Rustamov later generalized this work of Ozsváth and Szabó to the case of negative semi-definite plumbings with $b_{1}=1$ and at most one bad vertex (see \cite{Rustamov}). To cohesively adapt and combine the Dai and Manolescu work with the work of Rustamov, we first recast Rustamov's results in the language of lattice cohomology and graded roots. This requires us to slightly modify Neméthi's original definition of lattice cohomology.   

After establishing the above computational approach, we carry out a specific calculation of the plus version of the involutive Heegaard Floer homology of an infinite family $\{N_{j}\}_{j\in \mathbb{N}}$ of small Seifert fiber spaces. For $j\in \mathbb{N}$, we let $N_{j}=S^2\left(-\frac{2}{1}, \frac{-8j+1}{1}, \frac{16j-2}{8j+1}\right)$. $N_{j}$ can also be realized as surgery on a 2-component link as follows:

\begin{figure}[h]
  \centering
  \includegraphics[scale=.7]{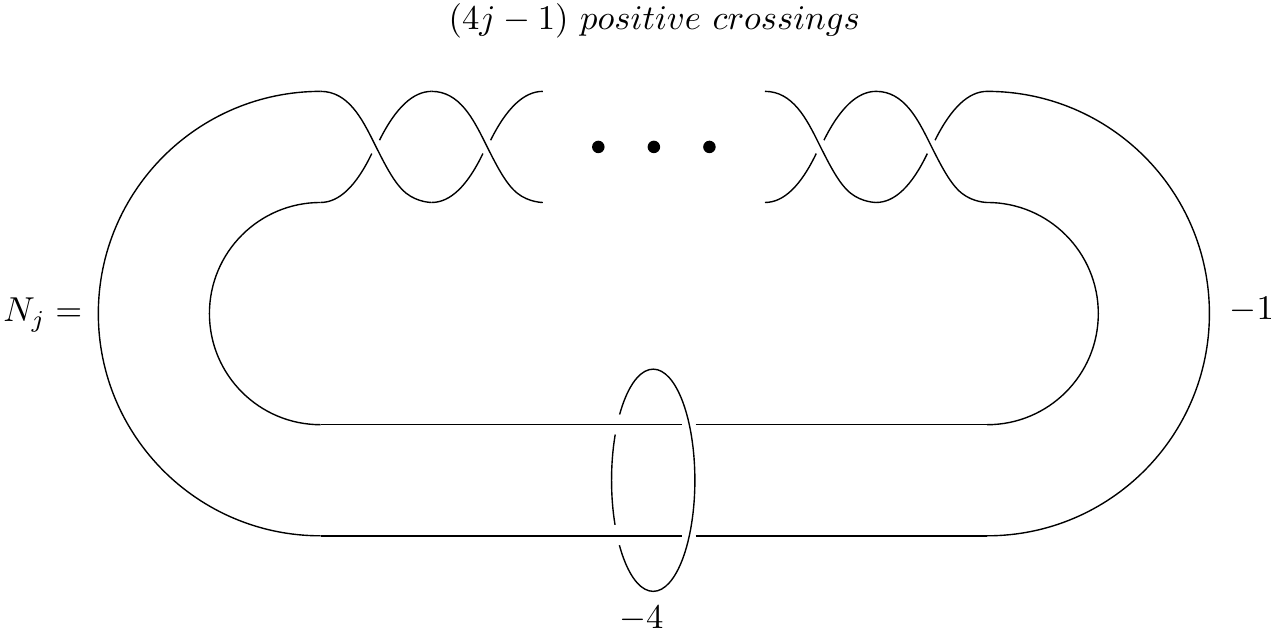}
  \end{figure}

The family $\{N_{j}\}_{j\in \mathbb{N}}$ was previously studied by Hedden, Kim, Mark, and Park in \cite{MR4029676}. The manifolds in this family all have first homology equal to $\Z$ and weight 1 fundamental groups, which are necessary conditions if said manifolds could be obtained by $0$-surgery on a knot in $S^3$. However, by using an obstruction in terms the Rokhlin invariant, Hedden, Kim, Mark, and Park proved that for all odd positive integers $j$, $N_{j}$ cannot be obtained by $0$-surgery on a knot in $S^3$ (see \cite[Theorem 7.3]{MR4029676}). In the same paper, they also show that if $Y$ is a 3-manifold that is homology cobordant to a Seifert fibered homology $S^1\times S^2$, then $Y$ automatically satisfies the same $d_{1/2}$ and $d_{-1/2}$ bounds as a manifold obtained by $0$-surgery on a knot in $S^3$ (see \cite[Theorem 5.2]{MR4029676}). In other words, the non-involutive version of \hyperref[cor: C]{Corollary (C)} (see \cite[Proposition 4.11]{MR1957829}) cannot obstruct a Seifert fibered homology $S^1\times S^2$ from being $0$-surgery on a knot in $S^3$. However, it turns out that the extra information contained in involutive Heegaard Floer homology can detect Seifert fibered $0$-surgery. In particular, as an application of \hyperref[cor: C]{Corollary (C)} and \hyperref[cor: D]{Corollary (D)},  we are able to prove the following extension of \cite[Theorem 7.3]{MR4029676}:
\begin{thm*}[E]\label{thm: E}
 For all positive integers $j$, $N_{j}$ cannot be obtained by $0$-surgery on a knot in $S^3$. In fact, $N_{j}$ is not the oriented boundary any smooth negative semi-definite spin 4-manifold.
\end{thm*}

To provide further context for the above theorem, it is worth noting that there do exist small Seifert fiber spaces which are obtained by $0$-surgery on a knot in $S^3$. For example, by work of Moser (see \cite{MR383406}), $0$-surgery on torus knots are small Seifert fibered spaces. More recently, Ichihara, Motegi, and Song discovered an infinite family of hyperbolic knots $\{K_{n}\}_{n\in \Z-\{0,-1,-2\}}$ with small Seifert fibered $0$-surgery (see \cite{MR2499823}). These small Seifert manifolds are different from those obtained by $0$-surgery on torus knots. 

\newpage
Interestingly, as we describe in section \ref{subsection: comparison}, 
\begin{align*}
    HF^+(-N_{1}, \mathfrak{s}_{0})\cong HF^+(-S^3_{0}(K_{1}), \mathfrak{s}_{0})
\end{align*}
where $S^3_{0}(K_{1})$ denotes $0$-surgery on $K_{1}$ and, on each side of the equation, $\mathfrak{s}_{0}$ is the unique self-conjugate $\spinc$ structure. However, 
\begin{align*}
    HFI^+(-N_{1}, \mathfrak{s}_{0})\ncong HFI^+(-S^3_{0}(K_{1}), \mathfrak{s}_{0})
\end{align*}
This gives a very concrete example of how involutive Heegaard Floer homology detects Seifert fibered $0$-surgery whereas regular Heegaard Floer homology does not.

\subsection{Organization of the paper.}
In section \ref{section:involutive}, we review involutive Heegaard Floer homology and prove Theorems \hyperref[thm: spin boundary constraints]{(A)} and \hyperref[thm: involutive zero surgery inequalities]{(B)}. In section \ref{section: plumbings}, we review some basic facts about plumbed manifolds. In section \ref{section: homology of plumbings}, we define a slightly modified version of lattice cohomology for negative semi-definite plumbings and describe how it fits with prior work of Ozsváth and Szabó, Némethi, and Rustamov. In section \ref{section:computation method}, we adapt the computational techniques of Dai-Manolescu and Rustamov to describe a method for computing the involutive Heegaard Floer homology of certain negative semi-definite plumbings. In section \ref{section: examples}, we use this method to compute the involutive Heegaard Floer homology of the manifolds $\{N_{j}\}_{j\in \mathbb{N}}$ as well as $S^3_{0}(K_{1})$. In particular, these calculations together with the results of section \ref{section:involutive}, enable us to prove \hyperref[thm: nj not zero surgery]{Theorem (E)}. 

\subsection{Acknowledgements.} I thank my advisor Tom Mark for suggesting this project. I am deeply grateful for his guidance and countless helpful conversations. I was supported by NSF RTG grant DMS-1839968.

\newpage

\section{Involutive Heegaard Floer homology}\label{section:involutive}
In this section, we briefly review the construction of involutive Heegaard Floer homology. We then recall the involutive $d$ invariants, $\underline{d}$ and $\bar{d}$,  defined by Manolescu and Hendricks for rational homology spheres and define analogous invariants, $\underline{d}_{\pm 1/2}$ and $\bar{d}_{\pm 1/2}$, for closed oriented 3-manifolds with first homology $\Z$. We show that $\underline{d}_{\pm 1/2}$ and $\bar{d}_{\pm 1/2}$ are spin integer homology cobordism invariants and use them to establish constraints on the intersection forms of negative semi-definite spin 4-manifolds whose boundary is a 3-manifold with first homology $\Z$. Furthermore, we establish new obstructions to a 3-manifold being realized as $0$-surgery on a knot in an integer homology sphere. 

We assume the reader is familiar with Heegaard Floer homology (see for example: \cite{MR2113019}, \cite{MR2113020}, \cite{MR2249247}, \cite{MR2249248}).

\subsection{Notation/Conventions:} 
\begin{itemize}
\item We use $\F=\Z_{2}$ coefficients for all Heegaard Floer and involutive Heegaard Floer homology groups.

\item Given a graded $\F[U]$-module $\mathcal{A}$, we let $\mathcal{A}[r]$ be the graded $\F[U]$-module defined by $\mathcal{A}[r]_{k}= \mathcal{A}_{k+r}$. The subscripts denote the homogeneous elements of the corresponding grading. 

\item We let $\mathcal{T}^+=\F[U,U^{-1}]/(U\cdot \F[U])$ be the graded $\F[U]$-module where $\gr(U^n) = -2n$.

\item We let
$\mathcal{T}^+_{d}\coloneqq\mathcal{T}^+[-d]$. In other words, $\mathcal{T}^+_{d}$ is the $\F[U]$-module $\mathcal{T}^+$ with grading shifted so that the minimal non-zero grading level is $d$. 
\end{itemize}

\subsection{Review of involutive Heegaard Floer homology}
For complete details of the construction of involutive Heegaard Floer homology see \cite{MR3649355}.

Let $Y$ be any closed, connected, oriented 3-manifold. Fix a $\spinc$ structure $\mathfrak{s}$ on $Y$ and let $\overline{\omega} = \{\mathfrak{s},\bar{\mathfrak{s}}\}$ be the orbit of $\mathfrak{s}$ under the conjugation action. Let $\mathcal{H} = (H, J)$ be a Heegaard pair, i.e., $H = (\Sigma, \alpha, \beta, z)$ is a pointed Heegaard diagram for $Y$ admissible with respect to $\mathfrak{s}$ and $J$ is a generic family of almost complex structures on $\Sym^g(\Sigma)$. Given this setup, define 
\begin{align*}
CF^{\circ}(\mathcal{H},\overline{\omega}) = \bigoplus_{\mathfrak{t}\in\overline{\omega}}CF^{\circ}(\mathcal{H}, \mathfrak{t})
\end{align*}
where $CF^{\circ}(\mathcal{H}, \mathfrak{t})$ is the usual Heegaard Floer chain complex associated to $(\mathcal{H},\mathfrak{t})$.

We call $\overline{\mathcal{H}} = (\overline{H}, \bar{J})$ the conjugate Heegaard pair where $\overline{H}= (-\Sigma, \beta, \alpha, z)$ and where $\bar{J}$ is the corresponding conjugate family of almost complex structures. As shown by Ozsváth and Szabó in \cite[Theorem 2.4]{MR2113020}, there is a canonical isomorphism of chain complexes:
\begin{align*}
    \eta: CF^{\circ}(\mathcal{H}, \mathfrak{s})\to CF^{\circ}(\overline{\mathcal{H}}, \bar{\mathfrak{s}})
\end{align*}
Moreover, $H$ and $\overline{H}$ both represent the same 3-manifold $Y$; swapping the order of the $\alpha$ and $\beta$ curves and reversing the orientation of $\Sigma$ both have the effect of reversing the orientation on $Y$ and thus cancel each other out. One may think of $\overline{H}$ as being obtained from $H$ by flipping the handle decomposition corresponding to $H$ upside down. 

Using naturality results of Juhász, Thurston, and Zemke (see \cite{JTZ}), it is observed by Hendricks and Manolescu in \cite[Proposition 2.3]{MR3649355} that given two Heegaard pairs representing the same 3-manifold there is a chain homotopy equivalence between their respective Heegaard Floer chain complexes. Furthermore, these chain homotopy equivalences form a transitive system. In particular, since $\mathcal{H}$ and $\overline{\mathcal{H}}$ both represent $Y$, we get a chain homotopy equivalence:
\begin{align*}
    \Phi(\overline{\mathcal{H}}, \mathcal{H}): CF^{\circ}(\overline{\mathcal{H}}, \mathcal{\bar{\mathfrak{s}}})\to CF^{\circ}(\mathcal{H}, \mathcal{\bar{\mathfrak{s}}})
\end{align*}
Taking the composition of $\eta$ and $\Phi$, we obtain a map:
\begin{align*}
    \iota = \Phi(\overline{\mathcal{H}}, \mathcal{H})\circ\eta: CF^{\circ}(\mathcal{H},\mathfrak{s})\rightarrow CF^{\circ}(\mathcal{H},\bar{\mathfrak{s}})
\end{align*}
which is uniquely determined up to chain homotopy. By swapping the roles of $\mathfrak{s}$ and $\bar{\mathfrak{s}}$ in the above discussion, we get a second map going in the opposite direction which, by an abuse of notation, we again call $\iota$.
\begin{align*}
    \iota: CF^{\circ}(\mathcal{H},\bar{\mathfrak{s}})\rightarrow CF^{\circ}(\mathcal{H},\mathfrak{s})
\end{align*}
It is shown in \cite{MR3649355} that $\iota^2:CF^{\circ}(\mathcal{H},\mathfrak{s})\to CF^{\circ}(\mathcal{H},\mathfrak{s})$ is chain homotopic to the identity. 

By a further abuse of notation, we let $\iota$ also denote the direct sum of the two $\iota$ maps above, i.e., 
\begin{align*}
    \iota: CF^{\circ}(\mathcal{H},\overline{\omega})\rightarrow CF^{\circ}(\mathcal{H},\overline{\omega})
\end{align*}
We then define the involutive Heegaard Floer complex, $CFI^{\circ}(\mathcal{H},\overline{\omega})$, to be the mapping cone complex: 
\begin{align*}
CF^{\circ}(\mathcal{H}, \overline{\omega})\xrightarrow{Q(1+\iota)}Q\cdot CF^{\circ}(\mathcal{H}, \overline{\omega})[-1] 
\end{align*}
Here, $Q$ is a formal variable that shifts the grading down by $1$. Therefore, as graded $\F[U]$-modules, $Q\cdot CF^{\circ}(\mathcal{H}, \overline{\omega})[-1] \cong  CF^{\circ}(\mathcal{H}, \overline{\omega})$ (strictly, these are $\Z_{2}$-graded modules; there is only an absolute $\Q$-grading lifting the $\Z_{2}$-grading when $\mathfrak{s}$ is torsion, for example when $\mathfrak{s}$ is self-conjugate). Introducing the formal variable $Q$ gives $CFI^{\circ}(\mathcal{H},\overline{\omega})$ the extra structure of a $\F[Q,U]/(Q^2)$-module rather than just an $\F[U]$-module. The involutive Heegaard Floer homology, $HFI^{\circ}(\mathcal{H}, \overline{\omega})$, is then defined to be the homology of $CFI^{\circ}(\mathcal{H},\overline{\omega})$. It turns out that the isomorphism class of $HFI^{\circ}(\mathcal{H}, \overline{\omega})$ as a graded $\F[Q,U]/(Q^2)$-module is independent of the choice of auxiliary data $\mathcal{H}$. Therefore, we will write $HFI^{\circ}(Y, \overline{\omega})$ rather than $HFI^{\circ}(\mathcal{H}, \overline{\omega})$. If $\mathfrak{s}$ is self-conjugate ($\mathfrak{s}=\bar{\mathfrak{s}})$, we write $HFI^{\circ}(Y, \mathfrak{s})$. 

\begin{rem}
Since $HFI^{\circ}(Y, \overline{\omega})$ is currently only defined up to isomorphism, it is important to highlight that when one considers elements of (or maps on) $HFI^{\circ}$, one needs to make a choice of auxiliary data. It is not known whether canonical $\F[Q,U]/(Q^2)$-modules can be associated to each pair $(Y, \overline{\omega})$. For that, one would need higher order naturality results. See \cite[Section 2.4]{MR3649355} for more details about this issue. 
\end{rem}

\subsection{Involutive $d$ invariants}
In \cite[Section 5]{MR3649355}, Hendricks and Manolescu define involutive $d$ invariants, denoted $\bar{d}$ and $\underline{d}$, for self-conjugate $\spinc$ structures of rational homology spheres. Before recalling their definitions and generalizing them to 3-manifolds with $H_{1}=\Z$, we need to review a few basic properties. 

\newpage
\begin{prop}[{See \cite[Proposition 4.6]{MR3649355}}]\label{prop: HFI exact triangle}
Suppose $Y$ is a closed, connected, oriented 3-manifold and $\mathfrak{s}\in\Spinc(Y)$ with $\mathfrak{s}=\bar{\mathfrak{s}}$. Then, there exists an exact triangle of $\F[U]$-modules: 
\begin{center}
\begin{tikzcd}
HF^{\circ}(Y,\mathfrak{s})\arrow[rr,"Q(1+\iota_{*})"] && Q\cdot HF^{\circ}(Y,\mathfrak{s})[-1]\arrow[dl,"g"]\\
&HFI^{\circ}(Y,\mathfrak{s})\arrow[ul, "h"]&
\end{tikzcd}
\end{center}
where $h$ decreases grading by $1$ and the maps $Q(1+\iota_{*})$ and $g$ preserve grading. 
\end{prop}

\begin{cor}\label{cor: vanishing of map}
With $(Y,\mathfrak{s})$ as in the previous proposition, if $HF^{\circ}_{r}(Y, \mathfrak{s}) \cong 0$ or $\F$, then the map $Q(1+\iota_{*}): HF^{\circ}_{r}(Y, \mathfrak{s})\rightarrow Q\cdot HF^{\circ}_{r}(Y, \mathfrak{s})[-1]$ is trivial.
\end{cor}
\begin{proof}
Since $\iota^2$ is chain homotopic to the identity, the induced map $\iota^2_{*} = 1$. In particular, $\iota_{*}$ is an automorphism. Since the only automorphisms of $\F$ or $0$ are the identity, if $r$ is a grading for which $HF^{\circ}_{r}(Y, \mathfrak{s}) \cong 0$ or $\F$, then $\iota_{*}$ is the identity. Thus, $Q(1+\iota_{*}) = Q(1+1) = 0$.
\end{proof}

Next we recall a structure result for the $\infty$-flavor of Heegaard Floer homology. To be consistent with the reference \cite{MR2113020}, we phrase the next theorem in terms of $\Z$-coefficients. However, we will only be concerned with the mod 2 reduction of this result.

\begin{thm}[{See \cite[section 10]{MR2113020}}]\label{thm - OS HF infinity}
Let $Y$ be a closed, connected, oriented 3-manifold. If $b_{1}(Y)\leq 2$, then there exists an equivalence class of orientation system over $Y$ such that for any torsion spin$^c$ structure $\mathfrak{s}$, we have:
\begin{align*}
HF^{\infty}(Y,\mathfrak{s}) \cong \Z[U, U^{-1}]\otimes_{\Z}\Lambda^*H^1(Y;\Z)
\end{align*}
as $\Z[U]\otimes_{\Z}\Lambda^*(H_{1}(Y;\Z)/\text{Tors})$-modules.
\end{thm}

In Heegaard Floer terminology, $HF^{\infty}$ is said to be \textit{standard} if it satisfies the conclusion of the above theorem. In other words, Theorem \ref{thm - OS HF infinity} says that if $b_{1}(Y)\in\{0,1,2\}$, then $HF^{\infty}(Y, \mathfrak{s})$ is automatically standard. In particular, if $b_{1}(Y) = 0$, i.e., if $Y$ is a rational homology sphere, then $HF^{\infty}(Y,\mathfrak{s})\cong \F[U,U^{-1}]$. In this case, as graded $\F[U]$-modules, we have the following (non-canonical) splitting:
\begin{align*}
HF^+(Y,\mathfrak{s}) \cong \mathcal{T}^+_{d}\oplus HF^+_{\text{red}}(Y,\mathfrak{s})
\hspace{2em}
\end{align*}
where $d = d(Y, \mathfrak{s})$ is the usual $d$ invariant of $(Y, \mathfrak{s})$ and $\mathcal{T}^+_{d}$ corresponds to the image of $\pi_{*}:HF^{\infty}(Y,\mathfrak{s})\rightarrow HF^+(Y, \mathfrak{s})$. 

Similarly, if $H_{1}(Y;\Z)\cong \Z$ and $\mathfrak{s}_{0}$ is the unique torsion $\spinc$ structure on $Y$, then we have that $HF^{\infty}(Y,\mathfrak{s}_{0})\cong \F[U,U^{-1}]\oplus\F[U,U^{-1}]$ and we get the following (non-canonical) splitting:
\begin{align*}
HF^+(Y,\mathfrak{s}_{0})\cong \mathcal{T}^+_{d_{-1/2}}\oplus \mathcal{T}^+_{d_{1/2}}\oplus HF^+_{\text{red}}(Y,\mathfrak{s}_{0})
\end{align*}
where $d_{-1/2} = d_{-1/2}(Y,\mathfrak{s}_{0})$ and $d_{1/2} = d_{1/2}(Y, \mathfrak{s}_{0})$ are the two $d$ invariants for $(Y,\mathfrak{s}_{0})$ and $\mathcal{T}^+_{d_{-1/2}}\oplus \mathcal{T}^+_{d_{1/2}}$ corresponds to the $\Image(\pi_{*})$. Recall, $d_{\pm 1/2}\equiv \pm1/2\bmod 2$.

\begin{rem}
The previous paragraph applies more generally to $Y$ with $b_{1}(Y)=1$, not just $H_{1}(Y;\Z)\cong\Z$, but to simplify the exposition we will restrict to the case $H_{1}(Y;\Z)\cong\Z$. Ultimately, we are concerned with $0$-surgery applications, so this restriction suffices for our purposes. 
\end{rem}

\begin{prop}\label{prop: HFI exact triangle zero}
Let $Y$ be a closed, connected oriented 3-manifold with $b_{1}(Y)=0$ or $H_{1}(Y;\Z)\cong\Z$. If $\mathfrak{s}\in \Spinc(Y)$ with $\mathfrak{s}=\bar{\mathfrak{s}}$, then we get an exact triangle of $\F[U]$-modules:
\begin{center}
\begin{tikzcd}
HF^{\infty}(Y,\mathfrak{s})\arrow[rr,"0"] && Q\cdot HF^{\infty}(Y,\mathfrak{s})[-1]\arrow[dl,"g^{\infty}"]\\
&HFI^{\infty}(Y,\mathfrak{s})\arrow[ul, "h^{\infty}"]&
\end{tikzcd}
\end{center}
\end{prop}

\begin{proof}
By the above discussion, if $r$ is a grading for which $HF_{r}^{\infty}(Y,\mathfrak{s})\neq 0$, then $HF_{r}^{\infty}(Y,\mathfrak{s})\cong \F$. The proposition then follows immediately from Corollary \ref{cor: vanishing of map} and Proposition \ref{prop: HFI exact triangle}.
\end{proof}

We now analyze the conclusion of Proposition \ref{prop: HFI exact triangle zero} in the case $b_{1}=0$ and recall the definition of the involutive $d$ invariants $\bar{d}$ and $\underline{d}$. After this, we consider the case $H_{1}(Y;\Z)\cong\Z$. To minimize confusion, for the rest of this section we use the letter $M$ to denote rational homology spheres and the letter $Y$ to denote 3-manifolds with $b_{1} = 1$. 

Consider a rational homology sphere $M$ equipped with a self-conjugate $\spinc$ structure $\mathfrak{s}$. Then the exact triangle of Proposition \ref{prop: HFI exact triangle zero} decomposes into exact sequences:
\begin{center}
    \begin{tikzcd}
    0\arrow[r] & Q\cdot HF_{r}^{\infty}(M,\mathfrak{s})[-1]\arrow[r, ,"\sim"] & HFI_{r}^{\infty}(M,\mathfrak{s})\arrow[r]&0\text{ (if }r\equiv d(M,\mathfrak{s})\bmod 2)\\
    0\arrow[r]& HFI_{r}^{\infty}(M,\mathfrak{s})\arrow[r,"\sim"]& HF_{r-1}^{\infty}(M,\mathfrak{s})\arrow[r]&0\text{ (if }r\equiv d(M,\mathfrak{s})+1\bmod 2)
    \end{tikzcd}
\end{center}
Since the maps in the exact triangle are $U$-equivariant, we further get that $HFI^{\infty}$ splits as a graded $\F[U]$-module as follows:
\begin{align*}
    HFI^{\infty}(M,\mathfrak{s})\cong Q\cdot HF^{\infty}(M,\mathfrak{s})[-1]\oplus HF^{\infty}(M,\mathfrak{s})[-1]
\end{align*}
This splitting is canonical since $HF^{\infty}_{r}$ is supported in alternating degrees. Moreover, as graded $\F[Q,U]/(Q^2)$-modules (up to possibly an overall grading shift) one can check that 
\begin{align*}
    HFI^{\infty}(M,\mathfrak{s})\cong \F[Q,U, U^{-1}]/(Q^2)
\end{align*}
Therefore, we may think of $HFI^{\infty}(M,\mathfrak{s})$ as the direct sum of two doubly infinite towers:
one which is not in the image of $Q$, and the other which is the image of the first under multiplication by $Q$. Both towers have involutive grading congruent to $d(M,\mathfrak{s})\bmod 2\Z$. 

We now recall the definition of the involutive $d$ invariants introduced by Hendricks and Manolescu. To make sense of the definition, it is useful to recall that:
\begin{align*}
    \Image(\pi_{*}: HFI^{\infty}(M,\mathfrak{s})\rightarrow HFI^{+}(M,\mathfrak{s}))=\Image(U^n)
\end{align*}
for $n\gg 0$ (see \cite[Lemma 4.6]{MR2113019}). 

\begin{defns}[{See \cite[5.1 Definitions]{MR3649355}}]\label{defn: inv correction terms}
Let $M$ be an oriented rational homology 3-sphere and $\mathfrak{s}\in \Spinc(M)$ with $\mathfrak{s}=\bar{\mathfrak{s}}$. Define the lower and upper involutive correction terms of $(M, \mathfrak{s})$ to be $\underline{d}(M,\mathfrak{s})$ and $\bar{d}(M,\mathfrak{s})$, respectively, where
\begin{align*}
\underline{d}(M, \mathfrak{s}) &= \min\{r\ |\  \exists\  x \in HFI_{r}^+(M, \mathfrak{s}), x \in \Image(U^n), x\notin \Image(U^nQ)\text{ for }n\gg 0\}-1\\
\bar{d}(M,\mathfrak{s}) &= \min\{r\ |\  \exists\  x \in HFI_{r}^+(M, \mathfrak{s}), x\neq 0, x\in \Image(U^nQ)\text{ for }n\gg 0\}
\end{align*}
\end{defns}

It is conceptually useful to think of $\bar{d}$ and $\underline{d}$ in terms of a splitting of $HFI^+$ into towers and reducible elements as follows: 

\begin{cor}\label{cor: involutive rational homology structure}
Suppose $M$ is an oriented rational homology 3-sphere and $\mathfrak{s}\in \Spinc(M)$ with $\mathfrak{s}=\overline{\mathfrak{s}}$. Then, we get a (non-canonical) splitting as graded $\F[U]$-modules:
\begin{align*}
HFI^+(M,\mathfrak{s})\cong \mathcal{T}^+_{\bar{d}}\oplus\mathcal{T}^+_{\underline{d}+1}\oplus HFI^+_{\text{red}}(M,\mathfrak{s})
\end{align*}
Here, $\mathcal{T}^+_{\bar{d}}\oplus\mathcal{T}^+_{\underline{d}+1}$ corresponds to $\Image(\pi_{*})$, with  $\mathcal{T}^+_{\bar{d}}$ in the image of $Q$ and $\mathcal{T}^+_{\underline{d}+1}$ not in the image of $Q$.
\end{cor}

The invariants $\underline{d}$ and $\bar{d}$ satisfy the following basic properties:
\begin{prop}[{See \cite[Propositions 5.1, 5.2]{MR3649355}}]
With $M$ and $\mathfrak{s}$ as in Definitions \ref{defn: inv correction terms}, 
\begin{enumerate}
\item $\underline{d}(M, \mathfrak{s})\leq d(M, \mathfrak{s})\leq \bar{d}(M,\mathfrak{s})$

\item $\underline{d}(M,\mathfrak{s}) = -\bar{d}(-M, \mathfrak{s})$
\end{enumerate}
\end{prop}

Additionally, Hendricks and Manolescu generalize \cite[Theorem 9.6]{MR1957829} to the involutive setting to obtain:

\begin{thm}[{See \cite[Theorem 1.2]{MR3649355}}]\label{thm: negative-definite bound}
With $M$ and $\mathfrak{s}$ as in Definitions \ref{defn: inv correction terms}, if $X$ is a smooth negative definite 4-manifold with boundary $M$ and $\mathfrak{t}$ is a spin structure on $X$ such that $\mathfrak{t}|_{M}=\mathfrak{s}$, then 
\begin{align*}
    \rank(H^{2}(X;\Z))\leq 4\underline{d}(M,\mathfrak{s})
\end{align*}
 \end{thm}

The method of proof of Theorem \ref{thm: negative-definite bound} is used to further show that $\underline{d}$ and $\bar{d}$ are spin rational homology cobordism invariants. 

Now suppose $Y$ is a closed oriented 3-manifold with $H_{1}(Y;\Z)\cong\Z$ and let $\mathfrak{s}_{0}$ be the unique torsion $\spinc$-structure on $Y$. Then the exact triangle of Proposition \ref{prop: HFI exact triangle zero} decomposes into short exact sequences:
\begin{center}
    \begin{tikzcd}
    0\arrow[r] & Q\cdot HF_{r}^{\infty}(Y,\mathfrak{s}_{0})[-1]\arrow[r, ,"g^{\infty}"] & HFI_{r}^{\infty}(Y,\mathfrak{s}_{0})\arrow[r, "h^{\infty}"]&HF_{r-1}^{\infty}(Y,\mathfrak{s}_{0})\arrow[r] &0
    \end{tikzcd}
\end{center}
These short exact sequences are of the form:
\begin{center}
    \begin{tikzcd}
    0\arrow[r] &\F\arrow[r]&\F\oplus\F\arrow[r]&\F\arrow[r]&0
    \end{tikzcd}
\end{center}
Therefore, as vector spaces, we get a splitting:
\begin{align*}
    HFI_{r}^{\infty}(Y,\mathfrak{s}_{0})\cong Q\cdot HF_{r}^{\infty}(Y,\mathfrak{s}_{0})[-1]\oplus HF_{r-1}^{\infty}(Y,\mathfrak{s}_{0})
\end{align*}
where each summand is one dimensional. Unlike in the $b_{1}=0$ case, this splitting is not canonical. However, we are still able to get the following structure result:
\begin{prop}\label{prop: HFI b1 standard}
Suppose $Y$ is a closed connected oriented 3-manifold with $H_{1}(Y;\Z)\cong \Z$ and $\mathfrak{s}_{0}\in \Spinc(Y)$ is the unique $\spinc$ structure with $\mathfrak{s}_{0}=\bar{\mathfrak{s}}_{0}$. Then, as graded $\F[Q,U]/(Q^2)$-modules, 
\begin{align*}
    HFI^{\infty}(Y,\mathfrak{s}_{0})\cong \F[Q,U, U^{-1}]/(Q^2)\oplus \F[Q,U, U^{-1}]/(Q^2)
\end{align*}
where, on the right side of the equation, the first factor has gradings congruent to $1/2\bmod 2$ and the second factor has gradings congruent to $-1/2\bmod 2$.
\end{prop}
\begin{proof}
Fix a Heegaard pair $\mathcal{H}=(H,J)$ representing $Y$ and admissible with respect to $\mathfrak{s}_{0}$. Let $\partial^{I}$ be the boundary map on the involutive chain complex. We can compactly write $\partial^{I}$ as $\partial^I= \partial +Q(1+\iota)$ where $\partial$ is the usual boundary map on the Heegaard Floer chain complex extended by $Q$-linearity. 

By Theorem \ref{thm - OS HF infinity}, $HF_{1/2}^{\infty}(\mathcal{H},\mathfrak{s}_{0})\cong \F$ and $HF_{-1/2}^{\infty}(\mathcal{H},\mathfrak{s}_{0})\cong \F$. Let $\alpha\in HF_{1/2}^{\infty}(\mathcal{H},\mathfrak{s}_{0})$ and $\beta\in HF_{-1/2}^{\infty}(\mathcal{H},\mathfrak{s}_{0})$ be the unique non-zero generators. Let $a, b\in CF^{\infty}(\mathcal{H},\mathfrak{s}_{0})$ be representatives of $\alpha$ and $\beta$ respectively. Then, 
the unique non-zero element in the image of 
\begin{align*}
    g^{\infty}:Q\cdot HF_{1/2}^{\infty}(\mathcal{H},\mathfrak{s}_{0})[-1]\to HFI_{1/2}^{\infty}(\mathcal{H},\mathfrak{s}_{0})
\end{align*}
is $[Qa]$. Similarly, $[Qb]$ is the unique non-zero element in the image of 
\begin{align*}
    g^{\infty}:Q\cdot HF_{-1/2}^{\infty}(\mathcal{H},\mathfrak{s}_{0})[-1]\to HFI_{-1/2}^{\infty}(\mathcal{H},\mathfrak{s}_{0})
\end{align*}
As we have observed above, $1+\iota_{*}$ is the zero map on homology. Therefore, there exists some $x,y\in CF^{\infty}(\mathcal{H}, \mathfrak{s}_{0})$ such that $(1+\iota)a = \partial x$ and $(1+\iota)b = \partial y$. Thus, $\partial^{I}(a+Qx) = 0$ and $\partial^{I}(b+Qy) = 0$. Furthermore, we have that $Q[a+Qx] = [Qa]$ and $Q[b+Qy]=[Qb]$. Therefore, the first summand in the decomposition can be taken to be $\left(\F[Q,U,U^{-1}]/(Q^2)\right)[b+Qy]$ and the second to be $\left(\F[Q,U,U^{-1}]/(Q^2)\right)[a+Qx]$. 
\end{proof}

The isomorphism in Proposition \ref{prop: HFI b1 standard} is not canonical with respect to a given Heegaard pair $\mathcal{H}=(H,J)$ because the elements $[a+Qx]$ and $[b+Qy]$ depend on our choice of representatives $a,b,x,y$. Despite this, we can still define involutive $d$ invariants in this situation. We only need to know the $\F[Q,U]/(Q^2)$-module structure of $HFI^{\infty}$, regardless of a canonical isomorphism. 

\begin{defns}\label{defn: b1=1 inv correction terms}
Let $Y$ be a closed oriented 3-manifold with $H_{1}(Y;\Z) \cong \Z$. Let $\mathfrak{s}_{0}$ be the unique $\spinc$ structure on $Y$ with $\mathfrak{s}_{0}=\bar{\mathfrak{s}}_{0}$. Define: 
{\footnotesize
\begin{align*}
\underline{d}_{1/2}(Y, \mathfrak{s}_{0}) &= \min\{r\ |\ r\equiv -1/2\bmod 2, \exists x \in HFI_{r}^+(Y, \mathfrak{s}_{0}), x \in \Image(U^n), x\notin \Image(U^nQ)\text{ for }n\gg 0\}-1\\
\underline{d}_{-1/2}(Y, \mathfrak{s}_{0}) &= \min\{r\ |\ r\equiv 1/2\bmod 2,  \exists x \in HFI_{r}^+(Y, \mathfrak{s}_{0}), x \in \Image(U^n), x\notin \Image(U^nQ)\text{ for }n\gg 0\}-1\\
\bar{d}_{1/2}(Y,\mathfrak{s}_{0}) &= \min\{r\ |\ r\equiv 1/2\bmod 2, \exists x \in HFI_{r}^+(Y, \mathfrak{s}_{0}), x\neq 0, x\in \Image(U^nQ)\text{ for }n\gg 0\}\\
\bar{d}_{-1/2}(Y,\mathfrak{s}_{0}) &= \min\{r\ |\  r\equiv -1/2\bmod 2, \exists x \in HFI_{r}^+(Y, \mathfrak{s}_{0}), x\neq 0, x\in \Image(U^nQ)\text{ for }n\gg 0\}
\end{align*}
}%
\end{defns}

\begin{rem}
Since $\mathfrak{s}_{0}$ is unique, we will often just write $\underline{d}_{\pm 1/2}(Y)$ and $\bar{d}_{\pm 1/2}(Y)$, or $\underline{d}_{\pm 1/2}$ and $\bar{d}_{\pm 1/2}$ if $Y$ is clear from context. 
\end{rem}
As in the $b_{1}=0$ case, it is again useful to think of these invariants in terms of a splitting of $HFI^+$.
\begin{cor}\label{cor: involutive H1=Z structure}
Suppose $Y$ is a closed, connected, oriented 3-manifold with $H_{1}(Y,\Z)\cong \Z$ and $\mathfrak{s}_{0}\in \Spinc(Y)$ is the unique $\Spinc$ structure with $\mathfrak{s}_{0}=\overline{\mathfrak{s}}_{0}$. Then, there exists a (non-canonical) splitting: 
\begin{align*}
HFI^{+}(Y,\mathfrak{s}_{0})\cong \mathcal{T}^{+}_{\bar{d}_{1/2}}\oplus \mathcal{T}^{+}_{\bar{d}_{-1/2}}\oplus\mathcal{T}^{+}_{\underline{d}_{1/2}+1}\oplus \mathcal{T}^{+}_{\underline{d}_{-1/2}+1}\oplus HFI^{+}_{\text{red}}(Y,\mathfrak{s}_{0})
\end{align*}
where $\mathcal{T}^{+}_{\bar{d}_{1/2}}\oplus \mathcal{T}^{+}_{\bar{d}_{-1/2}}\oplus\mathcal{T}^{+}_{\underline{d}_{1/2}+1}\oplus \mathcal{T}^{+}_{\underline{d}_{-1/2}+1}$ corresponds to $\Image(\pi_{*})$ and $\mathcal{T}^+_{\bar{d}_{1/2}}\oplus \mathcal{T}^+_{\bar{d}_{-1/2}}$ is contained in the image of multiplication by $Q$.
\end{cor}

\begin{prop}\label{ref:involutive correction term properties}
The involutive correction terms $\underline{d}_{\pm 1/2}$ and $\bar{d}_{\pm 1/2}$ satisfy the following basic properties:
\begin{enumerate}
\item $\underline{d}_{\pm 1/2}(Y)\leq d_{\pm 1/2}(Y)\leq \bar{d}_{\pm 1/2}(Y)$

\item $\underline{d}_{\pm 1/2}(Y) = -\bar{d}_{\mp 1/2}(-Y)$
\end{enumerate}
\end{prop}
\begin{proof}
The proof of $(1)$ follows from the same arguments as the proof of \cite[Proposition 5.1]{MR3649355}. The proof of $(2)$ follows from \cite[Proposition 4.4]{MR3649355} and the same arguments as in the proof of \cite[Proposition 5.2]{MR3649355}.
\end{proof}

\subsection{Spin filling constraints, homology cobordism invariance, and $0$-surgery obstruction.}
In \cite[Theorem 9.11]{MR1957829}, Ozsváth and Szabó establish constraints in terms of $d_{\pm 1/2}$ on the intersection form of a negative semi-definite 4-manifold with boundary a given 3-manifold $Y$ with $H_{1}(Y;\Z) \cong \Z$. Furthermore, Ozsváth and Szabó establish $0$-surgery obstructions in terms of $d_{\pm 1/2}$ (see \cite[Corollary 9.14, Proposition 4.11]{MR1957829}). In this subsection, we establish the analogous results in the involutive setting.

\begin{thm}\label{thm: spin boundary constraints}
Suppose $X$ is a smooth oriented negative semi-definite spin 4-manifold with boundary a 3-manifold $Y$ with $H_{1}(Y;\Z)\cong \Z$. 
\begin{enumerate}
    \item If the restriction $H^{1}(X;\Z)\to H^{1}(Y;\Z)$ is trivial, then 
    \begin{align*}
        b_{2}(X)-3\leq 4\underline{d}_{-1/2}(Y)
    \end{align*}
    
    \item If the restriction $H^{1}(X;\Z)\to H^{1}(Y;\Z)$ is non-trivial, then
    \begin{align*}
        b_{2}(X)+2\leq 4\underline{d}_{1/2}(Y)
    \end{align*}
\end{enumerate}
\end{thm}

\begin{proof}
Let $\mathfrak{s}$ be a spin structure on $X$. In particular, $c_{1}^2(\mathfrak{s}) = 0$. We follow the proof strategy of \cite[Theorem 9.11]{MR1957829}.

(1) Suppose the restriction $H^1(X;\Z)\to H^1(Y;\Z)$ is trivial. First, surger out all of $b_{1}(X)$ without changing the non-degenerate part of the intersection form of $X$. Then, remove a ball from $X$ to obtain $W$ which we regard as a cobordism $W:S^3\rightarrow Y$. As observed in the proof of \cite[Theorem 9.11]{MR1957829}, the map induced from the cobordism $W$ 
\begin{align*}
F^{\infty}_{W,\mathfrak{s}|_{W}}: HF^{\infty}(S^3)\rightarrow HF^{\infty}(Y,\mathfrak{s}|_{Y})
\end{align*}
is injective with image equal to the doubly infinite tower with degrees congruent to $-1/2\bmod 2$ and shifts degree by $\ell = \frac{b_{2}(X)-3}{4}$. Also, by \cite[Section 4.5]{MR3649355} there exists an induced map 
\begin{align*}
    F^{I,\infty}_{W,\mathfrak{s}|_{W},\alpha}:HFI^{\infty}(S^3)\to HFI^{\infty}(Y,\mathfrak{s}_{0})
\end{align*}
which also shifts degree by $\ell = \frac{b_{2}(X)-3}{4}$. Note that the involutive cobordism map $F^{I,\infty}_{W,\mathfrak{s}|_{W},\alpha}$ depends on an additional choice of auxiliary data $\alpha$.

Combining the results in \cite[Section 4.5]{MR3649355} with Proposition \ref{prop: HFI exact triangle zero}, we see that for every even integer $r$, we have the following commutative diagram with exact horizontal rows:
\begin{center}
\begin{tikzpicture}
\node[scale = .5] at (0,0)
{\begin{tikzcd}
0 \arrow[rr, "Q(1+\iota_{*})"] &                                    & {QHF^{\infty}_{r+1+\ell}(Y,\mathfrak{s}_{0})[-1]} \arrow[rr, "g^{\infty}_{Y}"] \arrow[dd, "\pi_{Y}", pos = .7] &                                                        & {HFI^{\infty}_{r+1+\ell}(Y,\mathfrak{s}_{0})} \arrow[rr, "h^{\infty}_{Y}"] \arrow[dd, "\pi^I_{Y}", pos = .7] &                                                     & {HF^{\infty}_{r+\ell}(Y,\mathfrak{s}_{0})} \arrow[rr, "Q(1+\iota_{*})"] \arrow[dd,"\pi_{Y}", pos = .7] &                         & 0\\
             & 0 \arrow[ru] \arrow[rr, crossing over] \arrow[dd] &                                                                      & HFI^{\infty}_{r+1}(S^3) \arrow[ru, "F^{I, \infty}_{W,\mathfrak{s}|_{W},\alpha}"] \arrow[rr, crossing over, "h^{\infty}_{S^3}", pos =.7]&                                                                      & HF^{\infty}_{r}(S^3) \arrow[ru, "F^{\infty}_{W, \mathfrak{s}|_{W}}"] \arrow[rr, crossing over, ]&                                                            &0 \arrow[ru] &                                         \\
             &                                    & {QHF^+_{r+1+\ell}(Y,\mathfrak{s}_{0})[-1]} \arrow[rr, "g^+_{Y}", pos = .3]                   &                                                        & {HFI^+_{r+1+\ell}(Y,\mathfrak{s}_{0})} \arrow[rr, "h^+_{Y}", pos = .3]                   &                                                     & {HF^+_{r+\ell}(Y,\mathfrak{s}_{0})} &                         &\hspace{10em}  \\
             & QHF^+_{r+1}(S^3)[-1] \arrow[ru] \arrow[rr]            &                                                                      & HFI^+_{r+1}(S^3)\arrow[uu, <-,crossing over, swap, "\pi^I_{S^3}", pos = .7] \arrow[ru, "F^{I,+}_{W,\mathfrak{s}|_{W},\alpha}"] \arrow[rr, "h^+_{S^3}", pos = .7]                   &                                                                      & HF^+_{r}(S^3)\arrow[uu, <-, crossing over, swap, "\pi_{S^3}", pos = .7] \arrow[ru, "F^+_{W, \mathfrak{s}|_{W}}"]                &                                                            & &                                        
\end{tikzcd}};
\end{tikzpicture}
\end{center}
By definition of $\underline{d}_{-1/2}(Y)$, there exists some $y^+$ in $HFI_{\underline{d}_{-1/2}+1}^{+}(Y,\mathfrak{s}_{0})$ such that $y^+\in \Image(U^n)$ for $n\gg0$ and $y^+\notin\Image(U^nQ)$ for $n\gg 0$. The condition $[y^+\in \Image(U^n)$ for $n\gg0]$ is equivalent to the condition $[y^+\in \Image(\pi^I_{Y})]$. Therefore, there exists some $y^{\infty}\in HFI^{\infty}_{\underline{d}_{-1/2}+1}(Y,\mathfrak{s}_{0})$ such that $\pi^I_{Y}(y^{\infty})=y^+$. The condition [$y^+\notin\Image(U^nQ)$ for $n\gg 0$] implies that $y^{\infty}\notin \Image(g^{\infty}_{Y})$. Therefore, by exactness, $h^{\infty}_{Y}(y^{\infty})\neq 0\in HF^{\infty}_{\underline{d}_{-1/2}}(Y,\mathfrak{s}_{0})$. By assumption, the map $F^{\infty}_{W,\mathfrak{s}|_{W}}: HF^{\infty}_{\underline{d}_{-1/2}-\ell}(S^3)\rightarrow HF^{\infty}_{\underline{d}_{-1/2}}(Y,\mathfrak{s}|_{Y})$ is an isomorphism. Moreover, by exactness, the map $h^{\infty}_{S^3}:HFI^{\infty}_{\underline{d}_{-1/2}+1-\ell}(S^3)\to HF^{\infty}_{\underline{d}_{-1/2}-\ell}(S^3)$ is also an isomorphism. Therefore, there exists some $x^{\infty}\in HFI^{\infty}_{\underline{d}_{-1/2}+1-\ell}(S^3)$ such that $(F^{\infty}_{W,\mathfrak{s}|_{W}}\circ h^{\infty}_{S^3})(x^{\infty})=h_{Y}^{\infty}(y^{\infty})$. Let $z^{\infty} = F^{I,\infty}_{W,\mathfrak{s}|_{W},\alpha}(x^{\infty})\in HFI^{\infty}_{\underline{d}_{-1/2}+1}(Y,\mathfrak{s}_{0})$. By commutativity, $h_{Y}^{\infty}(z^{\infty}) = h^{\infty}_{Y}(y^{\infty})$. Therefore, $z^{\infty}+y^{\infty}\in \ker(h_{Y}^{\infty})$. So, by exactness, there exists some $w^{\infty}\in Q\cdot HF^{\infty}_{\underline{d}_{-1/2}+1}(Y,\mathfrak{s}_{0})[-1]$ such that $g^{\infty}_{Y}(w^{\infty})=z^{\infty}+y^{\infty}$. If $\pi_{Y}^I(z^{\infty})=0$, then that would imply $\pi^I_{Y}(g^{\infty}(w^{\infty}))=y^+$. But this would be a contradiction because that would imply $y^+\in \Image(U^nQ)$ for $n\gg0$. Therefore, $\pi_{Y}^I(z^{\infty})\neq0$. Thus, $(\pi^I_{Y}\circ F^{I,\infty}_{W,\mathfrak{s}|_{W},\alpha})(x^{\infty})\neq 0$. So, by commutativity, $(F^{I,+}_{W,\mathfrak{s}|_{W},\alpha}\circ \pi^I_{S^3})(x^{\infty})\neq 0$. In particular, $\pi^I_{S^3}(x^{\infty})\neq 0$. Therefore, the element $x^+=\pi^I_{S^3}(x^{\infty})\in HFI^+_{\underline{d}_{-1/2}+1-\ell}(S^3)$ has the property that $x^+\in \Image(U^n)$ for $n\gg 0$ and $x^+\notin\Image(U^nQ)$ for $n\gg0$. It follows that:
\begin{align}\label{eq: bounds from proof}
    \underline{d}(S^3)+1\leq \underline{d}_{-1/2}(Y)+1-\ell
\end{align}
Observing that $\underline{d}(S^3)=0$ and rearranging/canceling the terms, we get:
\begin{align*}
    b_{2}(X)-3\leq 4\underline{d}_{-1/2}(Y)
\end{align*}

(2) Now suppose the restriction $H^{1}(X;\Z)\rightarrow H^{1}(Y;\Z)$ is non-trivial. Surger out the 1 dimensional homology of $X$ until $b_{1}(X)=1$ and so that the map $H^{1}(X;\Z)\rightarrow H^{1}(Y;\Z)$ is still non-trivial. Again, remove a ball from $X$ to obtain a cobordism $W:S^3\to Y$. In this case, the induced map
\begin{align*}
F^{\infty}_{W,\mathfrak{s}|_{W}}: HF^{\infty}(S^3)\rightarrow HF^{\infty}(Y,\mathfrak{s}|_{Y})
\end{align*}
is injective with image equal to the doubly infinite tower with degrees congruent to $+1/2\bmod 2$. The degree shift of this map is now $\frac{b_{2}(X)+2}{4}$. We then repeat the analogous diagram chase to establish the inequality. We leave the details to the reader.
\end{proof}

\begin{cor}\label{cor: obstruction to any negative semi-definite}
Suppose $Y$ is a closed oriented 3-manifold with $H_{1}(Y;\Z)\cong \Z$. If 
\begin{align*}
    \underline{d}_{-1/2}(Y)< -1/2\hspace{2em}\text{and}\hspace{2em}\underline{d}_{1/2}(Y)<1/2
\end{align*}
then $Y$ is not the boundary of any negative semi-definite spin manifold.
\end{cor}

\begin{proof}
Suppose $X$ is a smooth negative semi-definite spin 4-manifold with boundary $Y$. If the restriction $H^1(X;\Z)\to H^1(Y;\Z)$ is trivial, then the map $H^{1}(Y;\Z)\to H^2(X,Y;\Z)$ is injective. Since $H^{1}(Y;\Z)\cong H_{1}(Y;\Z)\cong \Z$ and $H^2(X,Y;\Z)\cong H_{2}(X;\Z)$, it follows that $b_{2}(X)\geq 1$. Hence, by Theorem \ref{thm: spin boundary constraints}, $-1/2\leq \underline{d}_{-1/2}(Y)$. If instead $H^1(X;\Z)\to H^1(Y;\Z)$ is non-trivial, then all we can say about $b_{2}(X)$ is that $b_{2}(X)\geq 0$. Theorem \ref{thm: spin boundary constraints} therefore implies $1/2\leq \underline{d}_{1/2}(Y)$. The conclusion now follows. 
\end{proof}

\begin{prop}
Suppose $Y_{1}$ and $Y_{2}$ are closed oriented 3-manifolds with $H_{1}(Y_{i};\Z)\cong \Z$ for $i\in \{1,2\}$. If there exists a spin integer homology cobordism $(W,\mathfrak{s}):Y_{1}\rightarrow Y_{2}$, then $\underline{d}_{\pm 1/2}(Y_{1})=\underline{d}_{\pm 1/2}(Y_{2})$ and $\bar{d}_{\pm 1/2}(Y_{1})=\bar{d}_{\pm 1/2}(Y_{2})$.
\end{prop}

\begin{proof}
The argument is the same as in the proof of \cite[Proposition 5.4]{MR3649355}, using the fact that $W$ induces an isomorphism 
\begin{align*}
    F^{\infty}_{W, \mathfrak{s},\alpha}: HFI^{\infty}(Y_{1},\mathfrak{s}|_{Y_{1}})\to HFI^{\infty}(Y_{2},\mathfrak{s}|_{Y_{2}})
\end{align*}
\end{proof}

\newpage
\begin{thm}\label{thm: involutive zero surgery inequalities}
Let $M$ be an oriented integer homology 3-sphere and let $Y$ and $M'$ be the 3-manifolds obtained via $0$ and $+1$ surgery respectively on a knot $K$ in $M$. Then, 
\begin{enumerate}
\item 
$\underline{d}(M)-\frac{1}{2}\leq \underline{d}_{-1/2}(Y)\hspace{2em}\text{and}\hspace{2em}\bar{d}(M)-\frac{1}{2}\leq \bar{d}_{-1/2}(Y)$

\item 
$\underline{d}_{1/2}(Y)-\frac{1}{2}\leq \underline{d}(M')\hspace{2em}\text{and}\hspace{2em}\bar{d}_{1/2}(Y)-\frac{1}{2}\leq \bar{d}(M')$
\end{enumerate}

\end{thm}

\begin{proof}
First, we prove the inequalities in $(1)$. 

Let $(W,\mathfrak{s})$ be the spin cobordism from $M$ to $Y$ obtained by attaching a 0-framed 2-handle along $K$ and let $\mathfrak{s}_{0}$ be the trivial $\spinc$ structure on $Y$. Then, then by \cite[Proposition 9.3]{MR1957829}, the induced map 
\begin{align*}
F_{W,\mathfrak{s}}^{\infty}: HF^{\infty}(M)\rightarrow HF^{\infty}(Y, \mathfrak{s}_{0})
\end{align*}
shifts grading by $-1/2$ and is injective with image equal to the doubly infinite tower with gradings congruent to $-1/2\bmod 2$. The first inequality of (1) now follows by repeating exactly the same argument as in the proof of Theorem \ref{thm: spin boundary constraints} where now $M$ assumes the role of $S^3$ and $\ell = -1/2$ (see inequality \ref{eq: bounds from proof}).

To establish the second inequality in (1), we consider the rightward continuation of the commutative diagram used in the proof of Theorem \ref{thm: spin boundary constraints} again replacing $S^3$ with $M$. Specifically, for $r$ even, we have the following commutative diagram with exact horizontal rows: 
\begin{center}
\begin{tikzpicture}
\node[scale = .5] at (0,0)
{\begin{tikzcd}
                                   & 0 \arrow[rr, "Q(1+\iota_{*})"]                         &                                                      & {QHF^{\infty}_{r-1/2}(Y,\mathfrak{s}_{0})[-1]} \arrow[rr, "g^{\infty}_{Y}"] \arrow[dd, "\pi_{Y}", pos = .7] &                                                      & {HFI^{\infty}_{r-1/2}(Y,\mathfrak{s}_{0})} \arrow[rr, "h^{\infty}_{Y}"] \arrow[dd, "\pi^I_{Y}", pos = .7] &                         & {HF^{\infty}_{r-1.5}(Y,\mathfrak{s}_{0})} \arrow[rr, "Q(1+\iota_{*})"] \arrow[dd, "\pi_{Y}", pos = .7] &  & 0 \\
0 \arrow[rr] \arrow[ru] &                                                   & QHF^{\infty}_{r}(M)[-1] \arrow[ru, "F^{\infty}_{W, \mathfrak{s}}"] \arrow[rr, crossing over, "g^{\infty}_{M}", pos = .7] \arrow[dd, "\pi_{M}"] &                                                                      & HFI^{\infty}_{r}(M) \arrow[ru, "F^{I,\infty}_{W,\mathfrak{s}, \alpha}"] \arrow[rr, crossing over] &                                                                      & 0 \arrow[ru] &                                                                     &  &   \\
                                   & &                                                      & {QHF^+_{r-1/2}(Y,\mathfrak{s}_{0})[-1]} \arrow[rr, "g^+_{Y}", pos = .3]                   &                                                      & {HFI^+_{r-1/2}(Y,\mathfrak{s}_{0})} \arrow[rr, "h^+_{Y}", pos = .3]                   &                         & {HF^+_{r-1.5}(Y,\mathfrak{s}_{0})}                              &  &   \\
\hspace{10em}&                                                   & QHF^+_{r}(M)[-1] \arrow[ru, "F^{+}_{W, \mathfrak{s}}"] \arrow[rr, "g^+_{M}"]                   &                                                                      & HFI^+_{r}(M)\arrow[uu, swap, <-, crossing over, "\pi^I_{M}", pos=.3] \arrow[ru, "F^{I,+}_{W,\mathfrak{s},\alpha}"] \arrow[rr]                   &                                                                      & HF^+_{r-1}(M) \arrow[ru]\arrow[uu, swap, <-, crossing over]            &                                                                     &  &  
\end{tikzcd}};
\end{tikzpicture}
\end{center}

Now we get that $g^{\infty}_{M}$ is an isomorphism, and we again know that $F^{\infty}_{W, \mathfrak{s}}$ is an isomorphism. Furthermore, $g^{\infty}_{Y}$ is injective with $\Image(g^{\infty}_{Y}) = \ker(h^{\infty}_{Y})$. Thus, $F^{I,\infty}_{W,\mathfrak{s},\alpha}$ maps $HFI^{\infty}_{r}(M)$ isomorphically onto $\Image(g^{\infty}_{Y})$. 

By definition of the value $\bar{d}_{-1/2}$, there exists some non-zero $y^+\in HFI^+(Y,\mathfrak{s}_{0})$ such that $\gr(y^+) = \bar{d}_{-1/2}$ and $y^+\in \Image(U^nQ)$ for $n\gg 0$. This implies that there exists some element $y^{\infty}\in \Image(g_{Y}^{\infty})\subset HFI^{\infty}_{\bar{d}_{-1/2}}(Y,\mathfrak{s}_{0})$ such that $\pi_{Y}^I(y^{\infty})=y^+$. Therefore, the unique non-zero element of $HFI_{\bar{d}_{-1/2}+1/2}^{\infty}(M)$, which we will call $x^{\infty}$, maps to $y^{\infty}$ under $F^{I,\infty}_{W,\mathfrak{s}, \alpha}$. Since $(\pi_{Y}^I\circ F^{I,\infty}_{W,\mathfrak{s},\alpha})(x^{\infty}) = y^+\neq 0$, the commutativity of the diagram implies $\pi_{M}^I(x^{\infty})\neq 0$. Additionally, $\pi_{M}^I(x^{\infty})\in \Image(U^n Q)$ for $n\gg 0$. Therefore, 
\begin{align*}
    \bar{d}(M)\leq \bar{d}_{-1/2}(Y)+\frac{1}{2}
\end{align*}

The proofs of the inequalities in $(2)$ follow the same arguments as the proofs of $(1)$, except that now we consider the maps:
\begin{align*}
F_{W', \mathfrak{s}'}^{\circ}: HF^{\circ}(Y, \mathfrak{s}_{0})\rightarrow HF^{\circ}(M')
\end{align*}
and 
\begin{align*}
F_{W', \mathfrak{s}', \alpha'}^{I,\circ}: HFI^{\circ}(Y, \mathfrak{s}_{0})\rightarrow HFI^{\circ}(M')
\end{align*}
induced by the spin cobordism $(W', \mathfrak{s}'): Y\rightarrow M'$ obtained by attaching a 2-handle to the dual of $K$ in $Y$ with framing so that the resulting space is $M'$. Analyzing the corresponding commutative diagrams and using the fact that for all $r$ even, \begin{align*}
F_{W', \mathfrak{s}'}^{\infty}: HF^{\infty}_{r+1/2}(Y, \mathfrak{s}_{0})\rightarrow HF^{\infty}_{r}(M')
\end{align*} 
is an isomorphism, we get statement $(2)$. We leave the details to the reader. 
\end{proof}

\begin{cor}\label{cor:zero surgery obstruction}
Suppose $K$ is a knot in $S^3$ and $Y$ is the result of $0$-surgery on $K$. Then,
\begin{enumerate}
\item $-\frac{1}{2}\leq \underline{d}_{-1/2}(Y)$
\item $\bar{d}_{1/2}(Y)\leq \frac{1}{2}$
\end{enumerate}
\end{cor}

\begin{proof}
$0 = d(S^3)=\underline{d}(S^3)=\bar{d}(S^3)$. Therefore, $(1)$ follows immediately from Theorem \ref{thm: involutive zero surgery inequalities}. For $(2)$, let $\bar{K}$ be the mirror of $K$. Then, $0$-surgery on $\bar{K}$ is $-Y$. Thus, we have $-\frac{1}{2}\leq \underline{d}_{-1/2}(-Y, \mathfrak{s}_{0})$. Now by Proposition \ref{ref:involutive correction term properties}, $\underline{d}_{-1/2}(-Y, \mathfrak{s}_{0}) = -\bar{d}_{1/2}(Y, \mathfrak{s}_{0})$. Therefore, $\bar{d}_{1/2}(Y,\mathfrak{s}_{0})\leq \frac{1}{2}$. 
\end{proof}

\newpage 

\section{Plumbings}\label{section: plumbings}
We now make a digression from our discussion of involutive Heegaard Floer homology to review the definition and some basic properties of plumbed 3- and 4-manifolds.  

\begin{notn}
Given a graph $\Gamma$, we denote the set of vertices of $\Gamma$ by $\mathcal{V}(\Gamma)$ and the set of edges by $\mathcal{E}(\Gamma)$. 
\end{notn}

\begin{defn}
A \textit{weighted graph} is a graph $\Gamma$ together with a function $m:\mathcal{V}(\Gamma)\rightarrow \Z$, called a \textit{weight function}. Given a vertex $v\in \mathcal{V}(\Gamma)$, we call $m(v)$ the weight of $v$. Usually we will refer to a weighted graph as $\Gamma$ and not explicitly write the weight function associated to it. 
\end{defn}

For the purposes of this paper, we will use the term \textit{plumbing graph} to mean a weighted graph $\Gamma$ such that $|\mathcal{V}(\Gamma)|<\infty$ and $\Gamma$ is a forest (i.e. a disjoint union of trees). Plumbing graphs in general can be more complicated, however for simplicity we only consider plumbing graphs of the type just described. As we explain in more detail below, our plumbed manifolds will be those obtained by plumbing disk bundles over 2-spheres according to such plumbing graphs.

\begin{con}[Plumbed manifolds]
Let $\Gamma$ be a plumbing graph and let $D^2$ denote the 2-dimensional disk. Suppose first that $\Gamma$ is connected. Then, the plumbed 4-manifold $X(\Gamma)$ is constructed in the following way:

\begin{enumerate}
    \item For each vertex $v\in \mathcal{V}(\Gamma)$, we assign a $D^2$-bundle $\pi_{v}: E(v)\rightarrow S^2$ over the 2-sphere with Euler number equal to $m(v)$. Here we are implicitly using the fact that the Euler number gives a bijection from bundle isomorphism classes of $D^2$-bundles over $S^2$ to $\Z$.

    \item For each edge, $[u,v]$, connecting vertices $u,v\in \mathcal{V}(\Gamma)$, we choose disks $D^2_{u, [u,v]}$ and $D^2_{v, [u,v]}$ in the base 2-spheres of the respective bundles $E(u)$ and $E(v)$. If a vertex $u$ is adjacent to multiple edges $[u,v_{1}], \ldots, [u,v_{\ell}]$, then we choose the discs,\\ $D^2_{u,[u,v_{1}]}, \ldots, D^2_{u,[u,v_{\ell}]}$, associated to $u$, to be pairwise disjoint in the base 2-sphere of the bundle $E(u)$.

    \item Since $D^2_{u,[u,v]}$ and $D^2_{v,[u,v]}$ are contractible, the restrictions $\pi_{u}: \pi_{u}^{-1}(D^2_{u,[u,v]})\rightarrow D^2_{u,[u,v]}$ and $\pi_{v}: \pi_{v}^{-1}(D^2_{v,[u,v]})\rightarrow D^2_{v,[u,v]}$ are trivial $D^2$-bundles. Now, for each edge $[u,v]$ we identify $\pi_{u}^{-1}(D_{u,[u,v]}^2)$ with $\pi_{v}^{-1}(D_{v,[u,v]}^2)$ by a diffeomorphism that swaps the two factors in the product structure, $D^2\times D^2$, of these bundles. In other words, after choosing trivializations of the two restriction bundles, we send $(x, y) \in D^2\times D^2$ to $(y,x)$.
\end{enumerate}
If $\Gamma$ has multiple connected components, then do the above construction to each component and boundary connect sum the resulting 4-manifolds. 
\end{con}

\begin{defn}
The 4-manifold $X(\Gamma)$ constructed from a plumbing graph $\Gamma$ by the process described above is called the \textit{plumbed 4-manifold} with plumbing graph $\Gamma$. The boundary of $X(\Gamma)$, denoted $Y(\Gamma)$, is called the \textit{plumbed 3-manifold} with plumbing graph $\Gamma$. 
\end{defn}

\begin{rem}
In general, a given plumbed 3-manifold $Y$ may bound many different plumbed 4-manifolds. In \cite{MR632532}, Neumann describes a calculus for passing between different plumbing graphs that describe the same 3-manifold. 
\end{rem}

Given a plumbing graph $\Gamma$, a Kirby diagram for $X(\Gamma)$ (which is also a surgery diagram for $Y(\Gamma)$) is given by an $m(v)$-framed unknot for each $v\in \mathcal{V}(\Gamma)$ such that any pair of these unknots is either Hopf linked or unlinked depending on whether or not there is an edge between the vertices with which the unknots correspond. 
\begin{exmp}
\end{exmp}
\begin{figure}[h]
\includegraphics[scale=.6]{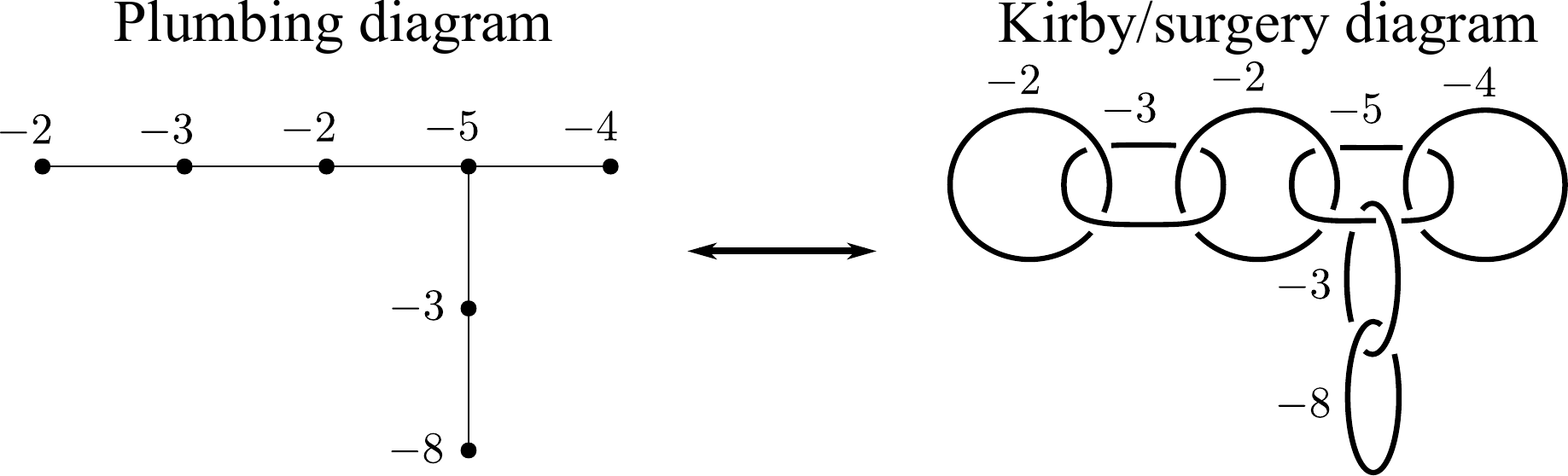}
\end{figure}

\subsection{Algebraic topological properties of plumbings}\label{subsection: alg top prop}
Fix a plumbing graph $\Gamma$ and let $X = X(\Gamma)$ and $Y = Y(\Gamma)$ be the associated plumbed 4- and 3-manifolds. Label the vertices of $\Gamma$ by $\mathcal{V}(\Gamma) = \{v_{1}, \ldots, v_{s}\}$ where $s = |\mathcal{V}(\Gamma)|$. For each $v_{j}\in \mathcal{V}(\Gamma)$, let $[v_{j}]\in H_{2}(X;\Z)$ be the homology class of the 2-sphere corresponding to the $0$-section of the $D^2$-bundle associated to $v_{j}$. Equivalently, $[v_{j}]$ is represented by the capped-off core of the corresponding 2-handle. In particular, it is easy to see that $H_{2}(X;\Z)\cong \bigoplus\limits_{j=1}^{s}\Z [v_{j}]$. Given $x = \sum a_{j}[v_{j}]\in H_{2}(X;\Z)$, we write $x\geq 0$ if $a_{j}\geq 0$ for all $j$. If in addition, $x\neq 0$, we write \hypertarget{greater than}{$x > 0$}. Given two elements $x,y\in H_{2}(X;\Z)$, we write $x\geq y$ ($x>y$) if $x-y\geq 0$ ($x-y>0$). 

Denote the intersection form of $X$ by 
\begin{align*}
    (\cdot, \cdot):H_{2}(X,\Z)\times H_{2}(X;\Z)\rightarrow \Z 
\end{align*}
By construction, 
\begin{align*}
([v_{i}],[v_{j}]) = 
\begin{cases}
m(v_{i})& i = j\\
1&\text{if }i\neq j\text{ and there is an edge }[v_{i},v_{j}]\text{ connecting }v_{i}\text{ and }v_{j}\\
0&\text{otherwise}
\end{cases}
\end{align*}
Let $B$ be the matrix of the intersection form with respect to the ordered basis $([v_{1}],\ldots, [v_{s}])$. Notice, $B$ is the incidence matrix of the graph $\Gamma$ with the $ith$-diagonal entry equal to $m(v_{i})$. 
\begin{defn}
We define the \textit{definiteness type} of a plumbing graph $\Gamma$ to be the definiteness type of its associated intersection form $(\cdot, \cdot)$, or equivalently the definiteness type of $B$. For example, we say $\Gamma$ is negative semi-definite if $(\cdot, \cdot)$ is negative semi-definite.
\end{defn}

By an abuse of notation, we will also refer to the corresponding intersection pairing on cohomology as $(\cdot, \cdot ): H^2(X, Y; \Z)\times H^2(X, Y;\Z)\rightarrow \Z$. It will be useful to, in addition, consider the slightly modified intersection pairing $(\cdot, \cdot)':  H^2(X; \Z)\times H^2(X, Y;\Z)\rightarrow \Z $ with a different domain, but still defined by the usual formula: $(\alpha,\beta)' = (\alpha\cup \beta)[X]$. 

Recall, the set of characteristic vectors of $X$, denoted $\Char(X)$, is defined by
\begin{align*}
\Char(X) &= \{\alpha\in H^2(X;\Z)\ |\ (\alpha,\beta)' \equiv (\beta, \beta)\bmod{2}, \forall \beta\in H^2(X, Y;\Z)\}\\
&= \{\alpha\in H^2(X;\Z)\ |\ \alpha(x) \equiv (x,x)\bmod{2}\text{ for all }x\in H_{2}(X;\Z)\}
\end{align*}

We now recall the relationship between the $\spinc$ structures on $X$ and $Y$ and the characteristic vectors of $X$. The first observation is that we have a commutative diagram:
\begin{center}
\begin{tikzcd}
\Spinc(X)\arrow[d,"c_{1}"]\arrow[r, "|_{Y}"]&\Spinc(Y)\arrow[d,  "c_{1}"]\\
\Char(X)\arrow[r, "\partial^{*}"]&H^2(Y;\Z)
\end{tikzcd}
\end{center}
Here, $c_{1}$ denotes the first Chern class of the determinant line bundle of the $\spinc$ structure, the top horizontal map is restriction to $Y$ and the bottom horizontal map is the restriction of the map $\partial^*:H^2(X;\Z)\rightarrow H^2(Y;\Z)$ in the long exact sequence in cohomology of the pair $(X,Y)$. The left vertical map is a bijection since $H_{1}(X,\Z)$ has no 2-torsion (see \cite[p. 56]{MR1707327} for details). Therefore, $c_{1}$ provides a canonical identification of $\Spinc(X)$ with $\Char(X)$. Furthermore, since $X$ is simply connected, we have the following commutative diagram: 
\begin{center}
\begin{tikzcd}[row sep = small, ar symbol/.style = {draw=none,"\textstyle#1" description,sloped}, isomorphic/.style = {ar symbol={\cong}}]

0 \arrow[r]           & H^1(Y;\Z) \arrow[r, "i^*"]             & {H^2(X, Y;\Z)} \arrow[r, "j^*"] & H^2(X;\Z) \arrow[r, "\partial^*"]                          & H^2(Y;\Z) \arrow[r]             & 0           \\
0 \arrow[r] & H_{2}(Y;\Z) \arrow[r, "i_{*}"] \arrow[u, isomorphic] & H_{2}(X;\Z) \arrow[r, "j_{*}"] \arrow[u, isomorphic]   & {H_{2}(X,Y;\Z)} \arrow[r, "\partial_{*}"] \arrow[u, isomorphic] & H_{1}(Y;\Z) \arrow[r] \arrow[u, isomorphic] & 0 
\end{tikzcd}
\end{center}
with exact rows coming from the long exact sequences in homology and cohomology of the pair $(X,Y)$ and with vertical isomorphisms given by Poincar\'e/Lefschetz duality. 

We have yet another commutative diagram: 
\begin{center}
\begin{tikzcd}[column sep = small, ar symbol/.style = {draw=none,"\textstyle#1" description,sloped}, isomorphic/.style = {ar symbol={\cong}}]

{\Hom(H_{2}(X;\Z), \Z)} \arrow[r, isomorphic] & H^2(X;\Z) \arrow[d, isomorphic]      \\
H_{2}(X;\Z) \arrow[r, "j_{*}"]\arrow[u, "\phi"]                      & {H_{2}(X,Y;\Z)}
\end{tikzcd}
\end{center}
Here, the top row is the isomorphism coming from the universal coefficient theorem, the right vertical map is the Lefschetz duality isomorphism, and the map $\phi$ is defined by $\phi(x) = (x,\cdot)$.

Combining the three previous diagrams we get the following commutative diagram:
\begin{center}
\begin{tikzcd}[column sep = tiny]
&&&&\Spinc(X)\isoarrow{d}{c_{1}}\arrow[r, "|_{Y}"]&\Spinc(Y)\arrow[dd, "c_{1}"]&\\
&&&&\Char(X)\arrow[d, hook]&&\\
0 \arrow[r]            & H^1(Y;\Z) \arrow[r, "i^*"]              & {H^2(X, Y;\Z)} \arrow[rr, "j^*"] &                                              & H^2(X;\Z) \arrow[r, "\partial^*"]                           & H^2(Y;\Z) \arrow[r]              & 0            \\
                       &                                           &                                    & {\Hom(H_{2}(X;\Z),\Z)} \arrow[ru,leftrightarrow, "\rotatebox{25}{\(\sim\)}"] &                                               &                                           &              \\
0 \arrow[r] & H_{2}(Y;\Z) \arrow[r, "i_{*}"] \doubleisoarrow{uu} & H_{2}(X;\Z) \arrow[rr, "j_{*}"]\arrow[ru, "\phi"] \doubleisoarrow{uu}  &                                              & {H_{2}(X,Y;\Z)} \arrow[r, "\partial_{*}"] \doubleisoarrow{uu} & H_{1}(Y;\Z) \arrow[r] \doubleisoarrow{uu} & 0 
\end{tikzcd}
\end{center}
\newpage
In addition, there is a free and transitive action of $H^2(X;\Z)$ on $\Char(X)$, defined by $(\alpha, k)\mapsto k + 2\alpha$ for all $\alpha\in H^2(X;\Z)$ and $k\in \Char(X)$. Restricting this action to $j^*(H^2(X,Y;\Z))$, we get an action of $j^*(H^2(X,Y;\Z))$ on $\Char(X)$. Let $\Char(X)/2j^*(H^2(X,Y;Z))$ denote the set of orbits of this action and denote the orbit $k + 2j^*(H^2(X,Y;\Z))$ of an element $k$ by $[k]$.

\begin{prop}
The map $\Psi:\Char(X)/2j^*(H^2(X,Y;\Z))\rightarrow \Spinc(Y)$ given by 
\begin{align*}
\Psi([k])=c_{1}^{-1}(k)|_{Y}
\end{align*}
is well-defined and is a bijection. 
\end{prop}

\begin{notn}\label{notn: characterstic equiv relation}
Justified by the above proposition, we will use $[k]$ to denote both the orbit $k + 2j^*(H^2(X,Y;\Z))$ as well as the corresponding $\spinc$ structure $\Psi([k])$.
\end{notn}

\begin{rem}\label{rem:char Q}
From the above diagram, one can see that if $k$ is a characteristic vector, then $[k]$ is a torsion $\spinc$ structure on $Y$ if and only if some integer multiple of $k$ is in the image of $j^*$. Equivalently, $[k]$ is torsion if and only if there exists some $z_{k}\in H_{2}(X;\Z)\otimes \Q$ such that $k(x) = (z_{k}, x)$ for all $x\in H_{2}(X;\Z)$. 
\end{rem}

\subsection{Rationality and weight conditions} We now recall some terminology that will be useful later when we discuss lattice cohomology and Heegaard Floer homology of plumbings.

If $\Gamma$ is a negative definite plumbing tree, then there is a special characteristic vector $K_{can}$ which is called the \textit{canonical characteristic vector}. It is defined by the equation $K_{can}(v) = -m(v)-2$ for all $v\in \mathcal{V}(\Gamma)$. 

\begin{defn}
A plumbing graph $\Gamma$ is called \textit{rational} if it is a negative definite tree which satisfies the following condition: if $x \in H_{2}(X(\Gamma);\Z)$ and \hyperlink{greater than}{$x>0$}, then 
\begin{align*}
    -\frac{K_{can}(x)-(x,x)}{2}\geq 1
\end{align*}
\end{defn}
In \cite{MR2140997}, Némethi introduces the following generalization of rational plumbings:

\begin{defn}[See {\cite[Definition 8.1]{MR2140997}}]
A negative definite plumbing tree $\Gamma$ is \textit{almost rational} if there exists a vertex $v\in \mathcal{V}(\Gamma)$ and some integer $r\leq m(v)$ such that if you replace the weight of $v$ with $r$, $\Gamma$ becomes rational. 
\end{defn}

A further generalization of this notion is the following:

\begin{defn}[See {\cite[Definition 2.1]{MR3317341}}]
A plumbing tree $\Gamma$ is \textit{type n} if there exist $n$ vertices of $\Gamma$ such that if we reduce their weights sufficiently, the plumbing becomes rational. 
\end{defn}

\begin{rem}
A type $n$ plumbing is not required to be negative definite. 
\end{rem}

Recall the \textit{degree}, denoted $\delta(v)$, of a vertex $v\in \mathcal{V}(\Gamma)$ is the number of edges adjacent to $v$. Following the terminology introduced in \cite{MR1988284}, we say a vertex is \hypertarget{bad vertex}{\textit{bad}} if $m(v)>-\delta(v)$. In particular, it can be shown that a negative definite plumbing with at most one bad vertex is almost rational. 

\newpage

\section{Heegaard Floer homology and lattice cohomology of plumbings}\label{section: homology of plumbings}
In this section, we review some of the key developments in the Heegaard Floer homology and lattice cohomology of plumbed 3-manifolds. We then present a modified version of lattice cohomology that involves passing to a quotient lattice. This presentation enables us to readily adapt and combine the work of Rustamov in \cite{Rustamov} and the work of Dai and Manolescu in \cite{MR4021102} to compute $HFI^+$ of certain negative semi-definite plumbed 3-manifolds with $b_{1} = 1$ and at most one bad vertex. 

\subsection{O-S description of $HF^+$ of negative definite plumbed 3-manifolds with at most one bad vertex}\label{subsection: O-S H plus}
In an early paper on Heegaard Floer homology (see \cite{MR1988284}), Ozsváth and Szabó provide a combinatorial description of the Heegaard Floer homology of 3-manifolds plumbed along negative definite forests with at most one bad vertex. We briefly review their description. 

Given a plumbing presentation $\Gamma$ of a 3-manifold $Y$, there is a naturally associated cobordism from $S^3$ to $Y$ via attaching two handles to $S^3\times [0,1]$ according to the plumbing graph $\Gamma$. One can turn this cobordism around and use the fact that there is an orientation preserving diffeomorphism from $-S^3$ to $S^3$ to yield a cobordism $W_{\Gamma}:-Y\to S^3$. For each $\spinc$ structure $\mathfrak{s}$ on $W_{\Gamma}$, we get a $U$-equivariant map
\begin{align*}
    F^+_{W_{\Gamma},\mathfrak{s}}:HF^+(-Y,\mathfrak{s}|_{Y})\to HF^+(S^3)
\end{align*}
It is easy to see that the $\spinc$ structures on $W_{\Gamma}$ correspond in a direct way to $\spinc$ structures on the plumbed 4-manifold $X(\Gamma)$ since $W_{\Gamma}$ is diffeomorphic to $X(\Gamma) - D^4$. Because of this we will work with $\spinc$ structures on $X(\Gamma)$ rather than on $W_{\Gamma}$. 

Now by the basic facts about $\spinc$ structures and characteristic vectors described in the previous section and the fact that $HF^+(S^3)\cong \mathcal{T}^+$ as a graded $\F[U]$-module, we can define a map $T^+:HF^+(-Y)\to \Map(\Char(X(\Gamma),\mathcal{T}^+)$ via the formula:
\begin{align*}
    T^+(\xi)(c_{1}(\mathfrak{s})) = F^+_{W_{\Gamma},\mathfrak{s}}(\xi)
\end{align*}
Here $\Map(\Char(X(\Gamma),\mathcal{T}^+)$ simply denotes the set of functions from $\Char(X(\Gamma))$ to $\mathcal{T}^+$. 

Let $\mathrm{H}^+(\Gamma)\subset \Map(\Char(X(\Gamma),\mathcal{T}^+)$ be the functions $\phi$ of finite support which satisfy the following adjunction relations: For each $k\in \Char(X(\Gamma))$ and $v_{i}\in \mathcal{V}(\Gamma)$, let $2n_{i} = k([v_{i}])+([v_{i}], [v_{i}])$. Then,
\begin{enumerate}
    \item if $n_{i}\geq 0$, we require $U^{n_{i}}\phi(k+2PDj_{*}[v_{i}]) = \phi(k)$
    \item if $n_{i}<0$, we require $U^{-n_{i}}\phi(k)= \phi(k+2PDj_{*}[v_{i}])$
\end{enumerate}
The set $\mathrm{H}^+(\Gamma)$ naturally inherits an $\F[U]$-module structure from $\mathcal{T}^+$. One can also introduce a grading on $\mathrm{H}^+(\Gamma)$ by defining $\phi\in \mathrm{H}^+(\Gamma)$ to be a homogeneous element of degree $d$ if $\phi(k)\in \mathcal{T}^+$ is a homogeneous element of degree $d+\dfrac{k^2+|\mathcal{V}(\Gamma)|}{4}$ for all $k\in \Char(X(\Gamma))$. Furthermore, we can decompose $\mathrm{H}^+(\Gamma)$ into a direct sum over $\spinc$ structures of $Y$ by defining $\mathrm{H}^+(\Gamma, [k])$ to be the elements of $\mathrm{H}^+(\Gamma)$ which are supported on the \textit{set} $[k]$. Recall $[k]$ denotes both a $\spinc$ structure on $Y$ as well as a subset of $\Char(X(\Gamma))$ (see Notation \ref{notn: characterstic equiv relation}).  

\begin{rem}
In \cite{MR1988284}, $\mathrm{H}^+(\Gamma)$ is instead denoted by $\Hplus(\Gamma)$. We have changed the notation in this paper to $\mathrm{H}^+(\Gamma)$ to avoid confusion with lattice cohomology which is denoted by $\Hbb^*(\Gamma)$.
\end{rem}

\newpage
The main result (Theorem 1.2) in \cite{MR1988284} states that if $\Gamma$ is a negative definite plumbing with at most one bad vertex, then $T^+:HF^{+}(-Y(\Gamma), [k])\rightarrow \mathrm{H}^+(\Gamma, [k])$ is an isomorphism of graded $\F[U]$-modules for all $\spinc$ structures $[k]$ on $Y(\Gamma)$. Moreover, $\mathrm{H}^+(\Gamma, [k])$ can be computed combinatorially from the data encoded by the plumbing graph. Therefore, this result enables one to compute $HF^{+}(-Y(\Gamma), [k])$ without having to count holomorphic disks. In particular, Ozsváth and Szabó provide a relatively simple algorithm to compute $\ker(U)\subset \mathrm{H}^+(\Gamma, [k])$. 

\subsection{Némethi's graded roots and lattice cohomology}
Building upon the work of Oszváth and Szabó, Némethi in \cite{MR2140997} provides an algorithm to compute the entire $\F[U]$-module $\mathrm{H}^+$ for almost rational plumbings by adapting methods of computation sequences used in the study of normal surface singularities. On the way to computing $\mathrm{H}^+$, Némethi's algorithm first computes an intermediate object called a graded root whose definition we review below (see Definitions \ref{defns: graded root}). For now, we will just mention that a graded root is weighted graph associated to $Y(\Gamma)$ from which one can easily calculate $\mathrm{H}^+$ and therefore $HF^+$. Furthermore, by using the language of graded roots, Némethi shows that \cite[Theorem 1.2]{MR1988284} holds for almost rational plumbed manifolds, a strictly larger class of plumbed 3-manifolds than the class of negative definite trees with at most one bad vertex.
\begin{rem}
We say trees in the previous sentence because strictly speaking almost rational plumbings are typically assumed to be connected. This assumption, however, is not important. The same methods apply to yield the isomorphism if you drop the connectedness assumption in the definition of almost rational. 
\end{rem}

Motivated by questions involving complex analytic normal surface singularities and the Seiberg-Witten invariant, Némethi further generalizes his work on negative definite plumbed 3-manifolds by introducing the broader framework of lattice cohomology in \cite{MR2426357}. Lattice cohomology assigns to any negative definite plumbed 3-manifold and $\spinc$ structure a graded $\F[U]$-module, which we denote $\Hbb^*$. 

Némethi's original definition provides two different, but equivalent, realizations of lattice cohomology. One realization is constructed by first decomposing Euclidean space $\R^s = \R\otimes H_{2}(X(\Gamma);\Z)$ into cubes using the $\Z$-lattice $H_{2}(X(\Gamma);\Z)$ with basis $[v_{1}], \ldots, [v_{s}]$. Then, one considers the usual cellular cohomology of $\R^s$, except with the differential modified by a set of weight functions which encode information about the intersection form of $X(\Gamma)$. The other realization is built by taking the cellular cohomology of certain sublevel sets of these weight functions on cubes. 

Lattice cohomology also comes equipped with an extra $\Z$-grading. Namely $\Hbb^*$ decomposes as $\Hbb^* = \bigoplus\limits_{q=0}^{\infty}\Hbb^q$ such that each $\Hbb^q$ is itself a $\Z$-graded $\F[U]$-module. In particular, together with his work in \cite{MR2140997}, Némethi shows that for a negative definite almost rational plumbed 3-manifold, $Y(\Gamma)$, and $\mathfrak{s}\in \spinc(Y(\Gamma))$, $\Hbb^0(Y(\Gamma), \mathfrak{s})$ is isomorphic to $HF^+(-Y(\Gamma), \mathfrak{s})$ as graded $\F[U]$-modules (up to an overall grading shift), and, moreover, $\Hbb^q(Y,\mathfrak{s})\cong 0$ for $q\geq 1$. In general, however, it is not the case that for arbitrary negative definite plumbed 3-manifolds $\Hbb^q\cong 0$ for all $q\geq 1$. For example, Némethi shows the existence of a negative definite plumbed rational homology sphere with non-trivial $\Hbb^1$ (see \cite[Example 4.4.1]{MR2426357}). Of course though, this plumbing is not almost rational.

\subsection{Modified formulation of lattice cohomology}
In this subsection we construct a modified version of lattice cohomology in order to deal with negative semi-definite plumbings rather than just negative definite plumbings. Before defining this modified version, it is important to point out that subsequent to Némethi's original definition of lattice cohomology several other variants/generalizations have been defined which apply to much more general plumbings including negative semi-definite plumbings (see for example: \cite{MR3031647}, \cite{MR2815139}, \cite{MR3317341}). The modified construction we provide is very similar to these formulations in many regards; the main difference is that we handle degenerate plumbings by passing to a certain quotient lattice. As in \cite{MR2426357}, we begin by giving the constructions in general terms, without reference to plumbings.

\subsubsection{Construction 1}\label{subsection: construction 1}
Let $A$ be a free finitely generated $\Z$-module with a specified ordered basis $(e_{1},\ldots, e_{n})$. Let $\bar{A}$ be a quotient of $A$ with the property that $\bar{A}$ is itself a free finitely generated $\Z$-module. Given $a\in A$, we write $\bar{a}$ for the corresponding element of $\bar{A}$.

We define a chain complex as follows. For each $0\leq q\leq n$, let $C_{q}$ be the free $\F$-module generated by the set $\mathcal{Q}_{q}=\bar{A}\times \{I \subseteq \{1,\ldots n\} \ |\  |I|=q\}$. Because later we will want to think of these generators as cubes in a cube complex (see \hyperref[subsection:construciton 2]{Construction 2}), we denote the generator of $C_{q}$ and the element of $\mathcal{Q}_{q}$ corresponding to $(\bar{a}, I)$ by $\square(\bar{a}, I)$. We define a differential $\partial: C_{q}\rightarrow C_{q-1}$ by the following formula, 
\begin{align*}
    \partial \square(\bar{a}, I) = \sum\limits_{i\in I}\Big[\square(\bar{a}, I-\{i\}) + \square(\bar{a}+\bar{e}_{i}, I-\{i\})\Big]
\end{align*}

\begin{rem}
Intuitively, it may be helpful to think of this differential as a cellular boundary map on cubes. We make this point of view precise in \hyperref[subsection:construciton 2]{Construction 2}. 
\end{rem}

\begin{prop}
$\partial^2 = 0$
\end{prop}
\begin{proof}

\begin{align*}
    \partial^2 \square(\bar{a}, I) &= \sum\limits_{i\in I}\sum\limits_{j\in I-\{i\}}\Big[\square(\bar{a}, I-\{i,j\}) +\square(\bar{a} +\bar{e}_{j}, I-\{i,j\})\Big]\\ 
    &+\sum\limits_{i\in I}\sum\limits_{j\in I-\{i\}}\Big[\square(\bar{a}+\bar{e}_{i}, I-\{i,j\}) +\square(\bar{a}+\bar{e}_{i}+ \bar{e}_{j}, I-\{i,j\})\Big]
\end{align*}
Now observe that the terms of the form $\square(\bar{a}, I-\{i,j\})$ cancel in pairs as $i$ and $j$ vary, as do the terms of the form $\square(\bar{a}+\bar{e}_{i}+ \bar{e}_{j}, I-\{i,j\})$. Finally, the cross terms also cancel. Therefore, $\partial^2 = 0$. 
\end{proof}

\begin{rem}
If one wanted to work over the coefficient ring $\Z$ instead of $\F$, then signs could be introduced as follows: Given a non-empty subset $I$ of $\{1,\ldots, n\}$ with $|I| = q$, let $g_{I}: I\rightarrow \{1,\ldots, q\}$ be the unique order preserving bijection. Define the differential via the formula:
\begin{align*}
    \partial \square(\bar{a}, I) = \sum\limits_{i\in I}(-1)^{g_{I}(i)}\Big[\square(\bar{a}, I-\{i\}) - \square(\bar{a}+\bar{e}_{i}, I-\{i\})\Big]
\end{align*}
One can check that we still have $\partial^2=0$. For the purposes of this paper, we will stick with the coefficient ring $\F$. 
\end{rem}

For each $0\leq q\leq s$, define $\mathcal{F}^q = \Hom_{\F}(C_{q},\mathcal{T}^+)$. We endow $\mathcal{F}^q$ with a $\F[U]$-module structure by the following formula: $(U^n\cdot\phi)(\square_{q}) = U^n\phi(\square_{q})$ for all $\square_{q}\in \mathcal{Q}_{q}$. Our goal now is to define a differential, $\delta_{w}$, on our cochain modules $\mathcal{F}^q$ by modifying the usual coboundary map by a set of weight functions $w$.

\begin{defn}[See {\cite[3.1.4. Definition]{MR2426357}}]
A set of functions $w_{q}: \mathcal{Q}_{q}\rightarrow \Z$, $0\leq q\leq n$ is called a set of compatible weight functions if the following hold:
\begin{enumerate}
    \item For any integer $k\in \Z$, the set $w_{0}^{-1}((-\infty, k])$ is finite. 
    \item For any $\square(\bar{a},I) \in\mathcal{Q}_{q}$ and any $i\in I$, $w_{q}(\square(\bar{a},I))\geq w_{q-1}(\square(\bar{a},I-\{i\}))$ and $w_{q}(\square(\bar{a},I))\geq w_{q-1}(\square(\bar{a}+\bar{e}_{i},I-\{i\}))$.
\end{enumerate}
Fix a set of compatible weight functions $w$ (we drop the subscript for simplicity). By using $w$, we are able to define a $\Z$-grading on our cochain modules $\mathcal{F}^q$. Specifically, we say that $\phi\in \mathcal{F}^q$ is homogeneous of degree $d\in \Z$ if $\phi(\square_{q})$ is a homogeneous element of $\mathcal{T}^+$ of degree $d-2w(\square_{q})$ whenever $\phi(\square_{q})\neq 0$.
 
\subsubsection{The differential.}
Mimicking the formula for the differential given in \cite[3.1.4 Definition]{MR2426357}, we define $\delta_{w}:\mathcal{F}^q\rightarrow \mathcal{F}^{q+1}$ as follows:
\begin{itemize}
    \item Let $\square_{q+1}\in \mathcal{Q}_{q+1}$ and write $\partial \square_{q+1} = \sum\limits_{k}\square_{q}^k$. 
    \item Given $\phi\in \mathcal{F}^q$, let 
    \begin{align*}
        (\delta_{w}\phi)(\square_{q+1}) = \sum\limits_{k}U^{w(\square_{q+1})-w(\square_{q}^{k})}\phi(\square_{q}^k)
    \end{align*}
\end{itemize}

\begin{prop}
$\delta_{w}^2 = 0$
\end{prop}
\begin{proof}
This follows directly from the definition and the fact that $\partial^2 = 0$. 
\end{proof}
\begin{defn}
The homology of the cochain complex $(\mathcal{F}^*, \delta_{w})$ is called the \textit{lattice cohomology} of the triple $(\bar{A}, (e_{1}, \ldots, e_{n}), w)$ and is denoted by $\Hbb^*(\bar{A}, (e_{1}, \ldots, e_{n}), w)$. 
\end{defn}

\begin{rems}
\hspace{1em}
\begin{enumerate}
    \item For each $q$, the $\Z$-grading on $\mathcal{F}^q$ induces a $\Z$-grading on $\Hbb^q$. Therefore, $\Hbb^q$ is a $\Z$-graded $\F[U]$-module.
    \item If $\bar{A}=A$, then we recover the usual lattice cohomology defined by Némethi in \cite{MR2426357}.
\end{enumerate}

\end{rems}
\end{defn}

\subsubsection{Construction 2}\label{subsection:construciton 2} We now give a more geometric, but equivalent formulation of the lattice cohomology theory we defined in \hyperref[subsection: construction 1]{Construction 1}. This is analogous to \cite[3.1.11 Definitions]{MR2426357}.

First, we give a geometric realization of the chain complex $C_{q}$. For each $1\leq q\leq s$, let $\textbf{c}_{q}$ be denote the $q$-dimensional cube $[0,1]^q$ oriented in the standard way. Additionally, let $\textbf{c}_{0}$ be a fixed $0$-dimensional cube (i.e. point) oriented positively. To each $\square(\bar{a},I)\in \mathcal{Q}_{q}$ we associate a distinct copy of $\textbf{c}_{q}$. By an abuse of notation, from now on we will regard each $\square(\bar{a},I)\in \mathcal{Q}_{q}$ as both a distinct copy of $\textbf{c}_{q}$ and a generator of $C_{q}$ depending on which point of view is more convenient in a given context. 

We now construct a cube complex $\mathcal{C}$ whose $q$-dimensional cubes are precisely the elements of $\mathcal{Q}_{q}$ with attaching maps defined as follows:
\begin{itemize}
    \item First, we prescribe a method for identifying each $(q-1)$-dimensional face of $\textbf{c}_{q}$ with $\textbf{c}_{q-1}$. Let $\{x_{j}\}_{j=1}^{q}$ be the standard coordinate functions on $\textbf{c}_{q} = [0,1]^q$. Each $(q-1)$-dimensional face of $\textbf{c}_{q}$ is defined by an equation $x_{i}=\epsilon$ for some $\epsilon\in\{0,1\}$. Denote this face by $f_{i, \epsilon}$. For $q\geq 2$, we identify $f_{i,\epsilon}$ with $\textbf{c}_{q-1}$ via the map $(x_{1},\ldots, x_{q})\mapsto (x_{1}, \ldots, \hat{x}_{i}, \ldots, x_{q})$. For $q=1$, we send the point $f_{i, \epsilon}$ to the point $\textbf{c}_{0}$. 
    
    \item Given $\square(\bar{a}, I)\in \mathcal{Q}_{q}$, the face $f_{i, \epsilon}$ of $\square(\bar{a}, I)$ gets glued to the cube $\square(\bar{a}+\epsilon \bar{e_{i}}, I-\{i\})$ via the map defined in the first bullet point. 
\end{itemize}
By construction the $q$-dimensional cellular chain group of the cube complex $\mathcal{C}$ is equal to $C_{q}$ and the cellular boundary map is equal to the differential $\partial: C_{q}\rightarrow C_{q-1}$ defined in \hyperref[subsection: construction 1]{Construction 1}.

Again, fix a set of compatible weight functions $w$. For every integer $n\geq 1$, let $S_{n}$ be the subcomplex of $\mathcal{C}$ consisting of all cubes $\square$ such that $w(\square_{q})\leq n$ where $q$ ranges over all dimensions. Let $m_{w} = \min\{w(\square_{q})\ |\ \square_{q}\in\mathcal{Q}_{q},\ 0\leq q\leq n\}$. Define 
\begin{align*}
    \mathbb{S}^q(\bar{A}, (e_{1},\ldots, e_{n}), w) = \bigoplus\limits_{n\geq m_{w}} H^q(S_{n};\F)
\end{align*}
where $H^q$ denotes the $qth$-cellular cohomology. For each fixed $q$, we give $\mathbb{S}^q(\bar{A}, (e_{1},\ldots, e_{n}), w)$ the structure of an $\F[U]$-module by defining the $U$ action to be the restriction map 
\begin{align*}
    U: H^q(S_{n+1};\Z)\rightarrow H^q(S_{n};\Z)
\end{align*}
We additionally put a $\Z$-grading on $\mathbb{S}^q(\bar{A}, (e_{1},\ldots, e_{n}), w)$ by declaring the elements of $H^q(S_{n},\Z)$ to be homogeneous of degree $2n$.

\begin{prop}
As graded $\F[U]$-modules, $\Hbb^*(\bar{A}, (e_{1},\ldots, e_{n}), w)\cong \mathbb{S}^*(\bar{A}, (e_{1},\ldots, e_{n}), w)$.
\end{prop}
\begin{proof}
This is proved in exactly the same way as \cite[3.1.12 Theorem (a)]{MR2426357}.
\end{proof}

\begin{notn}
From now on we will denote lattice cohomology by $\Hbb^*$ regardless of which construction we are using.  
\end{notn}

\subsubsection{Lattice cohomology associated to negative semi-definite plumbings}
Fix a negative semi-definite plumbing graph $\Gamma$ and let $k$ be a characteristic vector of $X(\Gamma)$ such that $[k]$ is a torsion $\spinc$ structure on $Y(\Gamma)$. 

We now show how to associate a lattice cohomology module to the pair $(\Gamma, k)$. Let $L = H_{2}(X(\Gamma);\Z)$ and $\bar{L} = H_{2}(X(\Gamma);\Z)/\ker(j_{*})$. By the long exact sequence in homology, $\bar{L}$ is isomorphic to a submodule of the free finitely generated $\Z$-module $H_{2}(X,Y;\Z)$ and therefore is itself free and finitely generated. As in section \ref{section: plumbings}, let $s =\rank(H_{2}(X;\Z))$. Also, let $\sigma = s-b_{1}(Y)$. With this notation, we have that $\bar{L}\cong \Z^{\sigma}$. Furthermore, after choosing an ordering on the vertices, the plumbing gives us an ordered basis $([v_{1}],\ldots, [v_{s}])$ of $L$. 

We now have almost all the data we need in order to get lattice cohomology. It remains to define a set of weight functions. To do this, we rely on our choice of characteristic vector $k$.  

\subsubsection{Weight functions}
Let $\chi_{k}:L\rightarrow \Z$ be the function defined by $\chi_{k}(x) = -\dfrac{k(x)+(x,x)}{2}$.

\begin{prop}
$\chi_{k}:L\rightarrow\Z$ descends to a well-defined function $\bar{\chi}_{k}:\bar{L}\rightarrow \Z$. 
\end{prop}

\begin{proof}
Since $[k]$ is assumed to be a torsion $\spinc$ structure on $Y$ there exists, by Remark \ref{rem:char Q}, some $z_{k}\in L\otimes \Q$ such that $k(x) = (z_{k}, x)$ for all $x\in L$. Now suppose $x\in L$ and $x'\in \ker(j_{*})$. Then,
\begin{align*}
\chi_{k}(x+x') = -\frac{k(x+x')+(x+x',x+x')}{2}&=\chi_{k}(x)-\frac{k(x')+2(x,x') + (x',x')}{2}\\
&=\chi_{k}(x)-\frac{1}{2}(z_{k}+2x+x',x')\\
&=\chi_{k}(x) - \frac{1}{2}PD[j_{*}(x')](z_{k}+2x+x')\\
&=\chi_{k}(x)
\end{align*}
\end{proof}

To make it easier to state some qualitative properties of $\bar{\chi}_{k}$, we now consider the extension of $\bar{\chi}_{k}$ by scalars to the function $\bar{\chi}^{\R}_{k}:\bar{L}\otimes \R\rightarrow \R$. Notice that the negative semi-definite intersection form $(\cdot, \cdot):L\times L\rightarrow \Z$ descends to a negative definite symmetric bilinear pairing on $\bar{L}$ which we denote by $(\cdot, \cdot)_{\bar{L}}$. Extending by scalars, we get a negative definite intersection form $(\cdot, \cdot)_{\bar{L}\otimes\R}:(\bar{L}\otimes\R)\times(\bar{L}\otimes\R)\rightarrow\R$. Therefore, we have 
\begin{align*}
\bar{\chi}^{\R}_{k}(\bar{x})= -\frac{k(x)+(\bar{x},\bar{x})_{\bar{L}\otimes\R}}{2}= -\frac{1}{2}(\bar{z_{k}}+\bar{x},\bar{x})_{\bar{L}\otimes \R}
\end{align*}
In particular, we see that $\bar{\chi}^{\R}_{k}$ is a positive definite quadratic form plus a linear shift. Putting these observations together yields the following proposition.
\begin{prop}\label{prop: chi level sets}
\hspace{1em}
\begin{enumerate}
\item $\bar{\chi}^{\R}_{k}$ is bounded below.
\item Let $\{\bar{x}_{1},\ldots, \bar{x}_{\sigma}\}$ be any $\R$-basis of $\bar{L}\otimes \R$. Identify $\bar{L}\otimes \R$ with $\R^{\sigma}$ via $\bar{L}\otimes\R = \bigoplus\limits_{j=1}^{\sigma}\R\bar{x}_{j}$. Then, the level sets of $\bar{\chi}_{k}^{\R}:\R^{\sigma}\rightarrow \R$ are $(\sigma-1)$-dimensional ellipsoids and the sublevel sets are $\sigma$-dimensional balls bounded by these ellipsoids. 
\end{enumerate}
\end{prop}

\begin{cor}\label{cor: sublevel sets finite}
$\bar{\chi}_{k}:\bar{L}\rightarrow \Z$ is bounded below and its sublevel sets are finite.
\end{cor}

\begin{defn}
Define $w_{q}: \mathcal{Q}_{q}\rightarrow \Z$ by 
\begin{align*}
    w(\square(\bar{l}, I)) = \max\{\bar{\chi}_{k}(\bar{x})\ |\  \bar{x}= \bar{l} + \sum\limits_{j\in J}\overline{[v_{j}]}, J\subseteq I\}
\end{align*}
Note, $w: \mathcal{Q}_{0}\rightarrow \Z$ is simply $\bar{\chi}_{k}$.
\end{defn}
By Corollary \ref{cor: sublevel sets finite}, $w$ is a valid set of weight functions.

\begin{defn}
Define $\Hbb^*(\Gamma, k) = \Hbb(\bar{L}, ([v_{1}], \ldots, [v_{s}]), w)$
\end{defn}

As in the case with negative definite plumbings, different choices of representatives for $[k]$ yield isomorphic lattice cohomology up to an overall grading shift. More specifically,

\begin{lem}[{See \cite[3.3.2 Lemma]{MR2140997}}]
If $k' = k +2PD[j_{*}(l)]$ for some $l\in L$, then $\Hbb^*(\Gamma, k)=\Hbb^*(\Gamma, k')[2\bar{\chi}_{k}(\bar{l})]$. 
\end{lem}

\begin{rem}
Némethi uses the opposite convention for grading shifts. Hence, \cite[3.3.2 Lemma]{MR2140997} is stated as: $\Hbb^*(\Gamma, k)=\Hbb^*(\Gamma, k')[-2\bar{\chi}_{k}(\bar{l})]$. 
\end{rem}

\newpage
\begin{exmps}\label{ex:S1S2}
\hfill
\begin{enumerate}
    \item(Compare \cite[Examples 3.11]{MR3317341}) 
Consider the plumbing graph $\Gamma$ consisting of a single $0$-framed vertex $v_{1}$. $\Gamma$ is a negative semi-definite plumbing whose corresponding plumbed 3-manifold, $Y(\Gamma)$, is diffeomorphic to $S^1\times S^2$. In this case, $\bar{L} = \langle \overline{[v_{1}]}\rangle = \{0\}$, $\mathcal{Q}_{0} = \{\square(0, \emptyset)\}$, and $\mathcal{Q}_{1}=\{\square(0, \{1\})\}$. Therefore, $\mathcal{F}^0\cong \mathcal{T}^+$ and $\mathcal{F}^{1}\cong \mathcal{T}^+$. Let $k = 0\in H_{2}(X(\Gamma);\Z)$. Then, $[k]$ is the unique torsion $\spinc$ structure on $Y(\Gamma)$ and $\bar{\chi}_{k} \equiv 0$. Therefore, the weight functions $w$ are identically zero. So in this very simple case, the lattice cohomology coboundary map $\delta_{w}$ is literally the dual of $\partial$. But, 
\begin{align*}
    \partial \square(0, \{1\}) &= -\square(0, \{\emptyset\})+\square(0+\overline{[v_{1}]}, \emptyset)\\
    &=-\square(0, \{\emptyset\}) + \square(0, \{\emptyset\})\\
    &=0
\end{align*}
Thus, $\delta_{w} = 0$. It follows that
\begin{align*}
    \mathbb{H}(\Gamma, k)\cong \mathbb{H}^0(\Gamma, k)\oplus \mathbb{H}^1(\Gamma, k)\cong \mathcal{T}^+\oplus \mathcal{T}^+
\end{align*}
In particular, up to the appropriate grading shifts, $\mathbb{H}^*(\Gamma, k)\cong HF^+(-(S^1\times S^2), [k])$. 

\item Even though our main focus is when $b_{1} = 1$, we think it is instructive to generalize the previous example. Specifically, let $\Gamma$ be the plumbing graph consisting of $s$ disjoint vertices $v_{1}, \ldots, v_{s}$ all with weight $0$ and no edges. Then, $Y(\Gamma) = \#_{s} S^1\times S^2$. Again, let $k=0\in H_{2}(X;\Z)$. One can check that the associated cube complex $\mathcal{C}$ is $T^s = \underbrace{S^1\times \cdots \times S^1}_{s\text{ times}}$. The singular cohomology ring of $T^s$ is $H^*(T^s;\F)\cong \Lambda(\F^s)$, where $\Lambda$ denotes the exterior algebra. Since in this case the weight functions are all identically zero, \hyperref[subsection:construciton 2]{Construction 2} tells us that $\mathbb{H}^*(\Gamma, k)\cong \Lambda(\F^s)\otimes \mathcal{T}^+$. In particular, up to the appropriate grading shifts, $\mathbb{H}^*(\Gamma, k)\cong HF^+(-T^s,[k])$. 
\end{enumerate}
\end{exmps}

\begin{rem}
In the above examples, there are of course many other plumbing descriptions of the same 3-manifolds. Therefore, it is worth noting that at this stage we have not yet shown that this modified version of lattice cohomology is independent of the plumbing description (i.e. that it is a topological invariant). However, these examples do suggest that up to grading shifts (at least for ``nice enough'' negative semi-definite plumbings) lattice cohomology agrees with $HF^+$, which of course is a topological invariant. We would like to point out though that in \cite{MR3317341} Oszváth, Stipsicz, and Szabó construct a spectral sequence relating a completed version of $HF^+$ to their version of lattice cohomology and show that these two objects coincide for plumbing trees of type 2. In particular, their isomorphism holds for negative semi-definite plumbings of type 2. 
\end{rem}

\newpage
\subsection{Graded roots associated to negative semi-definite plumbings}

\begin{defns}[{See \cite[3.2 Definitions]{MR2140997}}]\label{defns: graded root} 
\hspace{1em}
\begin{enumerate}
\item Let $R$ be an infinite tree with vertices $\mathcal{V}$ and edges $\mathcal{E}$. We denote by $[u,v]$ the edge with end-points $u$ and $v$. We say that $R$ is a graded root with grading $\chi: \mathcal{V}\rightarrow \Z$ if 
\begin{enumerate}
\item $\chi(u)-\chi(v)=\pm 1$ for any $[u,v]\in\mathcal{E}$
\item $\chi(u)>\min\{\chi(u), \chi(w)\}$ for any $[u,v], [u,w]\in \mathcal{E}, v\neq w$
\item $\chi$ is bounded below, $\chi^{-1}(k)$ is finite for any $k\in \Z$, and $\#\chi^{-1}(k) = 1$ if $k$ is sufficiently large.
\end{enumerate}
\item We say that $v\in \mathcal{V}$ is a local minimum point of the graded root $(R, \chi)$ if $\chi(v)<\chi(w)$ for any edge $[v,w]$.
\item If $(R,\chi)$ is a graded root, and $r\in \Z$, then we denote by $(R,\chi)[r]$ the same $R$ with the new grading $\chi[r](v) :=\chi(v)+r$. (This can be generalized for any $r\in \Q$ as well.)
\end{enumerate}
\end{defns}

\begin{exmp}
\end{exmp}

\begin{figure}[h]
\centering
\includegraphics[scale=.51]{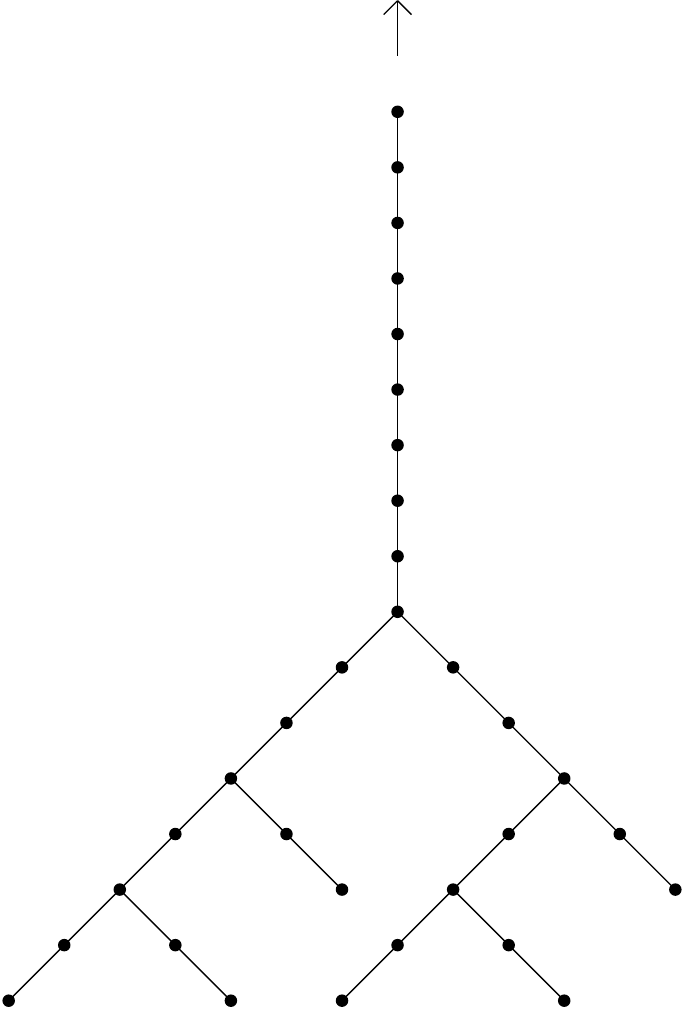}
\end{figure}

We now show how to associate a graded root to a pair $(\Gamma, k)$ where $\gamma$ is a negative semi-definite plumbing and $k$ is a characteristic vector of $X(\Gamma)$ such that $[k]$ is a torsion $\spinc$ structure on $Y(\Gamma)$. For each $n\in \Z$, let $\bar{L}_{k, \leq n}$ be the graph whose vertex set is $\mathcal{V}(\bar{L}_{k, \leq n})=\{\bar{x}\in \bar{L} : \bar{\chi}_{k}(\bar{x})\leq n\}$ and such that there is an edge between two vertices $\bar{x}_{1}, \bar{x}_{2}$ if and only if $\bar{x}_{1}-\bar{x}_{2} = \pm \overline{[v_{j}]}$ where the $v_{j}$ are as in subsection \ref{subsection: alg top prop}. Now let $\pi_{0}(\bar{L}_{k, \leq n})$ denote the set of connected components of the graph $\bar{L}_{k, \leq n}$. 

The graded root $(\bar{R}_{k}, \bar{\chi}_{k})$ associated to $\Gamma$ and $k$ is constructed as follows:
\begin{itemize}
    \item The vertex set is $\mathcal{V}(\bar{R}_{k})=\bigsqcup\limits_{n\in \Z}\pi_{0}(\bar{L}_{k,\leq n})$. By an abuse of notation, we denote the grading $\mathcal{V}(\bar{R}_{k})\rightarrow \Z$ by $\bar{\chi}_{k}$ where now $\bar{\chi}_{k}|_{\pi_{0}(\bar{L}_{k,\leq n})}= n$. 
    
    \item There is an edge between two vertices $v, v'\in \mathcal{V}(\bar{R}_{k})$, which correspond to connected components $C_{v}$ and $C_{v'}$, if and only if after possibly reordering $v$ and $v'$, we have $\bar{\chi}_{k}(v') = \bar{\chi}_{k}(v)+1$ and $C_{v}\subset C_{v'}$. 
\end{itemize}

\begin{rem}
When $\Gamma$ is negative definite, $(\bar{R}_{k}, \bar{\chi}_{k})$ is precisely the graded root, $(R_{k}, \chi_{k})$, defined by Némethi in \cite[Section 4]{MR2140997}.
\end{rem}

\begin{rem}
The graph $\bar{L}_{k, \leq n}$ is the 1-skeleton of the space $S_{n}$ considered above in \hyperref[subsection:construciton 2]{Construction 2} of lattice cohomology. In particular, we can think of $\pi_{0}(\bar{L}_{k, \leq n})$ equivalently as $\pi_{0}(S_{n})$.
\end{rem}

\begin{prop}[{See \cite[4.3 Proposition]{MR2140997}}]
$(\bar{R}_{k}, \bar{\chi}_{k})$ is a graded root. 
\end{prop}

\begin{proof}
This proof is essentially identical to the proof of \cite[4.3 Proposition]{MR2140997}. Condition (a) of Definition \ref{defns: graded root} (1) follows immediately from the construction of $(R_{k}, \bar{\chi}_{k})$. The proof of condition (b) is the same as in \cite[4.3 Proposition]{MR2140997}. The first two conditions of (c) follow from Corollary \ref{cor: sublevel sets finite}. The last condition of (c) follows the same argument as Némethi's proof, with mild modification. Essentially just replace the function $\chi_{k}$ in Némethi's proof with $\bar{\chi}_{k}$ and use that $\bar{\chi}_{k}$ has a (not necessarily unique) global minimum and that $(\cdot, \cdot)_{\bar{L}}$ is negative definite. 
\end{proof}

Again, as in the case with negative definite plumbings, the graded roots, $(\bar{R}_{k}, \bar{\chi}_{k})$ and $(\bar{R}_{k'}, \bar{\chi}_{k'})$ corresponding to two characteristic vectors $k$ and $k'$, which restrict to the same torsion $\spinc$ structure on $Y$, are equal up to an overall grading shift. More specifically, 

\begin{prop}[{See \cite[4.4 Proposition]{MR2140997}}]\label{prop: chi grade shift}
If $k' = k + 2PD[j_{*}(l)]$ for some $l\in L$ and $k\in \Char(X(\Gamma))$ with $[k]$ torsion, then
\begin{align*}
(\bar{R}_{k'}, \bar{\chi}_{k'}) = (\bar{R}_{k},\bar{\chi}_{k})[\bar{\chi}_{k}(\bar{l})]
\end{align*} 
\end{prop}

\subsection{The relationship between lattice cohomology, $\mathrm{H}^+$, and graded roots}\label{subsection: hf lattice relationship}
In subsection \ref{subsection: O-S H plus}, we recalled the definition of the $\F[U]$-module $\mathrm{H}^+(\Gamma, [k])$ introduced by Ozsváth and Szabó where $\Gamma$ is a negative definite plumbing and $[k]$ is a $\spinc$ structure on $Y(\Gamma)$. The same definition makes sense for negative semi-definite plumbings and $[k]$ torsion except that we adjust the grading as follows: we say $\phi\in \mathrm{H}^+(\Gamma, [k])$ is a homogeneous element of degree $d$ if for each $k'\in [k]$ with $\phi(k')\neq 0$, we have that $\phi(k')\in \mathcal{T}^+$ is a homogeneous element of degree 
\begin{align*}
    d+\dfrac{(k')^2+|\mathcal{V}(\Gamma)|-3b_{1}(Y(\Gamma))}{4}
\end{align*}

\begin{prop}\label{prop:graded root lattice iso}
As graded $\F[U]$-modules, 
\begin{align*}
    \mathrm{H}^+(\Gamma, [k])\cong \Hbb^{0}(\Gamma, k)\left[\frac{k^2+|V(\Gamma)|-3b_{1}(Y)}{4}\right]
\end{align*}
\end{prop}

\begin{proof}
The isomorphism is induced by the map $Z:\mathrm{H}^+(\Gamma, [k])\to \mathcal{F}^0$ defined by
\begin{align*}
    Z(\phi)(\square(\bar{l}, \emptyset))=\phi(k+2PDj_{*}(l))
\end{align*}
We leave the details to the reader. 
\end{proof}

As described in \cite{MR1988284} and \cite{Rustamov}, for calculation purposes it is convenient to consider the ``dual space'' of $\mathrm{H}^+(\Gamma, [k])$, which we denote by $\Kplus(\Gamma, [k])$. To recall their definition of $\Kplus(\Gamma, [k])$, first consider the set $\Z_{\geq 0}\times [k]$. Write elements $(m, k')\in \Z_{\geq 0}\times [k]$ as $U^m\otimes k'$. Define an equivalence relation $\sim$ on $\Z_{\geq 0}\times [k]$ in the following way: for each $k'\in [k]$ and $v_{i}\in \mathcal{V}(\Gamma)$, let $2n_{i} = k'([v_{i}])+([v_{i}], [v_{i}])$. Then, 
\begin{enumerate}
    \item if $n_{i}\geq 0$, we require $U^{n_{i}+m}\otimes (k' + 2PDj_{*}[v_{i}]) \sim U^{m}\otimes k'$\\
    \item if $n_{i}<0$, we require $U^{m}\otimes (k'+2PDj_{*}[v_{i}])\sim U^{m-n_{i}}\otimes k'$
\end{enumerate}
In other words, two elements $U^m\otimes k'$ and $U^{n}\otimes k''$ are equivalent if and only if there exists a finite sequence of elements $U^{m_{0}}\otimes k_{1}, \ldots, U^{m_{\ell}}\otimes k_{\ell}$ such that $U^{m_{0}}\otimes k_{1} = U^m\otimes k'$, $U^{m_{\ell}}\otimes k_{\ell} = U^{n}\otimes k''$ and each adjacent pair in the sequence is related by a relation of type (1) or (2) as given above. We call such a sequence a \textit{path} connecting $U^m\otimes k'$ and $U^{n}\otimes k''$.
\begin{rem}
In general, there are many different paths connecting a given pair of elements $U^m\otimes k'$ and $U^{n}\otimes k''$.
\end{rem}

Write the equivalence class containing $U^m\otimes k'$ as $\underline{U^m\otimes k'}$ and define $\Kplus(\Gamma, [k])$ to be the set of these equivalence classes. $\Kplus(\Gamma, [k])$ is the dual of $\mathrm{H}^+(\Gamma, [k])$ (or maybe more naturally $\mathrm{H}^+(\Gamma, [k])$ is the dual of $\Kplus(\Gamma, [k])$) in the following sense:
\begin{itemize}
    \item Define $\Kplus(\Gamma, [k])^*$ to be the set of finitely supported functions $\phi:\Kplus(\Gamma, [k])\to \mathcal{T}^+$ such that $\phi(\underline{U^{n+m}\otimes k'}) = U^{n}\phi(\underline{U^{m}\otimes k'})$ for all $n,m\geq 0$ and $k'\in [k]$. Endow $(\Kplus)^*$ with an $\F[U]$-module structure by inheriting that of $\mathcal{T}^+$. 

    \item Define a map $F: \mathrm{H}^+(\Gamma, [k])\to \Kplus(\Gamma, [k])^*$ by
    \begin{align*}
        F(\phi)(\underline{U^m\otimes k'})=U^m\phi(k')
    \end{align*}
    It is straightforward to check that $F$ is a well-defined $\F[U]$-module isomorphism. 
\end{itemize}

We can put more structure on $\Kplus(\Gamma, [k])$ by thinking of it as a graph. Specifically, define $g\Kplus(\Gamma, [k])$ to be the graph whose vertices are the elements of $\Kplus(\Gamma, [k])$ and such that there is an edge between to vertices $\underline{U^m\otimes k'}$ and $\underline{U^{n}\otimes k''}$ if and only if either $\underline{U^{m+1}\otimes k'}=\underline{U^{n}\otimes k''}$ or $\underline{U^{m}\otimes k'} = \underline{U^{n+1}\otimes k''}$. 

\begin{prop}\label{prop: H dual and graded root}
As graphs, $g\Kplus(\Gamma, [k])$ is isomorphic to the graded root $(\bar{R}_{k},\bar{\chi}_{k})$.
\end{prop}

\begin{proof}
This proof is essentially the same as Némethi's proof of \cite[Proposition 4.7]{MR2140997}. For completeness, we provide the details here.

By definition each element $k'\in [k]$ can be written as $k' = k+2PD[j_{*}(l)]$ for some $l\in L$. Let $\bar{l}_{k'} := \bar{l}\in \bar{L}$. Define a map $p:\Kplus(\Gamma, [k])\to \mathcal{V}(\bar{R}_{k})$ as follows:
\begin{align*}
    p(\underline{U^{m}\otimes k'}) = \text{ the connected component of }\bar{L}_{k, \leq \bar{\chi}_{k}(\bar{l}_{k'})+m}\text{ containing }\bar{l}_{k'}
\end{align*}
To show that $p$ is well-defined, let $2n_{i}=k'([v_{i}]) + ([v_{i}], [v_{i}])$. Suppose first that $n_{i}\geq 0$ so that we have $U^{n_{i}+m}\otimes (k' + 2PDj_{*}[v_{i}]) \sim U^m\otimes k'$. Let $k'' = k' + 2PDj_{*}[v_{i}]$. Then, $\bar{l}_{k''} = \bar{l}_{k'}+\overline{[v_{i}]}$. Thus,
\begin{align*}
    \bar{\chi}_{k}(\bar{l}_{k''}) + n_{i}+m &= \bar{\chi}_{k}(\bar{l}_{k'}) +\bar{\chi}_{k'}(\overline{[v_{i}]})+n_{i}+m\\
    &=\bar{\chi}_{k}(\bar{l}_{k'}) - n_{i}+n_{i}+m\\
    &=\bar{\chi}_{k}(\bar{l}_{k'})+m
\end{align*}
Therefore, $\bar{L}_{k, \leq \bar{\chi}_{k}(\bar{l}_{k'})+m} = \bar{L}_{k, \leq \bar{\chi}_{k}(\bar{l}_{k''})+n_{i}+m}$ and $\bar{l}_{k'}$ and $\bar{l}_{k''}$ are in the same connected component since they differ by $\overline{[v_{i}]}$. The case when $n_{i}<0$ is similar. This establishes that $p$ is well-defined.

Next we define a map $q:\mathcal{V}(\bar{R}_{k})\to\Kplus(\Gamma, [k])$ which we will show is the inverse of $p$. Suppose $v\in \mathcal{V}(\bar{R}_{k})$. Let $C_{v}$ be the corresponding connected component in $\bar{L}_{k,\leq \bar{\chi}_{k}(v)}$ and let $\bar{l}_{v}$ be some element in $\bar{L}\cap C_{v}$. Define 
\begin{align*}
    q(v) = \underline{U^{\bar{\chi}_{k}(v)-\bar{\chi}_{k}(\bar{l}_{v})}\otimes (k+2PDj_{*}(l_{v}))} 
\end{align*}
\newpage
To show $q$ is well-defined, suppose $\bar{l}'$ is some other element in $\bar{L}\cap C_{v}$. It suffices to consider the case that $\bar{l}' = \bar{l}_{v} + \overline{[v_{i}]}$ for some $i$. First note, 
\begin{align*}
\bar{\chi}_{k}(\bar{l}') = \bar{\chi}_{k}(\bar{l}_{v} + \overline{[v_{i}]})= \bar{\chi}_{k}(\bar{l}_{v}) + \bar{\chi}_{k}(\overline{[v_{i}]}) - ([v_{i}], l_{v})
\end{align*}
Also, 
\begin{align*}
    (k + 2PDj_{*}(l_{v}))(v_{i}) + (v_{i}, v_{i}) &= k(v_{i}) + (v_{i}, v_{i}) + 2(v_{i},l_{v})\\
    &=-2[\bar{\chi}_{k}(\overline{[v_{i}]}) - (v_{i}, l_{v})]
\end{align*}
Hence, if $-[\bar{\chi}_{k}(\overline{[v_{i}]}) - (v_{i}, l_{v})]\geq 0$, then
\begin{align*}
    U^{\bar{\chi}_{k}(v) - \bar{\chi}_{k}(\bar{l}_{v})}\otimes (k + 2PDj_{*}(l_{v}))&\sim U^{\bar{\chi}_{k}(v) - \bar{\chi}_{k}(\bar{l}_{v})-\bar{\chi}_{k}(\overline{[v_{i}]})+(v_{i}, l_{v})}\otimes (k + 2PDj_{*}(l_{v}+[v_{i}]))\\
    &=U^{\bar{\chi}_{k}(v) - \bar{\chi}_{k}(\bar{l}')}\otimes (k + 2PDj_{*}(l'))
\end{align*}
Similarly, if $-[\bar{\chi}_{k}(\overline{[v_{i}]}) - (v_{i}, l_{v})]< 0$, then
\begin{align*}
U^{\bar{\chi}_{k}(v)-\bar{\chi}_{k}(\bar{l}')}\otimes (k+2PDj_{*}(l'))&\sim U^{\bar{\chi}_{k}(v)-\bar{\chi}_{k}(\bar{l}')+\bar{\chi}_{k}(\overline{[v_{i}]}) - (v_{i}, l_{v})}\otimes (k+2PDj_{*}(l_{v}))\\
&=U^{\bar{\chi}_{k}(v) -\bar{\chi}_{k}(\bar{l}_{v})}\otimes (k+2PDj_{*}(l_{v}))
\end{align*}
Therefore, $q$ is well-defined. 

Now consider $qp(\underline{U^m\otimes k'})$ where $k' = k+2PDj_{*}(\bar{l}_{k})$. Let $v=p(\underline{U^m\otimes k'})$ and $C_{v}$ be the connected component of $\bar{L}_{k,\leq \bar{\chi}_{k}(\bar{l}_{k'})+m}$ containing $\bar{l}_{k'}$. Then, by definition
\begin{align*}
    q(v) &= \underline{U^{\bar{\chi}_{k}(\bar{l}_{k'}) + m - \bar{\chi}_{k}(\bar{l}_{k'})}\otimes k + 2PDj_{*}(\bar{l}_{k'})}\\
    &= \underline{U^m\otimes k'}
\end{align*}
Hence, $qp= Id$. The other direction, i.e. that $pq = Id$, is tautological. Therefore, $p$ is a bijection. To see that $p$ takes edges to edges bijectively, let $v_{1} = p(\underline{U^{m}\otimes k'})$ and $v_{2} = p(\underline{U^{m+1}\otimes k'})$. It follows directly from the definition that $C_{v_{1}}\subset C_{v_{2}}$ and $\bar{\chi}_{k}(v_{2}) -\bar{\chi}_{k}(v_{1}) = 1$. 
\end{proof}

\begin{rem}
It is useful to point out that under the isomorphism $p$ constructed in the above proof, we have that 
\begin{align*}
    \gr(p(\underline{U^m\otimes k'})) = m-\frac{(k')^2-k^2}{8}
\end{align*}
\end{rem}

\subsection{A quick review of Rustamov's results on negative semi-definite plumbings with $\mathbf{b_{1} =1}$}\label{subsection: rustamov results}

In \cite{Rustamov}, Rustamov generalizes the setting in which the isomorphism $T^+$, described in Subsection \ref{subsection: O-S H plus}, holds. In particular, Rustamov proves the following theorem:
\begin{thm}[{See \cite[Theorem 1.2]{Rustamov}}]\label{thm: rustamov}
Let $\Gamma$ be a negative semi-definite plumbing with at most one bad vertex and with $b_{1}(Y(\Gamma))=1$. Further, let $[k]$ be a torsion $\spinc$ structure. Then, 
\begin{enumerate}
    \item $T^+: HF_{odd}^+(-Y(\Gamma), [k])\to \mathrm{H}^+(\Gamma, [k])$ is an isomorphism of graded $\F[U]$-modules. \\
    \item $HF^+_{even}(-Y(\Gamma),[k])\cong \mathcal{T}^+_{d}$ where $d = d_{-1/2}(-Y(\Gamma), [k])$.
\end{enumerate}
Here $HF_{odd}^+(-Y(\Gamma), [k])$ and $HF_{even}^+(-Y(\Gamma), [k])$ refer to the submodules generated by elements of $HF^+(-Y(\Gamma), [k])$ of degrees congruent to $1/2\bmod 2$ and $-1/2\bmod 2$ respectively. 
\end{thm}

\newpage
Combining Rustamov's result with the observations of the previous subsection, we get:
\begin{cor}\label{cor: HF and lattice identification}
With $\Gamma$ as above,  $HF_{odd}^+(-Y(\Gamma), [k])\cong \Hbb^{0}(\Gamma, k)\left[\frac{k^2+|V(\Gamma)|-3}{4}\right]$ as graded $\F[U]$-modules. In particular, up to an overall grading shift, $\Hbb^{0}(\Gamma, k)$ is a topological invariant of $Y(\Gamma)$. 
\end{cor}

\begin{rem}
It is likely possible that one can prove $\Hbb^{0}(\Gamma, k)\left[\frac{k^2+|V(\Gamma)|-3}{4}\right]$ is a topological invariant without appealing to Heegaard Floer homology, by showing invariance under Neumann moves as in the proof of \cite[Proposition 4.6]{MR2140997}.
\end{rem}
\newpage

\section{Calculation method}\label{section:computation method}
Throughout this section, fix a negative semi-definite plumbing $\Gamma$ with at most one bad vertex and such that $b_{1}(Y(\Gamma)) = 1$. Let $[k]$ be a self-conjugate $\spinc$ structure on $Y(\Gamma)$. In other words, $[k]=[-k]$ or, equivalently, $k = PD[j_{*}(l)]$ for some $l \in L$. Note that by identifying $\bar{l}$ with $k$, we can think of $k$ as an element of $\bar{L}$.

\subsection{Involutions on lattice cohomology and Heegaard Floer homology}\label{subsection: involutions}
 As in \cite[Section 2]{MR4021102}, define $J_{0}: \bar{L}\to \bar{L}$ by $J_{0}(\bar{x}) = -\bar{x}-\bar{l}$. Clearly, $J_{0}^2 = Id$. We can extend $J_{0}$ to a cubical involution on the cube complex $\mathcal{C}$ considered in \hyperref[subsection:construciton 2]{Construction 2} of lattice cohomology via the formula, 
\begin{align*}
   J_{0}\square(\bar{a}, I) = \square(J_{0}(\bar{a}+\sum\limits_{i\in I}\overline{[v_{i}]}), I) 
\end{align*}
It is straightforward to check that $J_{0}$ is compatible with the gluing of the cells. Moreover, since $\bar{\chi}_{k}(J_{0}(\bar{x})) = \bar{\chi}_{k}(\bar{x})$ for all $\bar{x}\in \bar{L}$, $J_{0}$ maps the subcomplex $S_{n}$ of $\mathcal{C}$ to itself. Therefore, $J_{0}$ induces an involution on $H^q(S_{n};\Z)$ or each $n, q$ and hence on lattice cohomology. By an abuse of notation, we denote the involution on lattice cohomology again by $J_{0}$. In a similar manner, one could alternatively define $J_{0}$ by using \hyperref[subsection: construction 1]{Construction 1}, but we leave the details to the reader. 

Focusing our attention on the $0th$-level of lattice cohomology, we can think of the action of $J_{0}$ on $\Hbb^{0}$ from the dual perspective by realizing an involution on the associated graded root. More specifically, since $J_{0}$ acts continuously on $S_{n}$, $J_{0}$ also induces an involution on the connected components of $S_{n}$. Hence, $J_{0}$ induces an involution on the graded root $(\bar{R}_{k}, \bar{\chi}_{k})$. From another perspective, under the identification of $(\bar{R}_{k}, \bar{\chi}_{k})$ with  $g\Kplus(\Gamma, [k])$ given in Proposition \ref{prop: H dual and graded root}, the involution $J_{0}$ sends $\underline{U^m\otimes k'}$ to $\underline{U^m\otimes -k'}$. 

There is a fundamental difference between the action of $J_{0}$ on the graded root in the negative definite and negative semi-definite cases. Before describing this difference, we need to recall the following definition:

\begin{defn}[{See \cite[Definition 2.11]{MR4021102}}]\label{def: symmetric involution}
A \textit{symmetric graded root} is a graded root $(R, \chi)$ together with an involution $J: \mathcal{V}(R)\to \mathcal{V}(R)$ such that
\begin{itemize}
    \item $\chi(v) = \chi(Jv)$ for any vertex $v$
    \item $[v, w]$ is an edge in $R$ if an only if $[Jv, Jw]$ is an edge in $R$
    \item for every $r\in \Q$, there is at most one $J$ invariant vertex $v$ with $\chi(v) = r$
\end{itemize}
We call such a $J$ a \textit{symmetric involution}. 
\end{defn}

In \cite[Lemma 2.1]{MR3802260} (see also \cite[Section 2.1]{MR4021102}) it is shown that the graded root $(R_{k}, \chi_{k})$ of a negative definite almost rational plumbing with $[k]$ self-conjugate is symmetric and $J_{0}$ is a symmetric involution; in particular, this holds if the plumbing has at most one bad vertex. However, if the plumbing is negative semi-definite and has at most one bad vertex, then the proof of \cite[Lemma 2.1]{MR3802260} no longer works and, as we show in section \ref{subsection: comparison}, $J_{0}$ need not be a symmetric involution. 

The proof of \cite[Lemma 2.1]{MR3802260} uses the classical Lefshetz fixed-point theorem and relies crucially on the fact that for $\Gamma$ negative definite and almost rational, $\Hbb^q(\Gamma, k)=0$ for $q>0$. However, as we have seen in Examples \ref{ex:S1S2}, when $\Gamma$ is negative semi-definite, it is not necessarily true that $\Hbb^q(\Gamma, k)=0$ for $q>0$. Hence, the proof that $J_{0}$ is symmetric fails in this case. As we will demonstrate in section \ref{section: examples}, the possibility that $J_{0}$ is not symmetric has important implications on the involutive $d$ invariants and hence on properties regarding spin cobordism and $0$-surgery.

Despite the difference in behavior of the involution $J_{0}$ in the negative definite and negative semi-definite cases, the proof of \cite[Theorem 3.1]{MR4021102} still holds in the negative semi-definite setting to give an identification of $J_{0}$ on $\Hbb^{0}(\Gamma, k)$ with the involution $\iota_{*}$ on $HF^+(-Y(\Gamma), [k])$. More precisely, 
\begin{thm}[{See \cite[Theorem 3.1]{MR4021102}}]\label{thm: involution identification}
Let $\Gamma$ be a negative semi-definite plumbing with at most one bad vertex and such that $b_{1}(Y(\Gamma))=1$. If $[k]$ is a self-conjugate $\spinc$ structure, then under the isomorphism given in Corollary \ref{cor: HF and lattice identification} the maps $J_{0}$ and the restriction of $\iota_{*}$ to $HF^+_{odd}(-Y(\Gamma), [k])$ are identified. 
\end{thm}

The action of $\iota_{*}$ on the even part of $HF^+$ is less interesting. Since $HF_{even}^+(-Y(\Gamma), [k])\cong \mathcal{T}_{d}$ and $\iota_{*}$ is $U$-equivariant, the restriction of $\iota_{*}$ to the even part must be the identity. Moreover, if one knows $HF^+$ and $\iota_{*}$, then by using the mapping cone exact triangle in Proposition \ref{prop: HFI exact triangle}, one can completely determine $HFI^+$ as a graded $\F$-vector space. 

In the context of negative definite almost rational plumbings, Dai and Manolescu show that one can actually determine the entire $\F[U,Q]/(Q^2)$-module structure of $HFI^+$ just from knowing $J_{0}$ (see \cite[Sections 4-5]{MR4021102}). However, one encounters issues when trying to extrapolate their methods to the case of negative semi-definite plumbings with at most one bad vertex. The main difficulty is that in the negative definite almost rational case, $HF^+$ is supported in even gradings, whereas in the negative semi-definite case, $HF^+$ has gradings in both even and odd dimensions which allows for the possibility of a more complicated action of $\iota$ at the chain level. Despite this issue, for negative semi-definite plumbings with at most one bad vertex whose $HF^+$ and $\iota_{*}$ are sufficiently simple, it is still possible to compute much, if not all, of the $\F[U,Q]/(Q^2)$-module structure of $HFI^+$ as well as some of the involutive $d$ invariants just from the mapping cone exact triangle. We illustrate this via the examples in section \ref{section: examples}. 

\subsection{Computation of $HFI^+(-Y(\Gamma), [k])$ as a graded $\F$-vector space}\label{subsection: comp strategy} We summarize the strategy we use to compute $HFI^+(-Y(\Gamma), [k])$ as a graded graded $\F$-vector space in the following 3-step process and then elaborate on each individual step.
\begin{enumerate}
    \item Compute $HF^+(-Y(\Gamma), [k])$ using the methods from section \ref{section: homology of plumbings}.\label{step 1}
    
    \item Use Theorem \ref{thm: involution identification} to compute the involution 
    \begin{align*}
         \iota_{*}:HF^+(-Y(\Gamma), [k])\to HF^+(-Y(\Gamma), [k])
    \end{align*}
   
    \item Apply the exact triangle relating $HF^+$ and $HFI^+$ from Proposition \ref{prop: HFI exact triangle}.

\end{enumerate}
\vspace{1em}
\indent\underline{Step (1)}: 
To compute $HF^+(-Y(\Gamma), [k])$, the first and main step is to determine the set 
\begin{align*}
    \mathcal{L}(\Gamma, [k]): = \{x\in \Kplus(\Gamma, [k])\ |\ x\text{ has no representative of the form } U^{n}\otimes k' \text{ for }n>0\}
\end{align*}
It is easy to see that the elements of $\mathcal{L}(\Gamma, [k])$ correspond to the leaves of the graded root $(\bar{R}_{k}, \bar{\chi}_{k})$ under the isomorphism in Proposition \ref{prop: H dual and graded root}. Moreover, from the results in section \ref{subsection: rustamov results}, it follows that the leaves of $(\bar{R}_{k}, \bar{\chi}_{k})$ correspond to a basis of the $\mathbb{F}$-vector space: 
\begin{align*}
    \ker(U)\cap HF_{odd}^+(-Y(\Gamma), [k])
\end{align*}

In \cite[Section 3]{Rustamov}, Rustamov provides an algorithm to compute $\mathcal{L}(\Gamma, [k])$ which builds on the Ozsváth and Szabó algorithm in \cite[Section 3]{MR1988284} for negative definite plumbings. For our computations in section \ref{section: examples}, rather than use Rustamov's algorithm directly, we instead will use a simple criterion (see Proposition \ref{prop:algorithm} below) which characterizes the elements of $\mathcal{L}(\Gamma, [k])$.

To explain this criterion, first recall from section \ref{subsection: hf lattice relationship} that two elements $U^m\otimes k'$ and $U^n\otimes k''$ are equivalent (i.e. represent the same element of $\Kplus(\Gamma, [k])$) if and only if there is a path between them. In particular, every element of $\mathcal{L}(\Gamma, [k])$ is represented by an element of the form $U^{0}\otimes k'$ and every element of a path connecting $U^{0}\otimes k'$ to another representative must also have $0$ as the exponent on the $U$ term. Therefore, when discussing representatives or paths for elements in $\mathcal{L}(\Gamma, [k])$, we can drop the $U^0$ term and instead think of a representative as an element $k'\in [k]$ and a path as a sequence of vectors $k_{1}, \ldots, k_{j}\in [k]$. Furthermore, the relations defining such a path imply that for adjacent elements $k_{i}, k_{i+1}$ we have that $k_{i+1} = k_{i}\pm 2PD[v]$ for some $v\in \mathcal{V}(\Gamma)$ with $k_{i}(v) = \mp m(v)$. Additionally, it follows from the definition that a representative $k'$ of an element in $\mathcal{L}(\Gamma, [k])$ must satisfy the following property:
\begin{align*}
    m(v) \leq k'(v)\leq -m(v)
\end{align*}
for all $v\in \mathcal{V}(\Gamma)$. We refer to this property as $\star$\label{star} and we let $\star[k] = \{k'\in [k] : k'\text{ satisfies }\star\}$. 

Combining these observations, we get the following proposition:

\begin{prop}\label{prop:algorithm}
An element $k'\in [k]$ represents an element of $\mathcal{L}(\Gamma, [k])$ if and only if $k'$ satisfies \hyperref[star]{$\star$} and every element on every path containing $k'$ also satisfies \hyperref[star]{$\star$}. 
\end{prop}

After using Proposition \ref{prop:algorithm} to find elements $k_{1},\ldots, k_{n}\in [k]$ which represent the distinct elements of $\mathcal{L}(\Gamma, [k])$, it then follows that every other vertex of $(\bar{R}_{k}, \bar{\chi}_{k})$ corresponds to an element of the form $\underline{U^m\otimes k_{i}}$ for some $m$ and $i$. Of course, there could be relations of the form $\underline{U^{m_{1}}\otimes k_{i}} = \underline{U^{m_{2}}\otimes k_{j}}$. To determine these relations, in principle, one can write down the elements of the equivalence classes $\underline{U^{m_{1}}\otimes k_{i}}$ and $\underline{U^{m_{2}}\otimes k_{j}}$ and see whether they are equal. However, this can be quite tedious to do by hand and, in simple enough situations, there are shortcuts one can take by leveraging properties of $HF^+$. For example, we will use the relationship between Turaev torsion and $HF^+$ established in \cite[Theorem 10.17]{MR2113020} to complete the computation of $(\bar{R}_{k}, \bar{\chi}_{k})$ for the manifolds $N_{j}$.

By sections \ref{subsection: hf lattice relationship} and \ref{subsection: rustamov results}, once we have computed $(\bar{R}_{k}, \bar{\chi}_{k})$, we know $HF^+_{odd}(-Y(\Gamma), [k])$. Furthermore, by Rustamov, we know that $HF^+_{even}(-Y(\Gamma), [k]) = \mathcal{T}^+_{d_{-1/2}}$. So to complete the computation of  $HF^+(-Y(\Gamma), [k])$ it suffices to compute $d_{-1/2}(-Y(\Gamma), [k])$. As noted in Rustamov, one strategy to compute $d_{-1/2}(-Y(\Gamma), [k])$ is to first notice that $d_{-1/2}(-Y(\Gamma), [k]) = d_{1/2}(Y(\Gamma), [k])$. Then, if we can find a negative semi-definite plumbing with one bad vertex representing $-Y(\Gamma)$, we can repeat the above steps to compute $HF_{even}^+(Y(\Gamma), [k])$ which then gives us $d_{1/2}(Y(\Gamma), [k])$ and hence $d_{-1/2}(-Y(\Gamma), [k])$. This is the approach we take.

\underline{Step (2)}:
Having done the computations in step (1), it is now easy to complete step (2). By Theorem \ref{thm: involution identification}, to compute $\iota_{*}$, we just need to compute $J_{0}$. As noted in section \ref{subsection: involutions}, $J_{0}$ simply maps $\underline{U^m\otimes k'}$ to $\underline{U^m\otimes -k'}$. $J_{0}$ is also $U$-equivariant. Thus, to compute $J_{0}$, we just need to determine for each leaf representative $k_{i}$, which representative $k_{j}$ corresponds to $-k_{i}$. This amounts to finding a path from $-k_{i}$ to one of the $k_{j}$. 
\newpage

\underline{Step (3)}:
It follows from Proposition \ref{prop: HFI exact triangle} and basic homological algebra, that as a graded $\mathbb{F}$-vector space:
\begin{align*}
    HFI^+_{r}(-Y(\Gamma),[k])\cong \ker Q(1+\iota_{*})_{r-1} \oplus \coker Q(1+\iota_{*})_{r}
\end{align*}
where 
\begin{align*}
    \ker Q(1+\iota_{*})_{r-1} = \ker[Q(1+\iota_{*}):HF^+_{r-1}(-Y(\Gamma),[k])\to Q\cdot HF^+_{r-1}(-Y(\Gamma),[k])]
\end{align*}
and 
\begin{align*}
    \coker Q(1+\iota_{*})_{r} = \coker[Q(1+\iota_{*}):HF^+_{r}(-Y(\Gamma),[k])\to Q\cdot HF^+_{r}(-Y(\Gamma),[k])]
\end{align*}
Furthermore, steps (1) and (2) give us all of the ingredients to compute $\ker Q(1+\iota_{*})_{r-1}$ and $\coker Q(1+\iota_{*})_{r}$ for each $r$. 

\newpage

\section{Small Seifert fibered space examples}\label{section: examples}
In this section, we compute $HFI^+(-N_{j}, \mathfrak{s}_{0})$ for the infinite family of small Seifert fiber spaces $\{N_{j}\}_{j\in\mathbb{N}}$ described in the introduction. As an application, we prove \hyperref[thm: nj not zero surgery]{Theorem (E)}. We also compute $HFI^+(-S^3_{0}(K_{1}),\mathfrak{s}_{0})$ where $S^3_{0}(K_{1})$ is the manifold obtained by $0$-surgery on the Ichihara-Motegi-Song knot $K_{1}$ from \cite{MR2499823}. We then compare $HFI^+(- S^3_{0}(K_{1}),\mathfrak{s}_{0})$ and $HFI^+(- N_{1}, \mathfrak{s}_{0})$. 

\subsection{Moves between equivalent vectors}
Let $\Gamma$ be a negative semi-definite plumbing with at most one bad vertex and with $b_{1} = 1$. Suppose $\Gamma$ contains a linear subgraph $\Lambda$ with framing $-2$ at each vertex:  

\begin{figure}[h]
\centering
\includegraphics[scale=1]{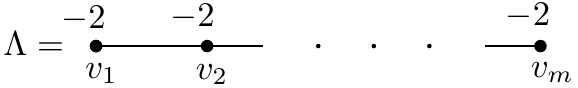}
\end{figure}
Let $[k]$ be a self-conjugate $\spinc$ structure on $Y(\Gamma)$. Given a characteristic vector $k'\in [k]$, let $k'_{\Lambda} = (a_{1}, \ldots, a_{m})$ be the subvector corresponding to the vertices $v_{1}, \ldots, v_{m}$. We call $k'_{\Lambda}$ the $\Lambda$-subvector of $k'$. 

Note, if $k'\in [k]$ and satisfies \hyperref[star]{$\star$}, then we must have $a_{i}\in\{-2, 0, 2\}$ for each $1\leq i\leq m$. If there exists some $i$ such that $a_{i} = \pm 2$, then $k''=k'\pm 2PD[v_{i}]$ is an equivalent vector. In particular, 
\begin{align*}
    k''_{\Lambda} = (a_{1}, \ldots, a_{i-1} \pm 2, \mp 2, a_{i+1} \pm 2, \ldots, a_{m})
\end{align*}
Of course, other entries of $k''$ not contained in $k''_{\Lambda}$ may also differ from those of $k'$. Specifically, any entry $a$ of $k'$ corresponding to a vertex adjacent to $v_{i}$ will change from $a$ to $a\pm 2$. We call the replacement of $k'$ with $k'' = k'\pm 2PD[v_{i}]$ where $k'(v_{i})=\pm 2$ a move of \textbf{type $\pm 2$}.

Next suppose $k'_{\Lambda} = (a_{1},\ldots, a_{i}, 0,\ldots, 0, 2, -2, a_{j}, \ldots, a_{m})$. Then, by iteratively applying type $+2$ moves to the $+2$-entry, we can convert $k'$ into an equivalent vector $k''$ with:
\begin{align*}
    k''_{\Lambda} = (a_{1},\ldots, a_{i}, 2,-2, 0,\ldots, 0, a_{j}, \ldots, a_{m}) 
\end{align*}
We call the replacement of $k'$ with $k''$ or $k''$ with $k'$ a $(2,-2)$-\textbf{slide}. We define a $(-2,2)$-\textbf{slide} analogously.

\begin{lem}
Let $k'\in [k]$ be a vector with $k'_{\Lambda} = (a_{1}, \ldots, a_{i}, 0, \pm 2, 0, \ldots, 0, \mp 2, a_{j}, \ldots, a_{m})$. Then, $k'$ is equivalent to a vector $k''$ with $k''_{\Lambda}= (a_{1}, \ldots, a_{i}, \pm 2, 0, \ldots, 0, \mp 2, 0, a_{j}, \ldots, a_{m})$. 
\end{lem}
\begin{proof}
Apply a type $\pm 2$ move to the $\pm 2$-entry to get an equivalent vector $h'$ with:
\begin{align*}
    h'_{\Lambda} = (a_{1},\ldots, a_{i}, \pm 2, \mp 2, \pm 2, 0, \ldots, 0,\mp 2, a_{j}, \ldots, a_{m})
\end{align*}
Now do a rightward $(\mp 2, \pm 2)$-slide to $h'$ to convert $h'$ into an equivalent vector $h''$ with: 
\begin{align*}
    h''_{\Lambda} = (a_{1},\ldots, a_{i}, \pm 2, 0,\ldots, 0, \mp 2, \pm 2, \mp 2, a_{j}, \ldots, a_{m})
\end{align*}
Finally apply a type $\pm 2$ move to the rightmost $\pm 2$-entry to get an equivalent vector $k''$ with: 
\begin{align*}
    k''_{\Lambda} = (a_{1}, \ldots, a_{i}, \pm 2, 0, \ldots, 0, \mp 2, 0, a_{j}, \ldots, a_{m})
\end{align*}
\end{proof}
By iterating the sequence of moves described in the above proof, we can now convert any vector $k'\in [k]$ with:
\begin{align*}
    k'_{\Lambda} = (a_{1},\ldots, a_{i}, 0,\ldots, 0, \pm 2, 0, \ldots, 0, \mp 2, a_{j}, \ldots, a_{m})
\end{align*}
into an equivalent vector $k''$ with:
\begin{align*}
    k''_{\Lambda} = (a_{1}, \ldots, a_{i}, \pm 2, 0,\ldots, 0, \mp 2, 0 ,\ldots, 0, a_{j},\ldots, a_{m})
\end{align*}
By an abuse of notation, we also call the replacement of $k'$ with $k''$ or $k''$ with $k'$ via the above sequence of moves a \textbf{$(\pm 2, \mp 2)$-slide}.

\begin{lem}\label{lem: alternating}
Suppose $k'\in [k]$ represents an element of $\mathcal{L}(\Gamma, [k])$, then either $k'_{\Lambda}$ is the zero vector or it has entries which alternate between $2$ and $-2$ with possibly $0$s inbetween. 
\end{lem}
\begin{proof}
Suppose $k'$ represents an element of $\mathcal{L}(\Gamma, [k])$ and $k'_{\Lambda}$ contains a subvector of the form $(2,\underbrace{0,\ldots, 0}_{j}, 2)$ where $j\geq 0$. Then, by doing a type $+2$ move on the leftmost $+2$-entry, $k'$ is equivalent to a vector whose corresponding subvector is $(-2,2, \underbrace{0,\ldots, 0}_{j-1}, 2)$ if $j\geq 1$ or $(-2,4)$ if $j=0$. In the latter case, the vector fails to satisfy \hyperref[star]{$\star$} and thus we get a contradiction by Proposition \ref{prop:algorithm}. So we can assume the subvector is $(-2,2, \underbrace{0,\ldots, 0}_{j-1}, 2)$ with $j\geq 1$. Now do a rightward $(-2,2)$-slide to produce an equivalent vector whose corresponding subvector is $(\underbrace{0,\ldots, 0}_{j-1},-2, 2,2)$. Next apply a type $+2$ move to get an equivalent vector whose corresponding subvector is $(\underbrace{0,\ldots, 0}_{j},-2,4)$. We again get a contradiction for the same reason as before. Therefore, $k'_{\Lambda}$ cannot contain a subvector of the form $(2,\underbrace{0,\ldots, 0}_{j}, 2), j\geq 0$. By an analogous argument, $k'_{\Lambda}$ also cannot contain a subvector of the form $(-2,\underbrace{0,\ldots, 0}_{j}, -2), j\geq 0$. This completes the proof. 
\end{proof}

\begin{lem}\label{lem: sparse vector}
Suppose $k'\in [k]$ represents an element of $\mathcal{L}(\Gamma, [k])$. Then, $k'$ is equivalent to a vector $k''$ such that $k''_{\Lambda}$ is the zero vector except for possibly one non-zero entry equal to $\pm 2$.  
\end{lem}

\begin{proof}
We induct on the number of non-zero entries of $k'_{\Lambda}$. Obviously the statement is true if $k'_{\Lambda}$ is the zero vector or has only one-nonzero entry. So suppose $k'_{\Lambda}$ has $n\geq 2$ non-zero entries. Let $a_{i}$ and $a_{i+j}$ be the leftmost non-zero entries. Then by the Lemma \ref{lem: alternating}, $a_{i} = \pm 2$ and $a_{i+j}=\mp 2$. For simplicity, assume $a_{i}=2$. (The argument when $a_{i}=-2$ is identical up to sign changes.) We can write $k'_{\Lambda}$ as:
\begin{align*}
    k'_{\Lambda}=(0, \ldots, 0,2, 0,\ldots, 0, -2, a_{i+j+1}, \ldots, a_{m})
\end{align*}
where there are possibly no initial $0$ entries and no $0$ entries between $a_{i}$ and $a_{i+j}$. If there are initial $0$ entries, then by doing a leftward $(2,-2)$-slide, $k'$ is equivalent to a vector whose $\Lambda$-subvector is
\begin{align*}
    (2,0, \ldots,0, -2,0,\ldots, 0, a_{i+j+1}, \ldots, a_{m})
\end{align*}
Now apply a type $+2$ move to the left most $+2$-entry to get an equivalent vector whose $\Lambda$-subvector is:
\begin{align*}
    (-2,2,0,\ldots, 0,-2, 0,\ldots, 0, a_{i+j+1}, \ldots, a_{m})
\end{align*}
if $j>1$ or: 
\begin{align*}
    (-2,0,\ldots, 0, a_{i+2}, \ldots, a_{m})
\end{align*}
if $j=1$. In the latter case, we have reduced the number of non-zero entries in the $\Lambda$-subvector by 1. Hence, we can assume $j>1$. In this case, if we do a rightward $(-2, 2)$-slide on leftmost $(-2,2)$-pair, we get an equivalent vector whose $\Lambda$-subvector is:
\begin{align*}
    (0,\ldots, 0, -2,2,-2,0,\ldots, 0, a_{i+j+1},\ldots, a_{m})
\end{align*}
Finally apply a type $+2$ move to produce an equivalent vector whose $\Lambda$-subvector is
\begin{align*}
    (0,\ldots, 0, 0,-2,0,0,\ldots, 0, a_{i+j+1}, \ldots, a_{m})
\end{align*}
We have reduced the number of non-zero entries by 1. Therefore, by induction the result follows.
\end{proof}

\begin{lem}\label{lem: one 2}
Suppose $k'\in [k]$ with:
\begin{align*}
    k'_{\Lambda} = (\underbrace{0,\ldots, 0}_{j\ copies}, 2, \underbrace{0,\ldots, 0}_{m-j-1\ copies})
\end{align*}
Then $k'$ is equivalent to a vector $k''$ with:
\begin{align*}
    k''_{\Lambda} = (\underbrace{0,\ldots, 0}_{m-j-1\ copies}, -2, \underbrace{0,\ldots, 0}_{j\ copies})
\end{align*}
\end{lem}

\begin{proof}
We list the sequence of moves needed to obtain the relevant vector. In each move, we only write the resulting $\Lambda$-subvector. 
\begin{enumerate}
    \item Type $+2$ move: 
    \begin{align*}
        (\underbrace{0,\ldots, 0}_{j-1\ copies}, 2,-2,2 \underbrace{0,\ldots, 0}_{m-j-2\ copies})
    \end{align*}
    
    \item Leftward $(2,-2)$-slide: 
    \begin{align*}
        (2,-2,\underbrace{0,\ldots, 0}_{j-1\ copies},2, \underbrace{0,\ldots, 0}_{m-j-2\ copies})
    \end{align*}
    
    \item Type $+2$ move: 
    \begin{align*}
        (-2,\underbrace{0,\ldots, 0}_{j\ copies},2, \underbrace{0,\ldots, 0}_{m-j-2\ copies})
    \end{align*}
    
    \item Rightward $(-2,2)$-slide: 
    \begin{align*}
        (\underbrace{0,\ldots, 0}_{m-j-2\ copies},-2,\underbrace{0,\ldots, 0}_{j\ copies},2)
    \end{align*}
    
    \item Type $+2$ move:
    \begin{align*}
        (\underbrace{0,\ldots, 0}_{m-j-2\ copies},-2,\underbrace{0,\ldots, 0}_{j-1\ copies},2,-2)
    \end{align*}

    \item Leftward $(2,-2)$-slide: 
    \begin{align*}
        (\underbrace{0,\ldots, 0}_{m-j-2\ copies},-2,2,-2\underbrace{0,\ldots, 0}_{j-1\ copies})
    \end{align*}
    
    \item Type $+2$ move: 
    \begin{align*}
        (\underbrace{0,\ldots, 0}_{m-j-1\ copies},-2,\underbrace{0,\ldots, 0}_{j\ copies})
    \end{align*}
\end{enumerate}
\end{proof}

\begin{rem}\label{rem:adj vertex change}
If one traces through the above sequence of moves, it is easy to see that if $v$ is a vertex not in $\Lambda$, but is adjacent to the initial vertex $v_{1}$ or terminal vertex $v_{m}$ of $\Lambda$, then $k''(v) = k'(v)+2$. 
\end{rem}

\subsection{Computation of $HFI^+(-N_{j}, \mathfrak{s}_{0})$}

\subsubsection{Step 1}

Recall, the 3-manifold $N_{j}, j\geq 1$ is given by the following surgery diagram:

\begin{figure}[h]
  \centering
  \includegraphics[scale=.7]{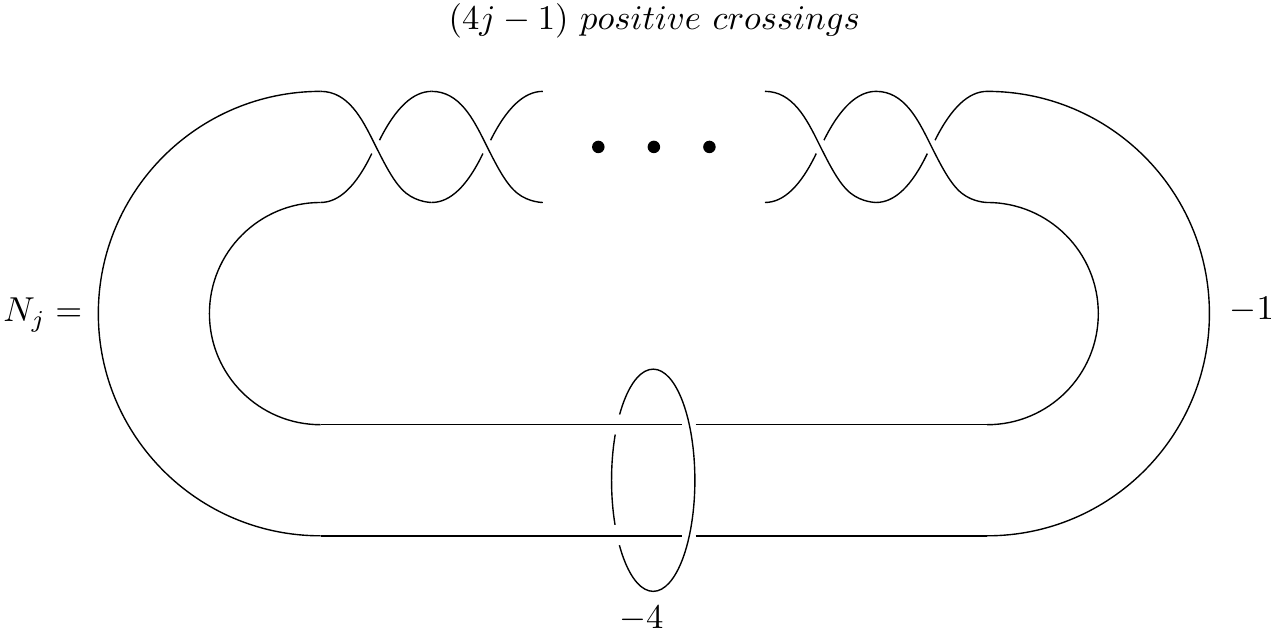}
  \end{figure}

In \cite[Section 7]{MR4029676}, it is shown via Kirby calculus that $N_{j}$ can be represented as a plumbing as follows:
\begin{figure}[h]
  \centering
  \includegraphics[scale=1]{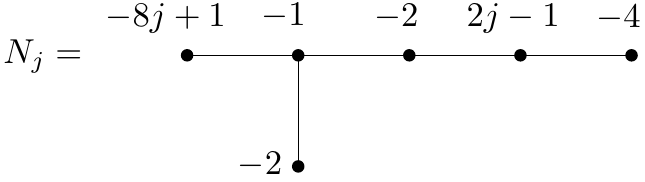}
  \end{figure}
  
\noindent By performing two slam dunks on the rightward stem, we get:
\begin{figure}[h]
  \centering
  \includegraphics[scale=1]{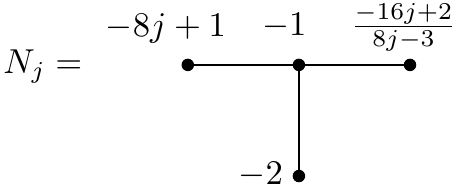}
  \end{figure}
  
\noindent One can further check that: 
\begin{align*}
    \frac{-16j+2}{8j-3}= -3-\cfrac{1}{-2-\cfrac{1}{\ddots-2-\cfrac{1}{\frac{-8j+3+4r}{8j-7-4r}}}}
\end{align*}
where there are $r$ copies of $-2$ along the diagonal. In particular, setting $r = 2j-2$, the last term becomes: 
\begin{align*}
    \frac{-8j+3+4(2j-2)}{8j-7-4(2j-2)}=-5
\end{align*}
Hence, by performing the corresponding slam dunks, we get:
\begin{figure}[h]
  \centering
  \includegraphics[scale=1]{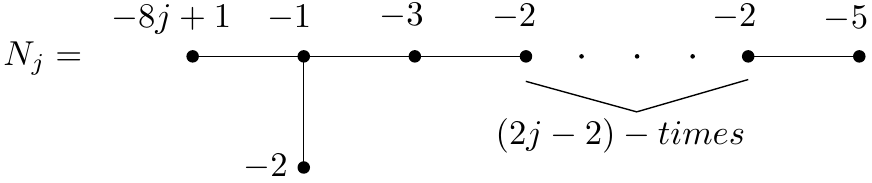}
  \end{figure}
  
\noindent Let $\Gamma_{j}$ be the above plumbing graph with vertices labeled as follows:

\begin{figure}[h]
  \centering
  \includegraphics[scale=1]{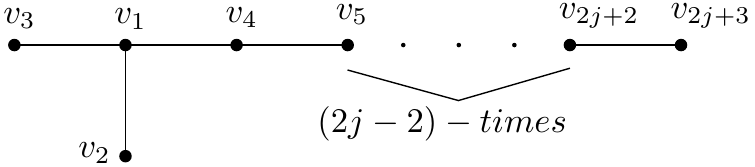}
  \end{figure}

\noindent With respect to the ordered basis $([v_{1}],\ldots, [v_{2j+3}])$, the matrix for the intersection form of $X(\Gamma_{j})$ is:
\begin{center}
$B_{j}=\begin{pmatrix}
-1&\phantom{-}1&1&\phantom{-}1&\phantom{--}&\phantom{-}& \phantom{-}&\phantom{-}&\phantom{--} &\phantom{--}\\
\phantom{-}1&-2&\phantom{-}&\phantom{--}&\phantom{--}&\phantom{-}& \phantom{-}&\phantom{-}&\phantom{--} &\phantom{--}\\
\phantom{-}1&\phantom{--}&-8j+1&\phantom{--}&\phantom{--}&\phantom{-}&\phantom{-}&\phantom{-}&\phantom{--} &\phantom{--}\\
\phantom{-}1&\phantom{--}&\phantom{-}&-3&\phantom{-}1&\phantom{-}&\phantom{-}&\phantom{-}&\phantom{--}&\phantom{--}\\
\phantom{--}&\phantom{--}&\phantom{-}&\phantom{-}1&-2&\phantom{-}1&\phantom{-}&\phantom{-}&\phantom{--}&\phantom{-]}\\
\phantom{--}&\phantom{--}&\phantom{-}&\phantom{--}&\phantom{-}1&-2&1&\phantom{-}&\phantom{--}&\phantom{--}\\
\phantom{--}&\phantom{--}&\phantom{-}&\phantom{--}&\phantom{--}&\phantom{--}&\ddots&\phantom{-}&\phantom{--}&\phantom{--}\\
\phantom{--}&\phantom{--}&\phantom{-}&\phantom{--}&\phantom{--}&\phantom{-}&\phantom{-}&1&-2&\phantom{-}1\\
\phantom{--}&\phantom{--}&\phantom{-}&\phantom{--}&\phantom{--}&\phantom{-}&\phantom{-}&\phantom{-}&\phantom{-}1&-5
\end{pmatrix}$
\end{center}
It is straightforward to check that $B_{j}$ is negative semi-definite and $H_{1}(N_{j};\Z)\cong \Z$, we leave this to the reader. 

Note, the $\Z$-kernel of $B_{j}$ is generated by the vector:
\begin{align*}
    x = (16j-2, 8j-1, 2, 8j-3, 8j-7,8j-11, \ldots, 1)
\end{align*}
Therefore, the unique self-conjugate $\spinc$-structure $\mathfrak{s}_{0}$ on $N_{j}$ can be thought of as:
\begin{align*}
    [k] = \{k'\in \Char(X(\Gamma_{j}))\ |\ k'\cdot x = 0\}
\end{align*}
Let $\Lambda_{j}$ be the linear subgraph of $\Gamma_{j}$ given by:
\begin{figure}[h]
  \centering
  \includegraphics[scale=1]{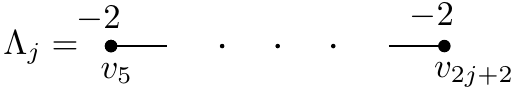}
  \end{figure}
  
\noindent We write vectors $k'\in [k]$ as: 
\begin{align*}
    k' = (a_{1}, a_{2},a_{3},a_{4}, b_{5},\ldots, b_{2j+2}, c_{2j+3})
\end{align*}
where $k'_{\Lambda_{j}} = (b_{5},\ldots, b_{2j+2})$. 

\newpage

\begin{lem}\label{lem: non-zero lambda vector}
If $k'\in [k]$ represents an element of $\mathcal{L}(\Gamma_{j}, [k])$, then $k'$ is equivalent to a vector whose $\Lambda_{j}$-subvector is not equal to the zero vector. 
\end{lem}

\begin{proof}
Suppose $k'\in [k]$ represents an element of $\mathcal{L}(\Gamma_{j}, [k])$. For the purpose of contradiction, suppose the $\Lambda_{j}$-subvector of every representative of every element of $\mathcal{L}(\Gamma_{j}, [k])$ is zero. Then, in particular, $k'_{\Lambda_{j}} = 0$. Also, since $k'$ represents an element of $\mathcal{L}(\Gamma_{j}, [k])$, it must satisfy \hyperref[star]{$\star$}. So we must have $a_{4}\in \{-3,-1,1,3\}$. If $a_{4}=\pm 3$, then by adding $\pm 2PD[v_{4}]$ to $k'$ we would obtain an equivalent vector with a non-zero $\Lambda_{j}$-subvector. Thus, $a_{4}\in \{-1,1\}$. 

Since $k'$ must satisfy \hyperref[star]{$\star$}, we also have $a_{1}=\pm 1$. If $a_{1} = 1$ and $a_{4}=1$, then by adding $2PD[v_{1}]$ to $k'$, $a_{4}$ becomes $3$. But we just showed that $a_{4}$ cannot be equal to 3. Similarly, if $a_{1}=-1$ and $a_{4}=-1$, then by adding $-2PD[v_{1}]$ to $k'$, $a_{4}$ becomes $-3$, which is again a contradiction. Hence, $a_{1}=\pm 1$ and $a_{4}=\mp 1$. By adding $-2PD[v_{1}]$ if necessary, we may assume $a_{1}=1$ and $a_{4}=-1$. Again, by \hyperref[star]{$\star$}, we must have $a_{2}\in \{-2,0,2\}$. If $a_{2}=2$, then by adding $2PD[v_{1}]$ to $k'$, we get an equivalent vector with $a_{2}=4$, which contradicts Proposition \ref{prop:algorithm}. Therefore, $a_{2}\in\{0,-2\}$. If $a_{2}=-2$, then by adding $-2PD[v_{2}]$ to $k'$ we obtain an equivalent vector with $a_{1}=-1$ and $a_{4}=-1$, which we already determined cannot happen. Therefore, $a_{2}=0$. Now add $2PD[v_{1}]$ to $k'$. The result is an equivalent vector with $a_{1}=-1$, $a_{2}=2$, and $a_{4}=1$. Since $a_{2}=2$, we can add $2PD[v_{2}]$ to get an equivalent vector with $a_{1}=1, a_{2}=-2$, and $a_{4}=1$, but we have already shown that we cannot have both $a_{1}=1$ and $a_{4}=1$. Therefore, we get a contradiction and hence $k'$ must be equivalent to some vector whose $\Lambda_{j}$-subvector is not equal to the zero vector.   
\end{proof}

Somewhat counter-intuitively, we are now going to use the previous lemma to find a small finite set of possible representatives of $\mathcal{L}(\Gamma_{j}, [k])$, all of whose $\Lambda_{j}$-subvectors are all equal to the zero vector. 

\begin{lem}
If $k'\in [k]$ represents an element of $\mathcal{L}(\Gamma_{j}, [k])$, then $k'$ is equivalent to a vector of the form 
\begin{align*}
    k'' = (-1,0,a_{3},3, 0,\ldots, 0, c_{2k+3})
\end{align*}
where $a_{3}\in \{-8k+1, -8k+3, \ldots, 8k+1\}$ and $c_{2j+3}\in\{-5,-3,-1,1,3\}$. 
\end{lem}

\begin{proof}
Suppose $k'$ represents an element of $\mathcal{L}(\Gamma_{j}, [k])$. Then, by combining Lemmas \ref{lem: sparse vector}, \ref{lem: one 2}, and \ref{lem: non-zero lambda vector}, we may assume:
\begin{align*}
    k'_{\Lambda_{j}} = (\underbrace{0,\ldots, 0}_{\ell\ copies}, 2, \underbrace{0,\ldots, 0}_{2j-3-\ell\ copies})
\end{align*}
for some $0\leq \ell\leq 2j-3$. By \hyperref[star]{$\star$}, $a_{1}=\pm 1$. If $a_{1}=-1$, then we can add $-2PD[v_{1}]$ from $k'$ to get an equivalent vector with $a_{1}=1$. This addition does not effect any of the entries in $k'_{\Lambda_{j}}$. Thus, we may assume $a_{1}=1$. 

Next, by \hyperref[star]{$\star$}, $a_{2}\in \{-2,0,2\}$. If $a_{2}=2$, then adding $2PD[v_{1}]$ to $k'$ yields an equivalent vector with $a_{2}=4$, which violates \hyperref[star]{$\star$}. Therefore, $a_{2}\in\{-2,0\}$. Suppose $a_{2}=-2$. Then, by adding $-2PD[v_{2}]$, we get an equivalent vector with $a_{2}=2$ and $a_{1}=-1$. $k'_{\Lambda_{j}}$ is unaffected by this move. If we then add $-2PD[v_{1}]$, we get an equivalent vector with $a_{1}=1$ and $a_{2}=0$. Again $k'_{\Lambda_{j}}$ is unaffected. Therefore, we may assume $a_{2}=0$.

Next, with:
\begin{align*}
    k' = (1, 0, a_{3}, a_{4}, \underbrace{0,\ldots, 0}_{\ell\ copies}, 2, \underbrace{0,\ldots, 0}_{2j-3-\ell\ copies}, c_{2j+3})
\end{align*}
add $2PD[v_{1}]$ to $k'$ to get the equivalent vector:
\begin{align*}
    (-1, 2, a_{3}+2, a_{4}+2, \underbrace{0,\ldots, 0}_{\ell\ copies}, 2, \underbrace{0,\ldots, 0}_{2j-3-\ell\ copies}, c_{2j+3})
\end{align*}
Next, add $2PD[v_{2}]$ to get: 
\begin{align*}
    (1, -2, a_{3}+2, a_{4}+2, \underbrace{0,\ldots, 0}_{\ell\ copies}, 2, \underbrace{0,\ldots, 0}_{2j-3-\ell\ copies}, c_{2j+3})
\end{align*}
Then add another $2PD[v_{1}]$, to get:
\begin{align*}
    (-1, 0, a_{3}+4, a_{4}+4, \underbrace{0,\ldots, 0}_{\ell\ copies}, 2, \underbrace{0,\ldots, 0}_{2j-3-\ell\ copies}, c_{2j+3})
\end{align*}
Now, if we apply the move in Lemma \ref{lem: one 2} and take into account Remark \ref{rem:adj vertex change}, one can check that we get an equivalent vector whose $4th$-entry is $a_{4}+6$. Since we assumed $k'$ represents an element of $\mathcal{L}(\Gamma_{j}, [k])$, we must therefore have that $a_{4}\in \{-3,-1,1,3\}$ and $a_{4}+6\in\{-3,-1,1,3\}$. Hence, we must have had $a_{4}=-3$. To summarize, we have now shown that we can assume:
\begin{align*}
    k' = (1, 0, a_{3}, -3, \underbrace{0,\ldots, 0}_{\ell\ copies}, 2, \underbrace{0,\ldots, 0}_{2j-3-\ell\ copies}, c_{2j+3})
\end{align*}

Next, by \hyperref[star]{$\star$}, $c_{2j+3}\in \{-5,-3,-1,1,3,5\}$. If $c_{2j+3}=5$, then again by applying the move from Lemma \ref{lem: one 2}, one can check that we transform $c_{2j+3}$ into $7$, which violates \hyperref[star]{$\star$}. Therefore, we must have had $c_{2j+3}\in \{-5,-3,-1,1,3\}$. 

Now add $-2PD[v_{4}]$ to get an equivalent vector (which we again call $k'$) with $a_{1}=-1$, $a_{2}=0$, $a_{4}=3$ and $k'_{\Lambda_{j}}$ unchanged except for the first entry which decreases by 2. Also, $c_{2k+3}$ remains unchanged. If $\ell=0$, then $k'_{\Lambda_{j}}$ is now the zero vector, so we are done. Thus, suppose $\ell>0$. Then, 
\begin{align*}
    k' = (-1,0, a_{3}, 3, -2, \underbrace{0,\ldots, 0}_{\ell-1\ copies},2, \underbrace{0,\ldots, 0}_{2j-3-\ell\ copies}, c_{2j+3})
\end{align*}
Now consider the following sequence of moves:
\begin{enumerate}
    \item Rightward $(-2, 2)$-slide:
\begin{align*}
    (-1,0, a_{3}, 3, \underbrace{0,\ldots, 0}_{2j-3-\ell\ copies}, -2, \underbrace{0,\ldots, 0}_{\ell-1\ copies},2, c_{2j+3})
\end{align*}

\item Type $+2$ move:
\begin{align*}
    (-1,0, a_{3}, 3, \underbrace{0,\ldots, 0}_{2j-3-\ell\ copies}, -2, \underbrace{0,\ldots, 0}_{\ell-2\ copies},2,-2, c_{2j+3}+2)
\end{align*}

\item Leftward $(2,-2)$-slide:
\begin{align*}
    (-1,0, a_{3}, 3, \underbrace{0,\ldots, 0}_{2j-3-\ell\ copies}, -2,2,-2 \underbrace{0,\ldots, 0}_{\ell-2\ copies}, c_{2j+3}+2)
\end{align*}

\item Type $+2$ move:
\begin{align*}
    (-1,0, a_{3}, 3, \underbrace{0,\ldots, 0}_{2j-2-\ell\ copies},-2, \underbrace{0,\ldots, 0}_{\ell-1\ copies}, c_{2j+3}+2)
\end{align*}

\item Apply Lemma \ref{lem: one 2} and Remark \ref{rem:adj vertex change}:
\begin{align*}
    (-1,0, a_{3}, 1, \underbrace{0,\ldots, 0}_{\ell-1\ copies},2,\underbrace{0,\ldots, 0}_{2j-2-\ell\ copies}, c_{2j+3})
\end{align*}

\item Add $-2PD[v_{1}]$:
\begin{align*}
    (1,-2, a_{3}-2, -1, \underbrace{0,\ldots, 0}_{\ell-1\ copies},2,\underbrace{0,\ldots, 0}_{2j-2-\ell\ copies}, c_{2j+3})
\end{align*}

\item Add $-2PD[v_{2}]$:
\begin{align*}
    (-1,2, a_{3}-2, -1, \underbrace{0,\ldots, 0}_{\ell-1\ copies},2,\underbrace{0,\ldots, 0}_{2j-2-\ell\ copies}, c_{2j+3})
\end{align*}

\item Add $-2PD[v_{1}]$:
\begin{align*}
    (1,0, a_{3}-4, -3, \underbrace{0,\ldots, 0}_{\ell-1\ copies},2,\underbrace{0,\ldots, 0}_{2j-2-\ell\ copies}, c_{2j+3})
\end{align*}
\end{enumerate}
The net effect of this sequence of moves is that the $+2$-entry in $k'_{\Lambda_{j}}$ shifts one space to the left while every other entry, excluding $a_{3}$, remains the same. So now we can repeat the above process until $+2$-entry is in the first position of $k'_{\Lambda_{j}}$. Then add $-2PD[v_{4}]$ to get:
\begin{align*}
    (-1,0, a_{3}', 3, 0,\ldots, 0, c_{2j+3})
\end{align*}
with $c_{jk+3}\in \{-5,-3,-1,1,3\}$ and, by \hyperref[star]{$\star$}, $a_{3}'\in \{-8j+1, -8j+3, \ldots, 8j+1\}$
\end{proof}
\begin{prop}\label{prop: -N reps}
If $k'$ represents an element of $\mathcal{L}(\Gamma_{j}, [k])$, then $k'$ is equivalent to 
\begin{align*}
    k_{1}=(-1, 0, 5-4j, 3,0, \ldots, 0, -3)\text{ or }k_{2}=(-1, 0, 3-4j, 3,0, \ldots, 0, 1)
\end{align*}
In particular, $|\mathcal{L}(\Gamma_{j}, [k])|\leq 2$. 
\end{prop}
\begin{proof}
Up to this point, we have not used the fact that $k'\cdot x = 0$ where $x$ is a generator of $\ker_{\Z}(B_{j})$ as above. So assume $k'$ is of the form in the previous lemma. Then, we get the following equation:
\begin{align*}
    0 = k'\cdot x =8j-7+2a_{3}+c_{2j+3}
\end{align*}
where $a_{3}\in \{-8j+1, -8k+3, \ldots, 8j+1\}$ and $c_{2j+3}\in\{-5,-3,-1,1,3\}$. The only solutions to this equation with the given constraints are: $(a_{3}, c_{2j+3}) = (5-4j, -3)$ and $(3-4j, 1)$, corresponding to $k_{1}$ and $k_{2}$, respectively. 
\end{proof}

We have not yet proved that $k_{1}$ and $k_{2}$ represent different elements of $\mathcal{L}(\Gamma_{j}, [k])$. To do this we will do a similar analysis for $-N_{j}$ and then use Turaev torsion. However, before we undertake this task, we first compute the $HF^+$ grading associated to the vectors $k_{1}$ and $k_{2}$. 

\begin{cor}\label{cor: d inv 1}
$d_{1/2}(-N_{j};\mathfrak{s}_{0}) =\frac{1}{2}$
\end{cor}

\begin{proof}
Let
\begin{align*}
    \alpha_{1} &= (-12j+6, -6j+3, -1, -6j+3, -6j, -6j-3, \ldots, 2, 1, 0)\\
    \alpha_{2} &= (4j+2, 2j+1, 1, 2j-1, 2j-2, 2j-3, \ldots,2, 1, 0)
\end{align*}
Then, $\alpha_{1}B_{j}= k_{1}$ and $\alpha_{2}B_{j}= k_{2}$. Thus, 
\begin{align*}
    k_{1}^2 &= k_{1}\cdot \alpha_{1} = -2j-2\\
    k_{2}^2 &= k_{2}\cdot \alpha_{2} = -2j-2
\end{align*}
Hence, under the isomorphism from Corollary \ref{cor: HF and lattice identification}, the elements of $HF^+(-N_{j}, \mathfrak{s}_{0})$ corresponding to $k_{1}$ and $k_{2}$ have gradings:
\begin{align*}
    \gr(k_{1}) = \gr(k_{2}) &= -\frac{k_{2}^2 + |\mathcal{V}(\Gamma_{j})|-3}{4}\\
               &= -\frac{-2j-2+2j+3-3}{4}\\
               &=\frac{1}{2}
\end{align*}
\end{proof}

We now find a plumbing representation of $-N_{j}$ and then do Kirby calculus to make it negative semi-definite:
\begin{figure}[h]
  \centering
  \includegraphics[scale=.8]{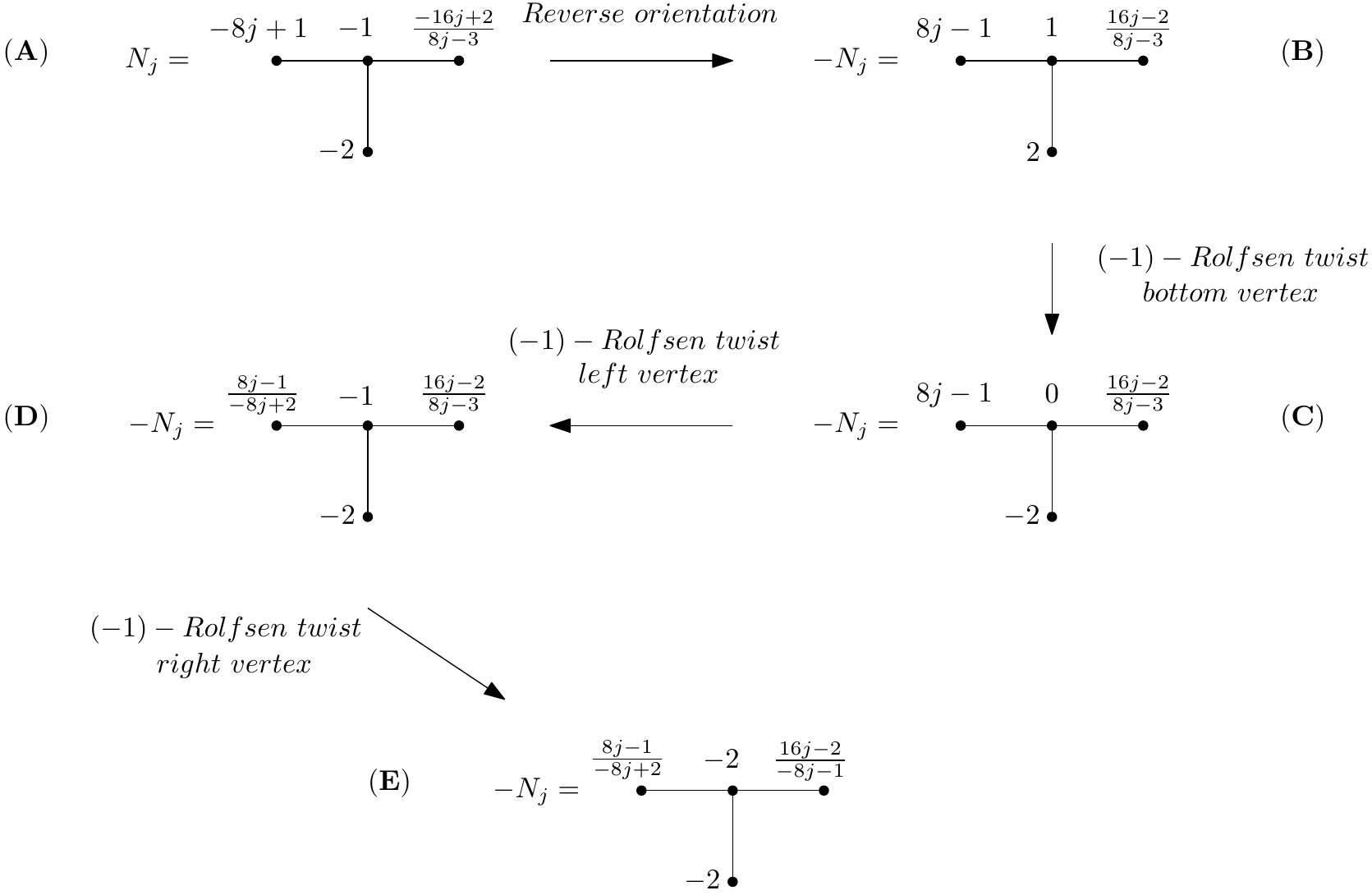}
  \end{figure}
  
\newpage

\noindent Now do slam dunks on the left and right vertices to get:
\begin{figure}[h]
  \centering
  \includegraphics[scale=1]{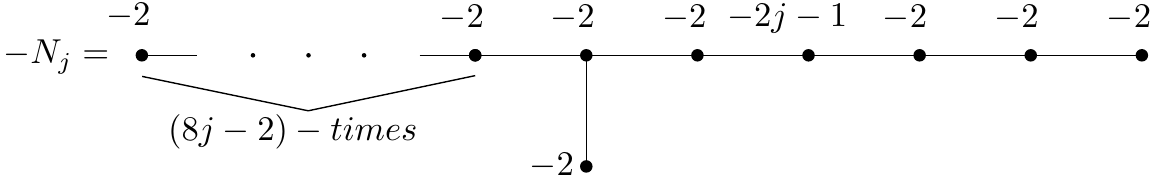}
  \end{figure}
  
\noindent Let $\Gamma'_{j}$ be the above plumbing graph with vertices labeled as follows:
\begin{figure}[h]
  \centering
  \hbox{\hspace{10.75em}
  \includegraphics[scale=1]{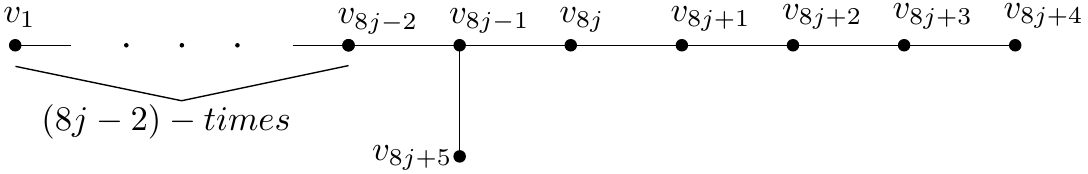}}
  \end{figure}

\noindent With respect to the ordered basis ($[v_{1}], \ldots, [v_{8j+5}])$, the matrix for the intersection form of $X(\Gamma'_{j})$ is:
\begin{center}
$B'_{j}=
\begin{pmatrix}
-2&\phantom{-}1&\phantom{--}\\
\phantom{-}1&-2&1\\
\phantom{--}&\phantom{--}&\phantom{-}&\ddots\\
\phantom{--}&\phantom{--}&\phantom{-}&1&-2&\phantom{-}1&\\
\phantom{--}&\phantom{--}&\phantom{-}&\phantom{-}&\phantom{-}1&-2&\phantom{-}1&\phantom{--}&\phantom{--}&\phantom{--}&\phantom{--}&\phantom{-}1\\
\phantom{--}&\phantom{--}&\phantom{--}&\phantom{-}&\phantom{-}&\phantom{-}1&-2&\phantom{-}1&\\
\phantom{--}&\phantom{--}&\phantom{--}&\phantom{--}&\phantom{-}&\phantom{-}&\phantom{-}1&-2j-1&\phantom{-}1\\
\phantom{--}&\phantom{--}&\phantom{--}&\phantom{--}&\phantom{--}&\phantom{-}&\phantom{-}&\phantom{-}1&-2&\phantom{-}1\\
\phantom{--}&\phantom{--}&\phantom{--}&\phantom{--}&\phantom{--}&\phantom{--}&\phantom{-}&\phantom{-}&\phantom{-}1&-2&\phantom{-}1\\
\phantom{--}&\phantom{--}&\phantom{--}&\phantom{--}&\phantom{--}&\phantom{--}&\phantom{--}&\phantom{-}&\phantom{-}&\phantom{-}1&-2&\phantom{--}\\
\phantom{--}&\phantom{--}&\phantom{--}&\phantom{--}&\phantom{--}&\phantom{-}1&\phantom{--}&\phantom{-}&\phantom{-}&\phantom{--}&\phantom{--}&-2\\
\end{pmatrix}$
\end{center}

Again, it is straightforward to check that $B'_{j}$ is negative semi-definite. Also, the $\Z$-kernel of $B'_{j}$ is generated by the vector:
\begin{align*}
    x' = (2,4, 6, \ldots, 16j-2, 8j+1, 4,3, 2, 1, 8j-1)
\end{align*}
Let $t$ denote a characteristic vector representing the trivial $\spinc$ structure $\mathfrak{s}_{0}$. Then again, we can think of $\mathfrak{s}_{0}$ as: 
\begin{align*}
    [t] = \{t'\in \Char(X(\Gamma_{j}'))\ |\ t'\cdot x' = 0\}
\end{align*}
Let $\Lambda'_{j}$ be the linear subgraph of $\Gamma'_{j}$ given by:
\begin{figure}[h]
  \centering
  \includegraphics[scale=1]{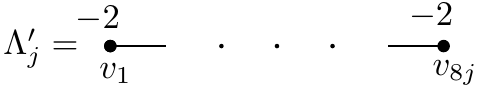}
  \end{figure}
  
\noindent We write vectors $t'\in [t]$ as: 
\begin{align*}
    t' = (a_{1}, a_{2},\ldots, a_{8j}, b_{8j+1}, c_{8j+2}, c_{8j+3}, c_{8j+4}, d_{8j+5})
\end{align*}
where $t'_{\Lambda'_{j}} = (a_{1}, a_{2},\ldots, a_{8j})$.

\begin{lem}
If $t'\in [t]$ represents an element of $\mathcal{L}(\Gamma'_{j}, [t])$, then $t'$ is equivalent to a vector whose $\Lambda'_{j}$-subvector is of the form:
\begin{align*}
    (0, \ldots, 0, a_{8j})
\end{align*}
where $a_{8j}\in\{0, 2\}$.
\end{lem}

\begin{proof}
Suppose $t'$ represents and element of $\mathcal{L}(\Gamma'_{j}, [t])$. By Lemmas \ref{lem: sparse vector} and \ref{lem: one 2}, it suffices to consider the case when: 
\begin{align*}
    t'_{\Lambda'_{j}} = (\underbrace{0,\ldots, 0}_{\ell\ copies}, 2, \underbrace{0,\ldots, 0}_{8j-1-\ell\ copies})
\end{align*}
for some $0\leq \ell\leq 8j-2$. Furthermore, by considering the linear subgraph of $\Gamma'_{j}$ whose endpoints are $v_{\ell+1}$ and $v_{8j+5}$, it follows from Lemma \ref{lem: alternating} that $d_{8j+5}\in \{0, -2\}$.

\underline{Case 1:} Suppose $d_{8j+5} = -2$ and $\ell = 8j-2$. If we add $-2PD[v_{8j+5}]$, then the $\Lambda'_{j}$-subvector of the resulting vector is zero, so we are done. 

\underline{Case 2:} Suppose $d_{8j+5} = -2$ and $\ell \leq 8j-3$. Consider the following sequence of moves:
\begin{enumerate}
    \item Add $-2PD[v_{8j+5}]$:
    \begin{align*}
    (\underbrace{0,\ldots, 0}_{\ell\ copies}, 2, \underbrace{0,\ldots, 0}_{8j-3-\ell\ copies},-2, 0, b_{8j+1}, c_{8j+2}, c_{8j+3}, c_{8j+4}, 2)
\end{align*}

\item Rightward $(2,-2)$-slide:
\begin{align*}
    (\underbrace{0,\ldots, 0}_{\ell+1\ copies}, 2, \underbrace{0,\ldots, 0}_{8j-3-\ell\ copies},-2, b_{8j+1}, c_{8j+2}, c_{8j+3}, c_{8j+4}, 0)
\end{align*}
Note, the rightmost entry of the vector changes from $2$ to $0$.

\item Type $-2$ move on the leftmost $-2$:
\begin{align*}
\begin{cases}
        (\underbrace{0,\ldots, 0}_{8j-1\ copies}, 2, b_{8j+1}-2, c_{8j+2}, c_{8j+3}, c_{8j+4}, 0)&\text{ if }\ell = 8j-3\\
      (\underbrace{0,\ldots, 0}_{\ell+1\ copies}, 2, \underbrace{0,\ldots, 0}_{8j-4-\ell\ copies},-2,2, b_{8j+1}-2, c_{8j+2}, c_{8j+3}, c_{8j+4}, 0)&\text{ if }\ell \leq 8j-4
\end{cases}
\end{align*}
If $\ell = 8j-3$ we are done. If $\ell = 8j-4$, then by applying a type $-2$ move on the leftmost $-2$, we get:
\begin{align*}
    (\underbrace{0,\ldots, 0}_{8j-2\ copies},2,0, b_{8j+1}-2, c_{8j+2}, c_{8j+3}, c_{8j+4}, -2)
\end{align*}
Hence, we are back to case 1. Therefore, we may assume $\ell\leq 8j-5$. We now continue as follows:

\item Leftward $(-2,2)$-slide:
\begin{align*}
    (\underbrace{0,\ldots, 0}_{\ell+1\ copies}, 2,-2,2 \underbrace{0,\ldots, 0}_{8j-4-\ell\ copies}, b_{8j+1}-2, c_{8j+2}, c_{8j+3}, c_{8j+4}, -2)
\end{align*}
Note, the rightmost entry of the vector now changes back to $-2$. 

\item Type $-2$ move on the leftmost $-2$:
\begin{align*}
    (\underbrace{0,\ldots, 0}_{\ell+2\ copies},2, \underbrace{0,\ldots, 0}_{8j-3-\ell\ copies}, b_{8j+1}-2, c_{8j+2}, c_{8j+3}, c_{8j+4}, -2)
\end{align*}
\end{enumerate}

We are now back to the vector we started with at the beginning of case 2, except that the $+2$ entry of the $\Lambda'_{j}$-subvector has shifted two positions to the right. Therefore, we can iterate this process until $\ell=8j-3$ or $8j-4$, and we have already dealt with both of those cases.  

\underline{Case 3:} Suppose $d_{8j+5}=0$ and $\ell = 8j-2$. Add $2PD[v_{8j-1}]$ to get the equivalent vector:
\begin{align*}
    (\underbrace{0,\ldots, 0}_{8j-3\ copies}, 2,-2,2,b_{8j+1}, c_{8j+2}, c_{8j+3}, c_{8j+4}, 2)
\end{align*}
Now add $2PD[v_{8j+5}]$ to get:
\begin{align*}
    (\underbrace{0,\ldots, 0}_{8j-3\ copies}, 2,0,2,b_{8j+1}, c_{8j+2}, c_{8j+3}, c_{8j+4}, -2)
\end{align*}
This vector violates Lemma \ref{lem: alternating} and hence cannot be a representative of $\mathcal{L}(\Gamma'_{j}, [t])$. 

\underline{Case 4:} Suppose $d_{8j+5}=0$ and $\ell \leq 8j-3$, so that we start with a vector of the form:
\begin{align*}
    (\underbrace{0,\ldots, 0}_{\ell\ copies},2,\underbrace{0,\ldots, 0}_{8j-1-\ell\ copies},b_{8j+1}, c_{8j+2}, c_{8j+3}, c_{8j+4}, 0)
\end{align*}
Now consider the following sequence of moves:
\begin{enumerate}
    \item Type $+2$ move:
    \begin{align*}
        (\underbrace{0,\ldots, 0}_{\ell-1\ copies},2,-2,2\underbrace{0,\ldots, 0}_{8j-2-\ell\ copies},b_{8j+1}, c_{8j+2}, c_{8j+3}, c_{8j+4}, 0)
    \end{align*}

\item Rightward $(-2,2)$-slide:
\begin{align*}
    (\underbrace{0,\ldots, 0}_{\ell-1\ copies}, 2,\underbrace{0,\ldots, 0}_{8j-3-\ell\ copies}, -2,2,0,b_{8j+1}, c_{8j+2}, c_{8j+3}, c_{8j+4}, 0)
\end{align*}

\item Add $2PD[v_{8j-1}]$:
\begin{align*}
    (\underbrace{0,\ldots, 0}_{\ell-1\ copies}, 2,\underbrace{0,\ldots, 0}_{8j-2-\ell\ copies},-2,2, c_{8j+2}, c_{8j+3}, c_{8j+4}, 2)
\end{align*}
\item Add $2PD[v_{8j+5}]$:
\begin{align*}
    (\underbrace{0,\ldots, 0}_{\ell-1\ copies}, 2,\underbrace{0,\ldots, 0}_{8j-1-\ell\ copies},2, c_{8j+2}, c_{8j+3}, c_{8j+4}, -2)
\end{align*}
\end{enumerate}
Again, this vector violates Lemma \ref{lem: alternating} and hence cannot be a representative of $\mathcal{L}(\Gamma'_{j}, [t])$. 
\end{proof} 

\begin{prop}
If $t'\in [t]$ represents an element of $\mathcal{L}(\Gamma'_{j}, [t])$, then $t'$ is equivalent to
\begin{align*}
    t_{1}=(0,\ldots, 0, -1, 0, 2, 0, 0)\hspace{2em}\text{or}\hspace{2em}t_{2}=(0, \ldots, 0, 2, -1, 0, 0,0, -2)
\end{align*}
\end{prop}

\begin{proof}
Suppose $t'\in [k]$ represents an element of $\mathcal{L}(\Gamma'_{j}, [t])$. By the previous lemma, we can assume 
\begin{align*}
    t' = (0, \ldots, 0, a_{8j}, b_{8j+1}, c_{8j+2}, c_{8j+3}, c_{8j+4}, d_{8j+5})
\end{align*}
where $a_{8j}\in \{0,2\}$, $b_{8j+1}\in \{-2j-1, -2j+1, \ldots, 2j-1, 2j+1\}$, $c_{8j+2}, c_{8j+3}, c_{8j+4}\in \{-2, 0, 2\}$, $d_{8j+5}\in \{-2, 0, 2\}$. 

Since we are assuming $t'$ represents an element of $\mathcal{L}(\Gamma'_{j}, [t])$, we must have: 
\begin{align}
    0 = t'\cdot x' = (8j+1)a_{8j} + (8j-1)d_{8j+5}+ 4b_{8j+1}+3c_{8j+2}+2c_{8j+3} + c_{8j+4}\label{eq: orthogonal condition}
\end{align}

By Lemmas \ref{lem: sparse vector} and \ref{lem: one 2}, we can assume $(c_{8j+2}, c_{8j+3}, c_{8j+4})$ is the zero vector or has exactly one non-zero entry equal to $+2$. In particular, we can assume 
\begin{align*}
    3c_{8j+2}+2c_{8j+3}+c_{8j+4}\in \{0, 2, 4, 6\}
\end{align*}
Note the moves required to put the subvector $(c_{8j+2}, c_{8j+3}, c_{8j+4})$ into this form only effect the entry $b_{8j+1}$ and leave all of the others unchanged. 

Now suppose $a_{8j}=2$ and $b_{8j+1}= 2j+1$. Then by adding $2PD[v_{8j+1}]$ we would obtain an equivalent vector with $a_{8j}=4$. But this violates \hyperref[star]{$\star$}. Hence, if $a_{8j}=2$, we can assume $b_{8j+1}\leq 2j-1$. 

Now suppose $a_{8j}=0$ and $b_{8j+1}= 2j+1$. If $(c_{8j+2}, c_{8j+3}, c_{8j+4})$ is not the zero vector, but rather a vector with precisely one non-zero entry equal to $+2$, then by applying the move in Lemma \ref{lem: one 2} and taking into account Remark \ref{rem:adj vertex change}, we would obtain an equivalent vector with $b_{8j+1}= 2j+3$, which violates $\star$. Therefore, if $a_{8j}=0$ and $b_{8j+1}= 2j+1$, we must have that $(c_{8j+2}, c_{8j+3}, c_{8j+4})$ is the zero vector. Plugging this into equation \ref{eq: orthogonal condition} yields:
\begin{align*}
    (8j-1)d_{8j+5} = -8j-4
\end{align*}
This clearly has no solutions with the given constraints. Therefore, we can assume $b_{8j+1}\leq 2j-1$, regardless of whether $a_{8j}=0$ or $2$. In particular, 
\begin{align*}
    -8j-4\leq 4b_{8j+1}+3c_{8j+2}+2c_{8j+3}+c_{8j+4}\leq 8j+2
\end{align*}

\noindent \underline{Case 1:} Suppose $a_{8j}=0$ and $d_{8j+5}=-2$. Then:
\begin{align*}
    0=t'\cdot x' = -16j+2+ 4b_{8j+1}+3c_{8j+2}+2c_{8j+3} + c_{8j+4}\leq -8j+4<0
\end{align*}
which is a contradiction. 
\\

\noindent \underline{Case 2:} Suppose $a_{8j}=0$ and $d_{8j+5}=0$. Then:
\begin{align*}
    0=t'\cdot x' = 4b_{8j+1}+3c_{8j+2}+2c_{8j+3} + c_{8j+4}
\end{align*}
The only solution to this equation given the constraints we have established is
\begin{align*}
    (b_{8j+1}, c_{8j+2}, c_{8j+3}, c_{8j+4}) = (-1, 0,2,0)
\end{align*}
which corresponds to $t_{1}$. 
\\

\noindent \underline{Case 3:} Suppose $a_{8j}=0$ and $d_{8j+5}=2$. Then:
\begin{align*}
    0=t'\cdot x' = 16j-2+ 4b_{8j+1}+3c_{8j+2}+2c_{8j+3} + c_{8j+4}\geq 8j-6>0 
\end{align*}
which again is a contradiction. 
\\

\noindent \underline{Case 4:} Suppose $a_{8j}=2$ and $d_{8j+5}=-2$. Then: 
\begin{align*}
    0=t'\cdot x' = 4+4b_{8j+1}+3c_{8j+2}+2c_{8j+3} + c_{8j+4}
\end{align*}
The only solution to this equation given the constraints we have established is 
\begin{align*}
    (b_{8j+1}, c_{8j+2}, c_{8j+3}, c_{8j+4}) = (-1, 0,0,0)
\end{align*}
which corresponds to $t_{2}$.
\\

\noindent \underline{Case 5:} Suppose $a_{8j}=2$ and $d_{8j+5}=0$. Then: 
\begin{align*}
   0=t'\cdot x' = 16j+2+4b_{8j+1}+3c_{8j+2}+2c_{8j+3} + c_{8j+4}\geq 8j-2>0
\end{align*}
which again is a contradiction. Finally, 
\\

\noindent \underline{Case 6:} Suppose $a_{8j}=2$ and $d_{8j+5}=2$. This case is ruled out by Lemma \ref{lem: alternating}.
\end{proof}

Again, we have not yet proved that $t_{1}$ and $t_{2}$ represent different elements of $\mathcal{L}(\Gamma'_{j}, [t])$, however, we do have:
\begin{cor}\label{cor: d inv 2}
$d_{1/2}(N_{j},\mathfrak{s}_{0}) = -2j+\frac{1}{2}$
\end{cor}

\begin{proof}
Let
\begin{align*}
    \beta_{1} &= (2, 4, 6, \ldots, 16j-2, 8j+1, 4,2,0, 0,8j-1)\\
    \beta_{2} &= (0, \ldots, 0, -1,0, 0,0,0,1)
\end{align*}
Then, $\beta_{1}B'_{j}= t_{1}$ and $\beta_{2}B'_{j}= t_{2}$. Thus, 
\begin{align*}
    t_{1}^2 &= t_{1}\cdot \beta_{1} = -4\\
    t_{2}^2 &= t_{2}\cdot \beta_{2} = -4
\end{align*}
Hence, under the isomorphism in Corollary \ref{cor: HF and lattice identification}, the elements of $HF^+(N_{j}, \mathfrak{s}_{0})$ corresponding to $t_{1}$ and $t_{2}$ have gradings:
\begin{align*}
    \gr(t_{1}) = \gr(t_{2}) &= -\frac{t_{2}^2 + |\mathcal{V}(\Gamma'_{j})|-3}{4}\\
               &= -\frac{-4+8j+5-3}{4}\\
               &= -2j+\frac{1}{2}
\end{align*}
\end{proof}

Now combining Corollaries \ref{cor: d inv 1}, \ref{cor: d inv 2}, and the basic fact that $d_{\pm 1/2}(-Y)=-d_{\mp 1/2}(Y)$ (see \cite[Proposition 4.10]{MR1957829}), we have:
\begin{align*}
    d_{1/2}(-N_{j}) = \frac{1}{2}\hspace{2em}\text{and}\hspace{2em}d_{-1/2}(-N_{j}) = 2j-\frac{1}{2}
\end{align*}
In particular, by Theorem \ref{thm: rustamov}, $HF^+_{even}(-N_{j}, \mathfrak{s}_{0}) = \mathcal{T}^+_{2j-1/2}$. 

\newpage 

We have yet to completely determine $HF^+_{odd}(-N_{j},\mathfrak{s}_{0})$. So far, from Proposition \ref{prop: -N reps}, we know that $\dim_{\F}[\ker(U)\cap HF^+_{odd}(-N_{j},\mathfrak{s}_{0})] = 1$ or $2$ depending on whether $k_{1}$ and $k_{2}$ represent the same element or not in $\mathcal{L}(\Gamma_{j},[k])$. Therefore, as graded $\F[U]$-modules, we have:
\begin{figure}[h]
  \centering
  \includegraphics[scale=.95]{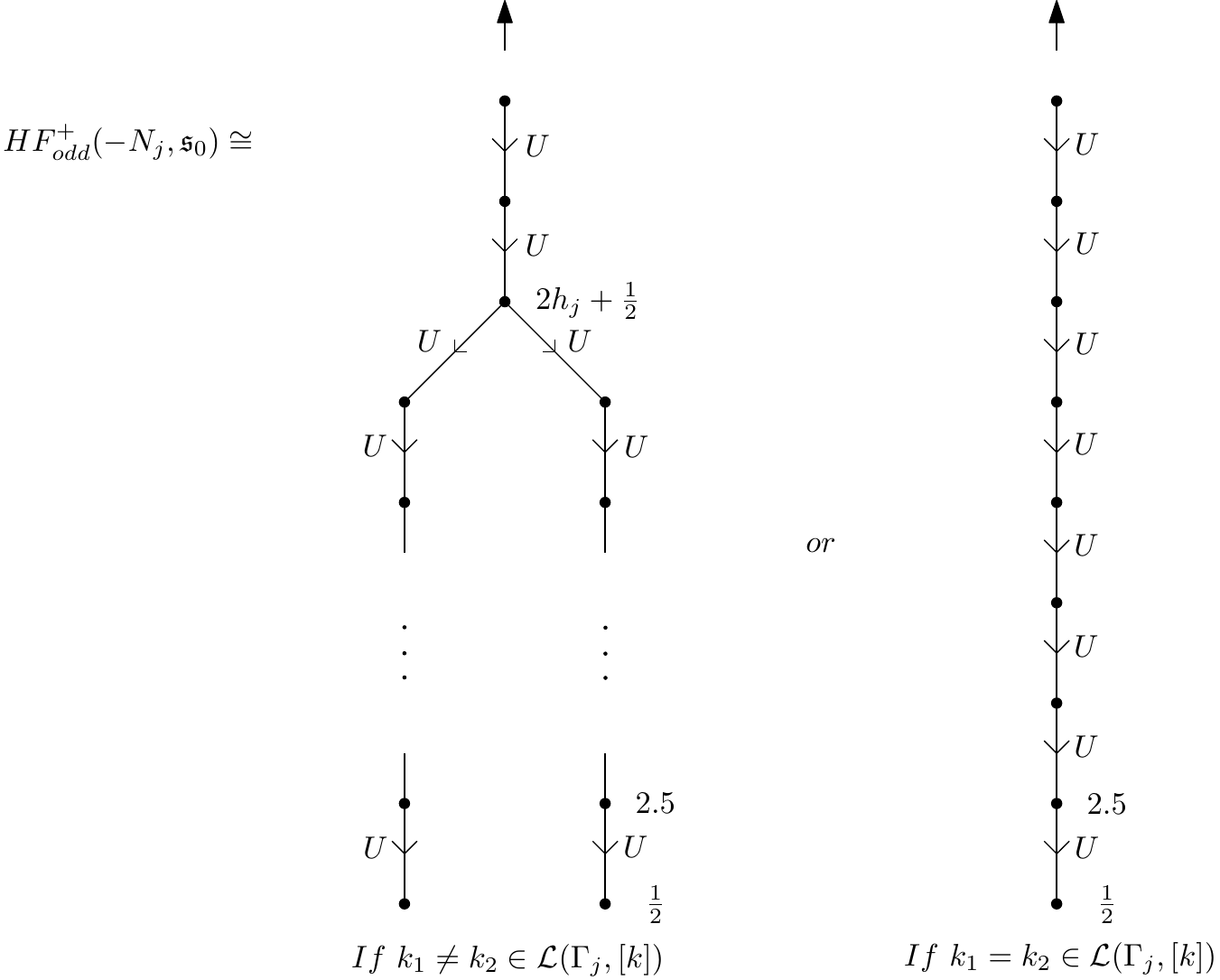}
  \end{figure}

\noindent Here, $h_{j}$ is some positive integer depending on $j$ which we have not yet determined. 

A word of explanation is in order since on the left side of the above isomorphism we have an $\F[U]$-module and on the right we have one of two possible graphs. The right side is to be interpreted as follows: 
\begin{itemize}
    \item Each vertex at grading $r$ corresponds to a basis element of the $\F$-vector space $HF_{r}^+(-N_{j},\mathfrak{s}_{0})$.
    \item If the edges emanating from a vertex $y$ are of the form:
    \begin{figure}[h]
  \centering
  \includegraphics[scale=1]{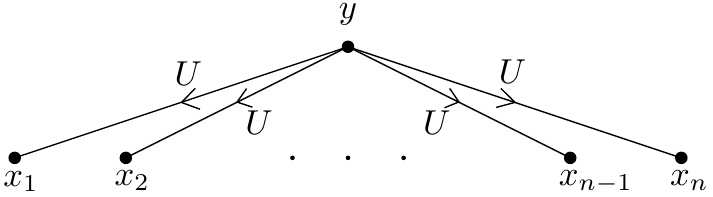}
  \end{figure}
  
  \noindent then $Uy = x_{1}+x_{2}+\cdots +x_{n-1}+x_{n}$. In particular, if there are no edges emanating from $y$, then $Uy = 0$. 
\end{itemize}

\newpage

We now utilize Turaev Torsion to complete step (\ref{step 1}). Combining our computations thus far with \cite[Theorem 10.17]{MR2113020}, we see that
\begin{align*}
    T_{N_{j}}(\mathfrak{s}_{0}) = \begin{cases}
    h_{j}+j&\text{ if }k_{1}\neq k_{2}\in \mathcal{L}(\Gamma_{j}, [k])\\
    j&\text{ otherwise}
    \end{cases}
\end{align*}
where $T_{N_{j}}$ is the Turaev torsion function associated to $N_{j}$ (see \cite[p. 119]{MR1958479}). Therefore, to precisely determine $HF^+_{odd}(-N_{j},\mathfrak{s}_{0})$, it suffices to compute $T_{N_{j}}(\mathfrak{s}_{0})$. 

There are many standard ways to compute $T_{N_{j}}(\mathfrak{s}_{0})$. For example, in \cite{MR1958479}, Turaev provides a formula in terms of a surgery description. We will now give a brief outline of how to carry out the calculation using this method, but we leave the details to the reader. 

\begin{enumerate}
    \item Let $H = H_{1}(N_{j};\Z)$. Consider the group ring $\Z[H]$. Since $H\cong \Z$, we can think of $\Z[H]$ as a $\Z[t, t^{-1}]$, the ring of Laurent polynomials in the indeterminate $t$. Let $Q(H)$ denote the field of fractions of $\Z[H]$. The first step is to compute the Turaev torsion $\tau(N_{j}, \mathfrak{s}_{0})\in Q(H)$. For this, we use the formula given in \cite[VII.2, Theorem 2.2]{MR1958479}. To apply this formula, we need to choose a surgery diagram for $N_{j}$ and orient the underlying link. We use the following surgery diagram with underlying link $L_{j}$ oriented as indicated by the arrows:
    \begin{figure}[h]
  \centering
  \includegraphics[scale=.75]{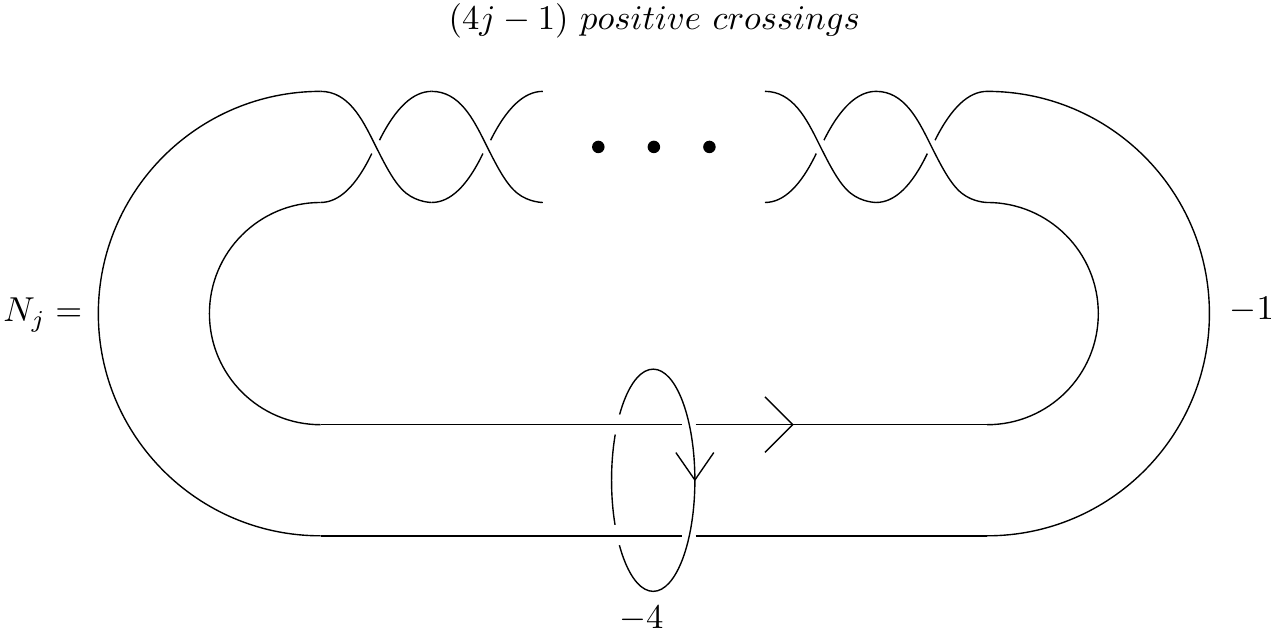}
  \end{figure}
    
    \noindent The bulk of the work in computing $\tau(N_{j}, \mathfrak{s}_{0})$ using \cite[VII.2, Theorem 2.2]{MR1958479} is calculating the multivariable Alexander-Conway function $\nabla(L_{j})$. Again, there are various approaches to computing $\nabla(L_{j})$. For example, in \cite{MR1197048} Murakami provides a skein formula for $\nabla$. Using this formula, we find that $\nabla(L_{j}) = yx^{4j-1}+y^{-1}x^{-4j+1}$ where the variable $x$ corresponds to the torus knot component and the variable $y$ corresponds to the unknot component. Plugging this into the formula for $\tau(N_{j}, \mathfrak{s}_{0})$, we get:
    \begin{align*}
        \tau(N_{j}, \mathfrak{s}_{0})=\frac{t^{8j-1}+1}{t^{4j-2}(t-1)^2(t+1)}
    \end{align*}
    
    \item Next, we compute $[\tau(N_{j}, \mathfrak{s}_{0})]$ which is a Laurent polynomial obtained by truncating $\tau(N_{j}, \mathfrak{s}_{0})$ in a certain way (see \cite[p.22]{MR1958479}). We find that:
    \begin{align*}
        [\tau(N_{j}, \mathfrak{s}_{0})] &= \frac{t^{8j-1}+1}{t^{4j-2}(t-1)^2(t+1)} - \frac{t}{(t-1)^2}\\
        &=\left(\sum\limits_{i=0}^{4j-4}t^{i-4j+2}\right)\left(\sum\limits_{i=1}^{2j-1}t^{2i}\right) + \sum\limits_{i=0}^{8j-4}t^{i-4j+2}\\
        &=2j + \text{non-constant terms}
    \end{align*}
    \item By definition, $T_{N_{j}}(\mathfrak{s}_{0})$ is the constant term of $[\tau(N_{j}, \mathfrak{s}_{0})]$. Hence, $T_{N_{j}}(\mathfrak{s}_{0}) = 2j$. 
\end{enumerate}
\vspace{1em}

Thus, we have the following isomorphism of graded $\F[U]$-modules:
\begin{figure}[h]
  \centering
  \includegraphics[scale=1]{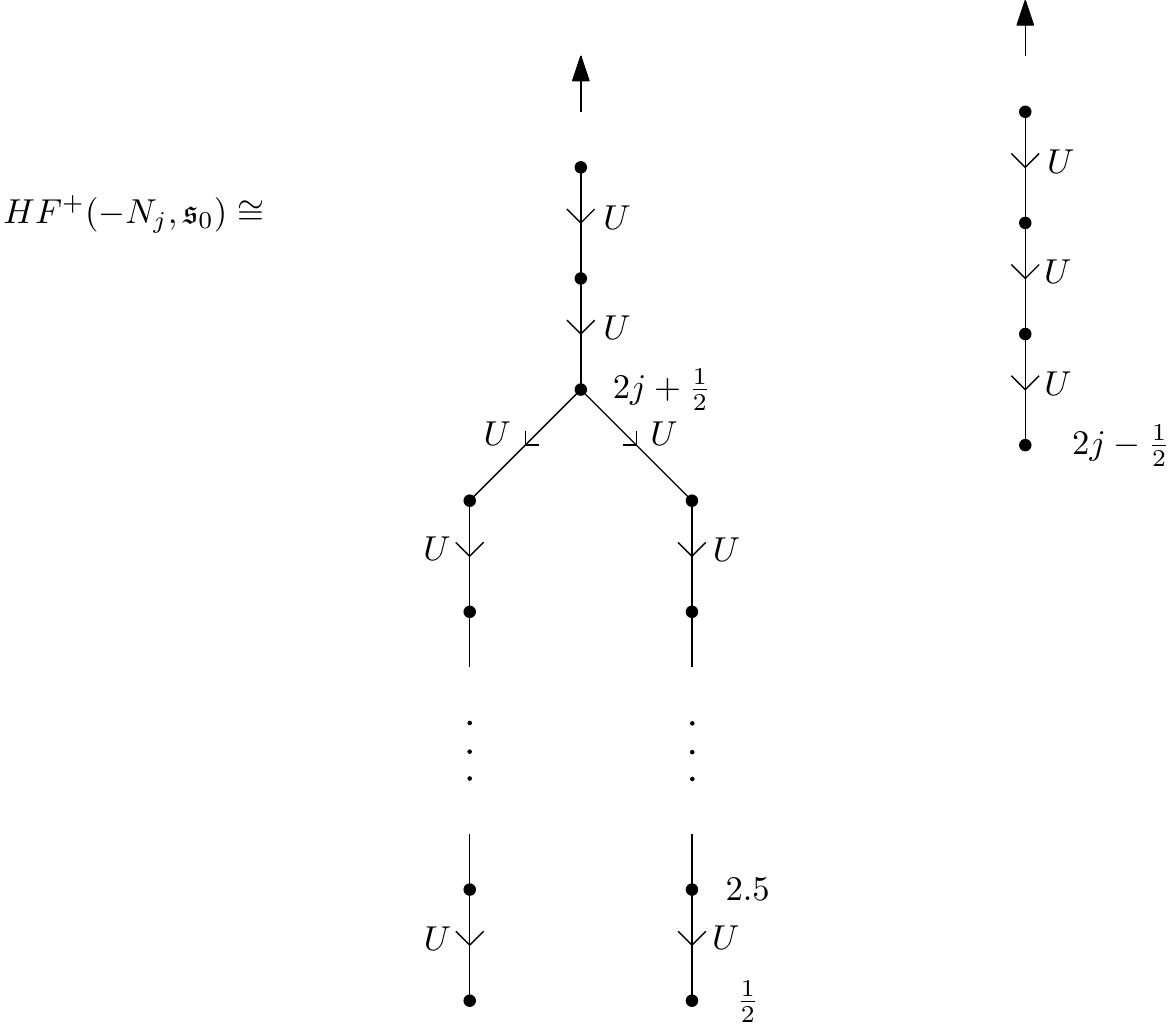}
  \end{figure}
  
\subsubsection{Step 2} We now compute the involution $\iota_{*}$ on homology. This amounts to determining whether $-k_{1}$ is equivalent to $k_{1}$ or $k_{2}$. If $-k_{1}$ is equivalent to $k_{2}$, then the involution swaps the two legs of the left-hand graph of the above figure and leaves the right-hand graph fixed. If $-k_{1}$ is equivalent to $k_{1}$, then $\iota_*$ is the identity. We know show that, in fact, $-k_{1}$ is equivalent to $k_{2}$. 

Recall, $-k_{1} = (1, 0, -5+4j, -3,0,\ldots, 0, 3)$ and $k_{2} = (-1, 0, 3-4j,3, 0,\ldots, 0, 1)$. Consider the following sequence of moves from $-k_{1}$ to $k_{2}$:
\begin{enumerate}
    \item Add $2PD[v_{4}]$:
    \begin{align*}
    \begin{cases}
          (-1, 0, -1, 3, 1) = k_{2}&\text{ if } j =1\\
          (-1,0,-5+4j,3,-2,\underbrace{0,\ldots, 0}_{2j-3}, 3)&\text{ if }j\geq 2
    \end{cases}
    \end{align*}
    So we can assume for the subsequent moves that $j\geq 2$. 
    \item Apply Lemma \ref{lem: one 2} and Remark \ref{rem:adj vertex change}:
    \begin{align*}
        (-1, 0, -5+4j, 1, \underbrace{0, \ldots, 0}_{2j-3}, 2, 1)
    \end{align*}
    \item Add $-2PD[v_{1}]$:
    \begin{align*}
        (1, -2, -5+4j-2, -1, \underbrace{0, \ldots, 0}_{2j-3}, 2, 1)
    \end{align*}
    \item Add $-2PD[v_{2}]$:
    \begin{align*}
        (-1, 2, -5+4j-2, -1, \underbrace{0, \ldots, 0}_{2j-3}, 2, 1)
    \end{align*}
    \item Add $2PD[v_{1}]$:
    \begin{align*}
        (1, 0, -5+4j-4, -3, \underbrace{0, \ldots, 0}_{2j-3}, 2, 1)
    \end{align*}
    \item Add $-2PD[v_{4}]$:
    \begin{align*}
         (-1,0, -5 + 4j -4, 3, -2, \underbrace{0, \ldots, 0}_{2j-4}, 2, 1)
    \end{align*}
    \item Add $-2PD[v_{5}]$:
    \begin{align*}
         (-1,0, -5 + 4j -4, 1, 2,-2 \underbrace{0, \ldots, 0}_{2j-5}, 2, 1)
    \end{align*}
    
    \item Rightward $(2,-2)$-slide:
    \begin{align*}
         (-1,0, -5 + 4j -4, 1, \underbrace{0, \ldots, 0}_{2j-5},2,-2 2, 1)
    \end{align*}
    
    \item Type $-2$ move:
    \begin{align*}
        (-1, 0, -5+4j-4, 1, \underbrace{0, \ldots, 0}_{2j-4}, 2, 0, 1)
    \end{align*}
\end{enumerate}
Now notice that we are back to the same vector as in (2), except we have decreased the 3rd entry by 4 and shifted the $+2$ entry one slot to the left. Therefore, if we iterate this sequence of moves $(2j-4)$-more times, we get the vector:
\begin{align*}
    (-1, 0, 7-4j, 1,2,\underbrace{0, \ldots, 0}_{2j-3},1)
\end{align*}
\newpage

Now consider the sequence of moves:
\begin{enumerate}
    \item Add $-2PD[v_{1}]$:
    \begin{align*}
        (1, -2, 5-4j, -1,2,\underbrace{0, \ldots, 0}_{2j-3},1)
    \end{align*}
    \item Add $-2PD[v_{2}]$:
    \begin{align*}
        (-1, 2, 5-4j, -1,2,\underbrace{0, \ldots, 0}_{2j-3},1)
    \end{align*}
    \item Add $-2PD[v_{1}]$:
    \begin{align*}
        (1, 0, 3-4j, -3,2,\underbrace{0, \ldots, 0}_{2j-3},1)
    \end{align*}
    \item Add $-2PD[v_{4}]$:
    \begin{align*}
        (-1, 0, 3-4j, 3, 0,\ldots, 0, 1) = k_{2}
    \end{align*}
\end{enumerate}

\subsubsection{Step 3}

\begin{thm}\label{thm: HFI N}
We have the following isomorphism of graded $\F[U,Q]/(Q^2)$-modules:
\end{thm}
\begin{figure}[h]
  \centering
  \includegraphics[scale=1]{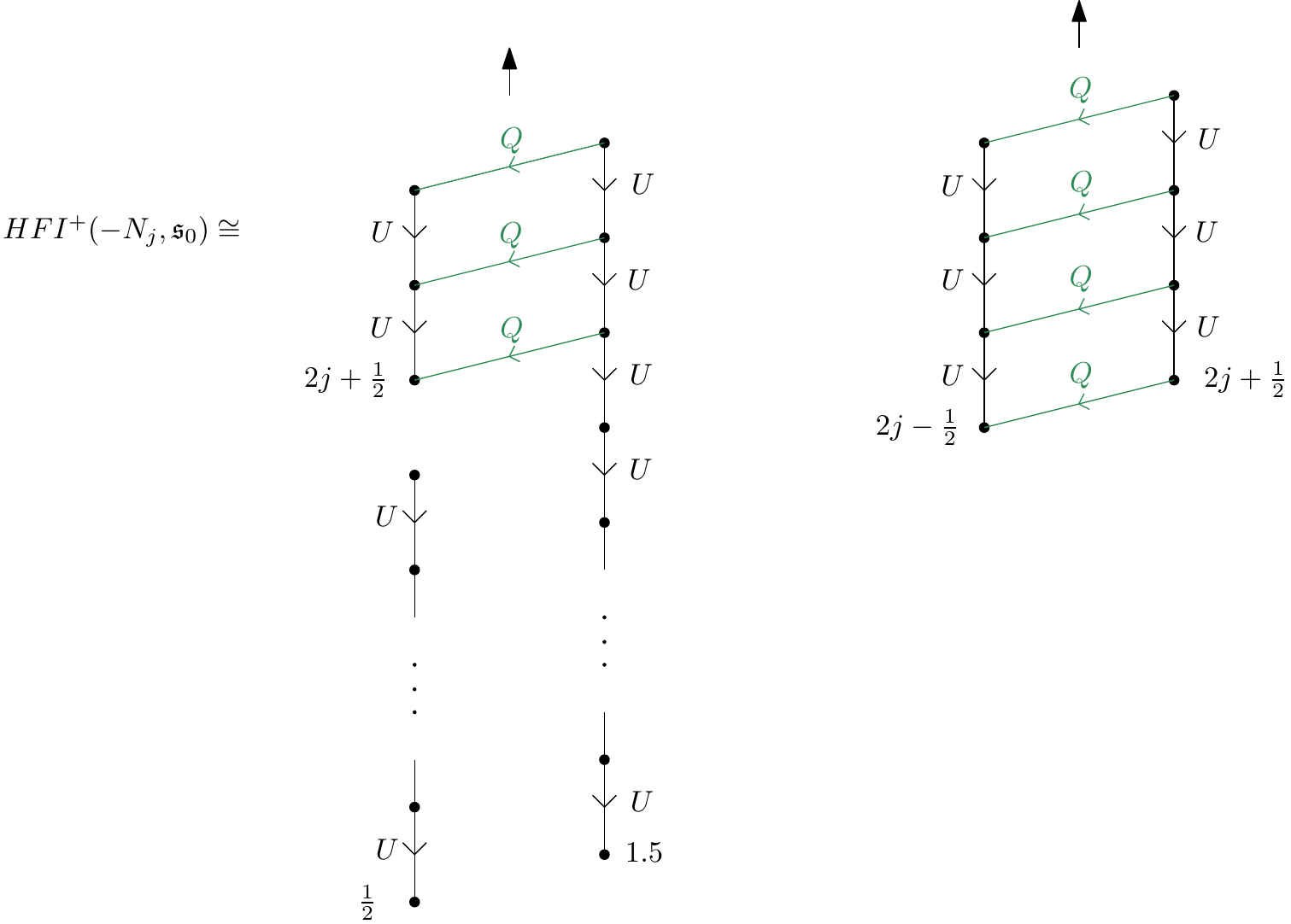}
  \end{figure}

\begin{rem}
The graph on the right-hand side of the above isomorphism should be interpreted as a graded $\F[U,Q]/(Q^2)$-module in a manner similar to what was described earlier in the context of $\F[U]$-modules, except now there are additional arrows labeled with $Q$ to indicate the action of $Q$. 
\end{rem}

\begin{proof}
For simplicity of exposition, we prove the statement for $j=1$. The proof for $j\geq 2$ is completely analogous and is left to the reader. 

Fix an admissible Heegaard pair $\mathcal{H} = (H, J)$ for $(-N_{1}, \mathfrak{s}_{0})$. We can choose representative cycles $a,b,c\in CF^+(\mathcal{H}, \mathfrak{s}_{0})$ such that:
\begin{align*}
    [a+b], [c]\in \Image [\pi_{*}: HF^{\infty}(\mathcal{H}, \mathfrak{s}_{0})\to HF^+(\mathcal{H}, \mathfrak{s}_{0})]
\end{align*}
and the corresponding $HF^+$ homology generators are:

\begin{figure}[h]
  \centering
  \includegraphics[scale=1]{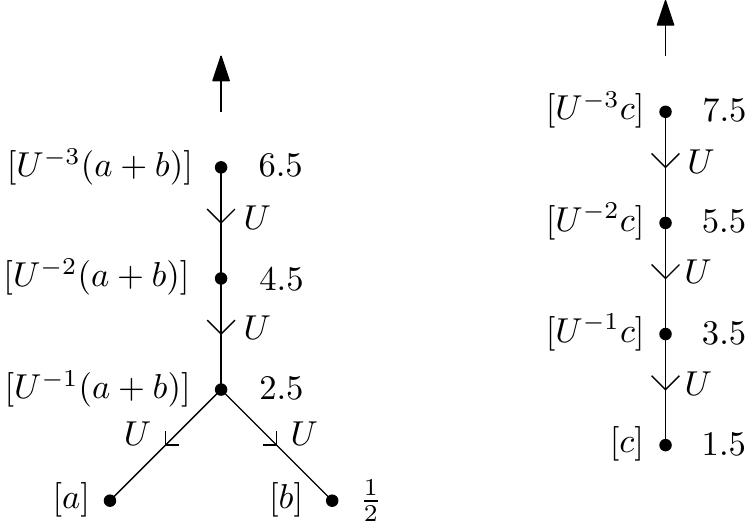}
  \end{figure}
  
\noindent Since $\iota_{*}([a]) = \iota_{*}([b])$, we have that $(1+\iota_{*})([a+b]) = 0$. Therefore, there exists some $d\in CF^+(\mathcal{H}, \mathfrak{s}_{0})$ such that $\partial d = a+b+\iota(a+b) $. Similarly, since $(1+\iota_{*})([c]) = 0$, there exists some $e\in CF^+(\mathcal{H}, \mathfrak{s}_{0})$ such that $\partial e = c+\iota(c)$. It then follows from Proposition \ref{prop: HFI exact triangle} and step 3 of section \ref{subsection: comp strategy}, that as graded $\F$-vector spaces we have $HFI^+(-N_{1};\mathfrak{s}_{0})\cong$

\begin{figure}[h]
  \centering
  \includegraphics[scale=1]{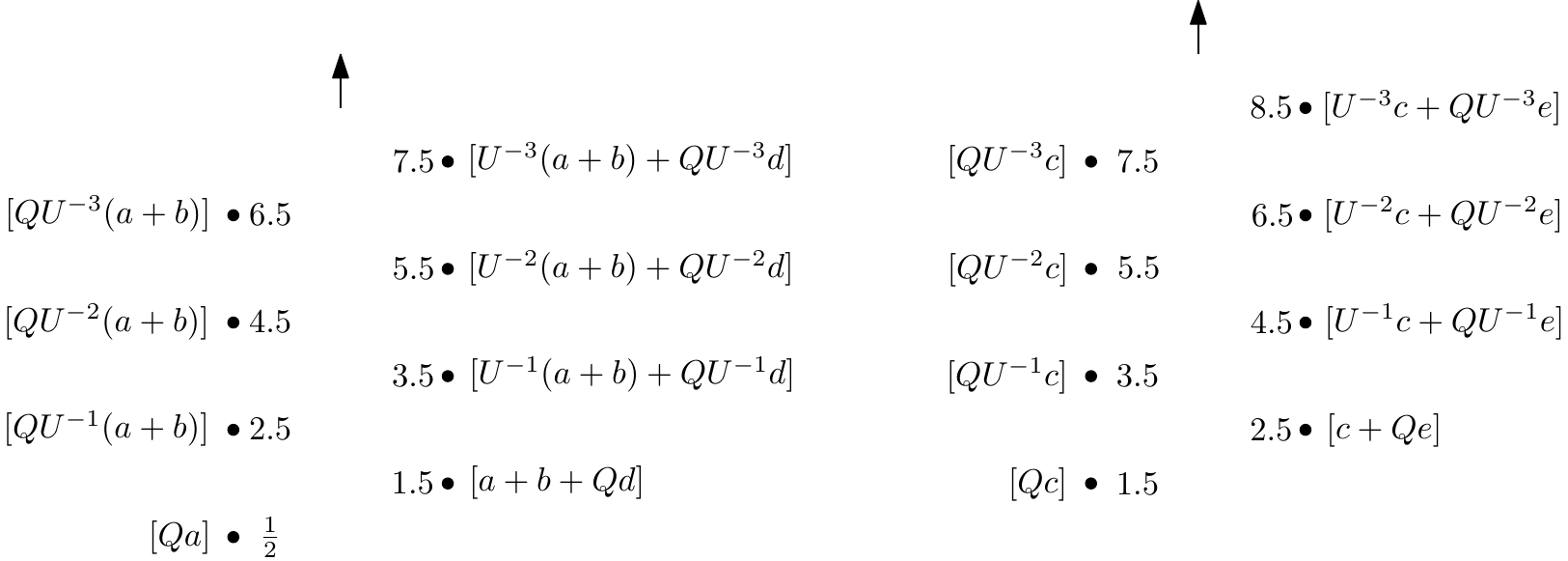}
  \end{figure}
  
\newpage
From this explicit description of generators, we see that for $n\geq 2$:
\begin{align*}
    U\cdot [QU^{-n}(a+b)] = [QU^{-n+1}(a+b)]
\end{align*}
and for $n\geq 1$: 
\begin{align*}
    U\cdot [U^{-n}(a+b) + QU^{-n}d] &= [U^{-n+1}(a+b) + QU^{-n+1}d]\\
    U\cdot [QU^{-n}c] &= [QU^{-n+1}c]\\
    U\cdot [U^{-n}c+QU^{-n}e]&=[U^{-n+1}c+QU^{-n+1}e]
\end{align*}
Next, we have:
\begin{align*}
    U\cdot [QU^{-1}(a+b)] = [Q(a+b)] = [\partial^I a] = 0
\end{align*}
Moreover, by grading considerations, we must have:
\begin{align*}
    U\cdot[Qa] &= 0\\
    U\cdot [a+b+Qd]&= 0\\
    U\cdot [Qc] &= 0
\end{align*}
Also, either $U\cdot [c+Qe] = 0$ or $U\cdot [c+Qe] = [Qa]$. In the former case, we would have:
\begin{align*}
    \dim_{\F}[\ker (U:HFI^+(-N_{j}, \mathfrak{s}_{0})\to HFI^+(-N_{j}, \mathfrak{s}_{0}))] &= 5\\
    \dim_{\F}[\coker (U:HFI^+(-N_{j}, \mathfrak{s}_{0})\to HFI^+(-N_{j}, \mathfrak{s}_{0}))] &= 1
\end{align*}
whereas in latter we would have:
\begin{align*}
    \dim_{\F}[\ker (U:HFI^+(-N_{j}, \mathfrak{s}_{0})\to HFI^+(-N_{j}, \mathfrak{s}_{0})] &= 4\\
    \dim_{\F}[\coker (U:HFI^+(-N_{j}, \mathfrak{s}_{0})\to HFI^+(-N_{j}, \mathfrak{s}_{0}))] &= 0
\end{align*}
Thus, by \cite[Proposition 4.1]{MR3649355}, we would have either:
\begin{align*}
    \dim_{\F}(\widehat{HFI}(-N_{j}, \mathfrak{s}_{0})) = 6\text{ or } 4
\end{align*}
But by \cite[Corollary 4.7]{MR3649355} we see that:
\begin{align*}
    \dim_{\F}(\widehat{HFI}(-N_{j}, \mathfrak{s}_{0})) &=\dim_{\F}[\ker(Q(1+\iota_{*}):\widehat{HF}(-N_{j}, \mathfrak{s}_{0})\to Q\cdot \widehat{HF}(-N_{j}, \mathfrak{s}_{0}))]\\
    &+\dim_{\F}[\coker(Q(1+\iota_{*}):\widehat{HF}(-N_{j}, \mathfrak{s}_{0})\to Q\cdot \widehat{HF}(-N_{j}, \mathfrak{s}_{0}))]\\
    &=3+3\\
    &=6
\end{align*}
Hence, we must have had $U\cdot[c+Qe]=0$. We have now completely determined the $U$-action on $HFI^+(-N_{j}, \mathfrak{s}_{0})$. 

Next, for the $Q$-action, it follows from the explicit description of the generators that for $n\geq 1$:
\begin{align*}
    Q\cdot [U^{-n}(a+b)+QU^{-1}d] &= [QU^{-n}(a+b)]
\end{align*}
and for $n\geq 0$:
\begin{align*}
    Q\cdot [U^{-n}c + QU^{-n}e] = [QU^{-n}c]
\end{align*}
Also, 
\begin{align*}
    Q\cdot [a+b+Qd] = [Q(a+b)] = [\partial^I a] = 0
\end{align*}
It is clear that the action of $Q$ on all of the other generators is zero. Thus, we have:

\begin{figure}[h]
  \centering
  \includegraphics[scale=1]{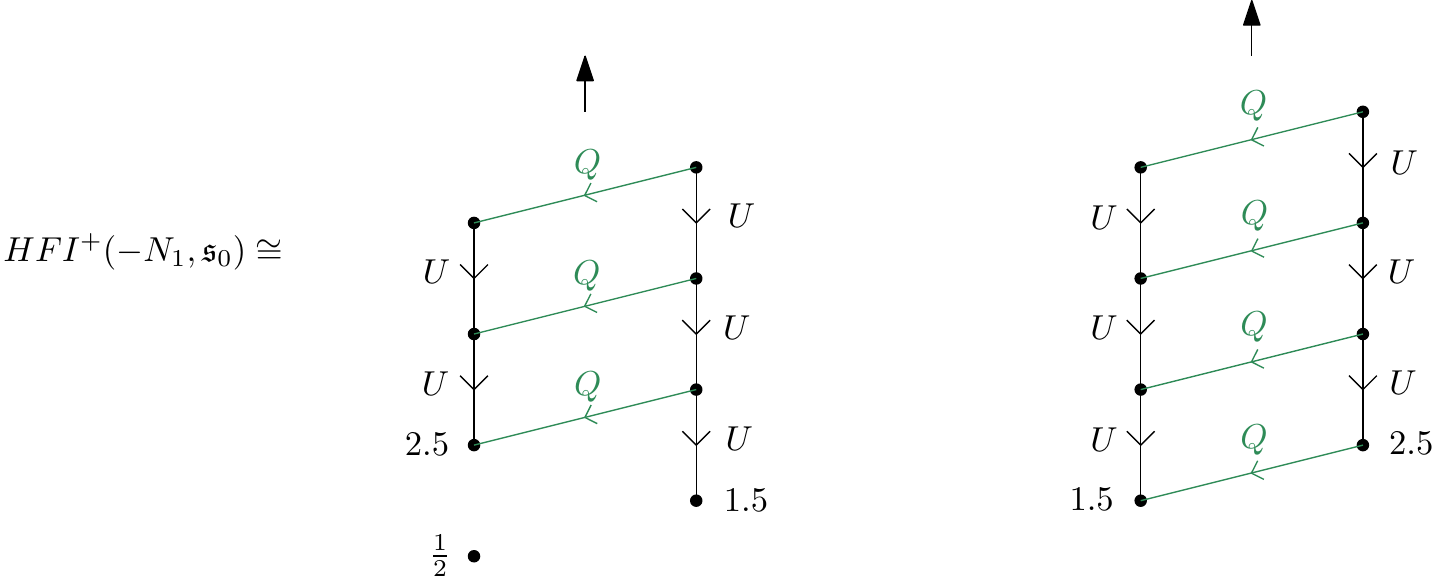}
  \end{figure}
\end{proof}

\begin{thm}\label{thm: nj not zero surgery}
For all positive integers $j$, $N_{j}$ cannot be obtained by $0$-surgery on a knot in $S^3$. In fact, $N_{j}$ is not the oriented boundary any smooth negative semi-definite spin 4-manifold.
\end{thm}

\begin{proof}
From previous theorem, we have: 
\begin{align*}
    \bar{d}_{1/2}(-N_{j}) = 2j+\frac{1}{2}&\hspace{2em}\bar{d}_{-1/2}(-N_{j}) = 2j-\frac{1}{2}\\
    \underline{d}_{1/2}(-N_{j}) = \frac{1}{2}&\hspace{2em}  \underline{d}_{-1/2}(-N_{j}) = 2j-\frac{1}{2}
\end{align*}
Equivalently, 
\begin{align*}
\underline{d}_{-1/2}(N_{j}) = -2j-1/2&\hspace{2em}\underline{d}_{1/2}(N_{j}) = -2j+\frac{1}{2}\\
\bar{d}_{-1/2}(N_{j}) =-\frac{1}{2}&\hspace{2em}\bar{d}_{1/2}(N_{j}) = -2j+\frac{1}{2}
\end{align*}
The conclusion now follows immediately from Corollaries \ref{cor: obstruction to any negative semi-definite} and \ref{cor:zero surgery obstruction}. 

\end{proof}

\subsection{$HFI^+(-S^3_{0}(K_{1}),\mathfrak{s}_{0})$}\label{subsection: comparison}
As mentioned in the introduction, Ichihara, Motegi, and Song discovered an infinite family of hyperbolic knots which admit small Seifert fibered $0$-surgery (see \cite{MR2499823}). In particular, for the knot $K_{1}$ in their family, they show that:
\begin{align*}
    S^3_{0}(K_{1}) = S^2\left(\frac{3}{2}, -\frac{5}{2}, -\frac{15}{4}\right)
\end{align*}
Since, by definition, $S^3_{0}(K_{1})$ is $0$-surgery on a knot in $S^3$, we know from Corollary \ref{cor:zero surgery obstruction} that:
\begin{align*}
    -\frac{1}{2}\leq \underline{d}_{-1/2}(S^3_{0}(K_{1}))\text{ and }\Bar{d}_{1/2}(S^3_{0}(K_{1}))\leq \frac{1}{2}
\end{align*}
We now verify these bounds directly by computing $HFI^+(-S^3_{0}(K_{1}), \mathfrak{s}_{0})$ and then we compare this to $HFI^+(-N_{1}, \mathfrak{s}_{0})$. In the interest of brevity, we are only going to give an outline of the calculation and leave the details to the reader.

\subsubsection{Step 1}
We use Kirby calculus and the fact that $S^3_{0}(K_{1})$ is a small Seifert fibered space to find negative-semi definite plumbing representations of $S^3_{0}(K_{1})$ and $-S^3_{0}(K_{1})$:

\begin{figure}[h]
    \centering
    \includegraphics{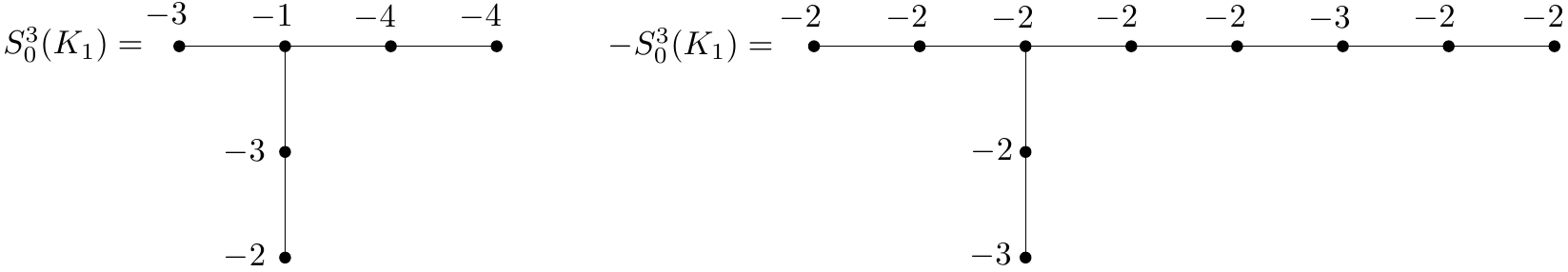}
\end{figure}

\noindent Label the vertices of the above left plumbing graph as:

\begin{figure}[h]
    \centering
    \includegraphics{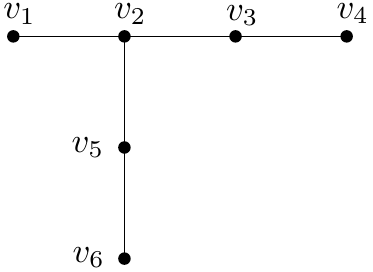}
\end{figure}

\noindent Using the methods of the previous section, one can show that as graded $\F[U]$-modules:

\begin{figure}[h]
    \centering
    \includegraphics{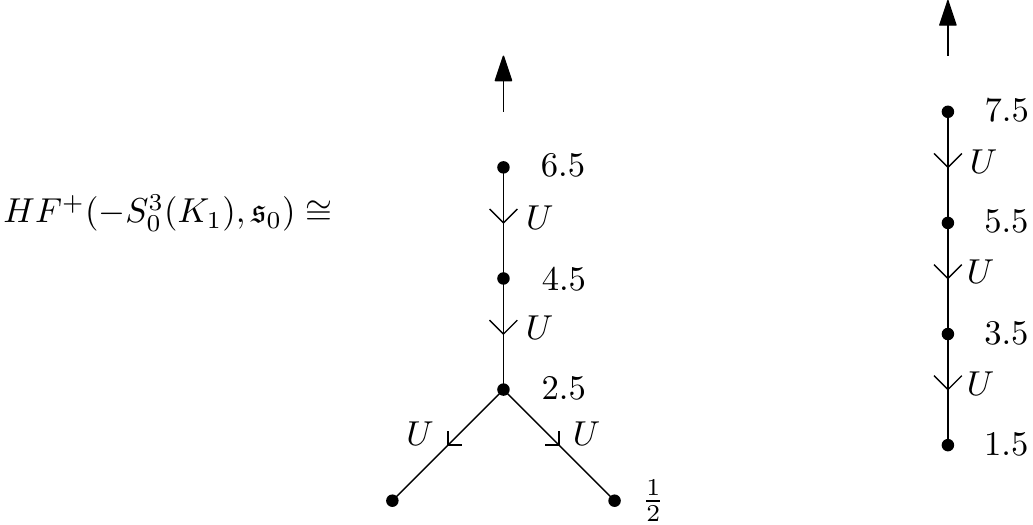}
\end{figure}
\noindent where the two leaves on the left graph correspond to the representative vectors:
\begin{align*}
    z_{1} = (-1,-1,4,-2,1,0)\text{ and }z_{2} = (1,-1,0,4,1,0)
\end{align*}
Note that $HF^+(-S_{0}(K_{1}),\mathfrak{s}_{0})\cong HF^+(-N_{1}, \mathfrak{s}_{0})$.

\subsubsection{Step 2}
To determine $\iota_{*}$, consider the following sequence of moves starting with the vector $-z_{1} = (1,1,-4,2,-1,0)$:
\begin{enumerate}
    \item Add $-2PD[v_{3}]$: $(1,-1,4,0,-1,0)$

    \item Add $-2PD[v_{2}]$: $(-1,1,2,0,-3,0)$
    
    \item Add $-2PD[v_{5}]$: $(-1,-1,2,0,3,-2)$
    
    \item Add $-2PD[v_{6}]$: $(-1,-1,2,0,1,2)$
    
    \item Add $-2PD[v_{2}]$: $(-3, 1, 0,0,-1,2)$
    
    \item Add $-2PD[v_{1}]$: $(3,-1,0,0,-1,2)$
    
    \item Add $-2PD[v_{2}]$: $(1,1,-2,0,-3,2)$
    
    \item Add $-2PD[v_{5}]$: $(1,-1,-2,0,3,0)$
    
    \item Add $-2PD[v_{2}]$: $(-1,1,-4,0,1,0)$
    
    \item Add $-2PD[v_{3}]$: $(-1,-1,4,-2,1,0)=z_{1}$
\end{enumerate}
Therefore, $\iota_{*}$ is the identity. In particular, unlike for $HF^+(-N_{1}, \mathfrak{s}_{0})$, $\iota_{*}$ is not a symmetric involution as defined in Definition \ref{def: symmetric involution}.

\subsubsection{Step 3}
Applying the same methods as in the proof of Theorem \ref{thm: HFI N}, we get:
\begin{thm}
As graded $\F[U,Q]/(Q^2)$-modules: 

\begin{figure}[h]
    \centering
    \includegraphics{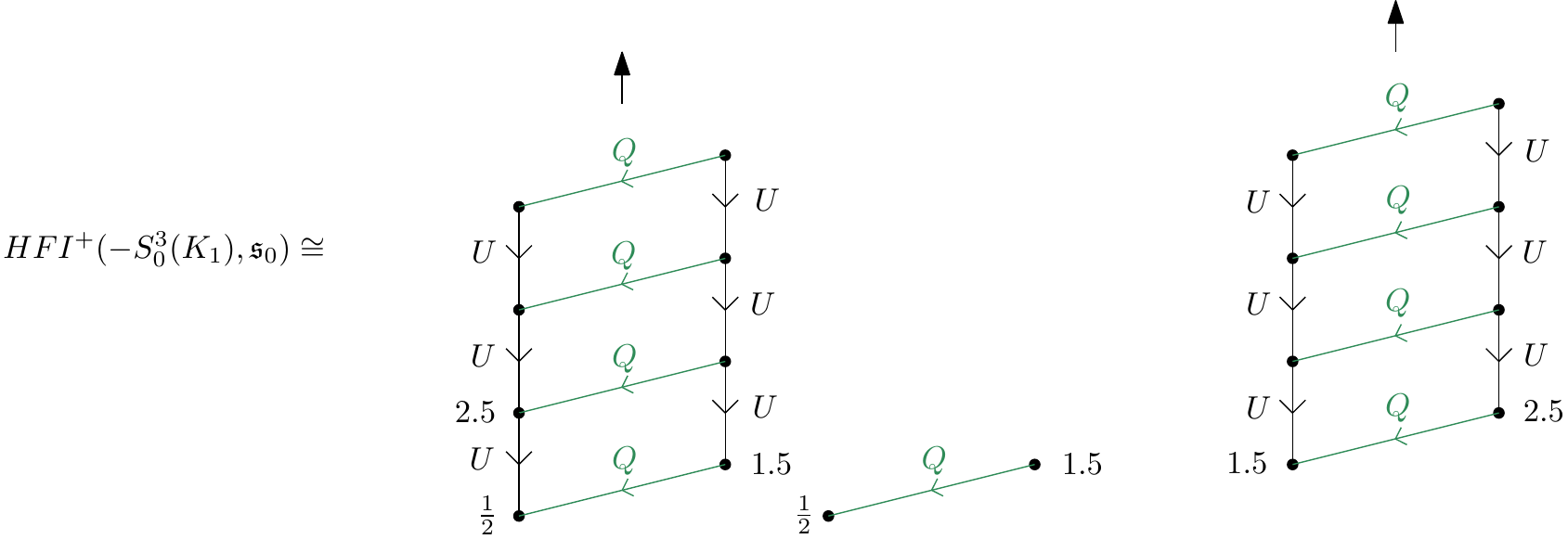}
\end{figure}
\noindent In particular,
\begin{align*}
    \bar{d}_{1/2}(-S_{0}^3(K_{1})) = \frac{1}{2}&\hspace{2em}\bar{d}_{-1/2}(-S_{0}^3(K_{1})) =1.5 \\
    \underline{d}_{1/2}(-S_{0}^3(K_{1})) = \frac{1}{2} &\hspace{2em}  \underline{d}_{-1/2}(-S_{0}^3(K_{1})) = 1.5
\end{align*}
\end{thm}

\noindent In summary, even though 
\begin{align*}
    HF^+(-N_{1}, \mathfrak{s}_{0})\cong HF^+(-S_{0}^3(K_{1}),\mathfrak{s}_{0})
\end{align*}
we see that 
\begin{align*}
    HFI^+(-N_{1}, \mathfrak{s}_{0})\ncong HFI^+(-S_{0}^3(K_{1}),\mathfrak{s}_{0})
\end{align*}

\bibliographystyle{amsalpha}
\bibliography{Involutive}

\end{document}